\documentclass{amsart}

\usepackage{amsmath,amssymb,enumerate,amsrefs,fullpage}

\newcommand{\E}[1]{\mathcal{E}^{#1}}
\renewcommand{\O}[1]{\mathcal{O}^{#1}}

\newcommand{\dbarb}{\overline{\partial}_b}
\newcommand{\boxb}{\square_b}
\newcommand{\boxbbar}{\overline{\square}_b}
\newcommand{\zbar}{\overline{z}}
\newcommand{\Zbar}{\overline{Z}}
\newcommand{\wbar}{\overline{w}}
\newcommand{\Wbar}{\overline{W}}
\newcommand{\hZ}{\hat{Z}}
\newcommand{\hZbar}{\hat{\Zbar}}
\newcommand{\T}{\textrm{Tor}_{\nabla}}
\newcommand{\hor}{H}

\newcommand{\omegabar}{\overline{\omega}}
\newcommand{\re}{\textrm{Re}\,}
\newcommand{\im}{\textrm{Im}\,}
\newcommand{\m}{\theta \wedge d\theta}

\newcommand{\tdbarb}{\tilde{\dbarb}}
\newcommand{\tvdbarb}{\tilde{\vartheta_b}}
\newcommand{\tboxb}{\tilde{\square_b}}

\newcommand{\tS}{\tilde{\Pi}}

\newcommand{\tb}{\tilde{\beta}}
\newcommand{\hM}{\hat{M}}
\newcommand{\hthe}{\hat{\theta}}
\newcommand{\hdbarb}{\hat{\dbarb}}
\newcommand{\hboxb}{\hat{\square_b}}
\newcommand{\hnabla}{\hat{\nabla}}
\newcommand{\hrho}{\hat{\rho}}
\newcommand{\hS}{\hat{\Pi}}
\newcommand{\hK}{\hat{K}}
\newcommand{\hN}{\hat{N}}
\newcommand{\hR}{\hat{R}}
\newcommand{\hm}{\hat{m}}
\newcommand{\tm}{\tilde{m}}
\newcommand{\tN}{\tilde{N}}
\newcommand{\psibar}{\overline{\psi}}
\newcommand{\tf}{\tilde{f}}
\newcommand{\hobar}{\hat{\overline{\omega}}}
\newcommand{\hDelb}{\hat{\Delta}_b}

\newcommand{\hD}{\hat{\mathcal{D}}}

\newcommand{\hX}{\hat{X}}
\newcommand{\hY}{\hat{Y}}
\newcommand{\hT}{\hat{T}}

\newcommand{\cZ}{\mathcal{Z}}

\newtheorem{thm}{Theorem}
\newtheorem{prop}[thm]{Proposition}
\newtheorem{lem}[thm]{Lemma}
\newtheorem{cor}[thm]{Corollary}

\begin{document}

\title[Solution of the tangential Kohn Laplacian]{Solution of the tangential Kohn Laplacian on a class of non-compact CR manifolds}
\author{Chin-Yu Hsiao}
\address{Institute of Mathematics, Academia Sinica, 6F, Astronomy-Mathematics Building, No. 1, Sec. 4, Roosevelt Road, Taipei 10617, Taiwan}
\email{chsiao@math.sinica.edu.tw}
\author{Po-Lam Yung}
\address{Department of Mathematics, The Chinese University of Hong Kong, Shatin, Hong Kong}
\email{plyung@math.cuhk.edu.hk}
%\date{\today}
\begin{abstract} 
We solve $\boxb$ on a class of non-compact 3-dimensional strongly pseudoconvex CR manifolds via a certain conformal equivalence. The idea is to make use of a related $\boxb$ operator on a compact 3-dimensional strongly pseudoconvex CR manifold, which we solve using a pseudodifferential calculus. The way we solve $\boxb$ works whenever $\dbarb$ on the compact CR manifold has closed range in $L^2$; in particular, as in \cite{BG88}, it does not require the CR manifold to be the boundary of a strongly pseudoconvex domain in $\mathbb{C}^2$. Our result provides in turn a key step in the proof of a positive mass theorem in 3-dimensional CR geometry, by Cheng, Malchiodi and Yang \cite{MR3600060}, which they then applied to study the CR Yamabe problem in 3 dimensions.
\end{abstract}
\maketitle

\tableofcontents

\setlength{\parskip}{0.5em}

\section{Introduction}

In a recent paper \cite{MR3600060}, Cheng, Malchiodi and Yang initiated the study of a positive mass theorem in 3-dimensional CR geometry. Among other things, they considered, on a class of closed 3-dimensional pseudohermitian manifolds, the Green's function of the conformal sublaplacian. They developed asymptotics of this Green's function near its singularity in CR normal coordinates, and obtained, under certain conformally invariant geometric assumptions, an important result about the sign of the constant term in this asymptotic expansion. This in turn has applications in the study of the CR Yamabe problem. 

The approach in \cite{MR3600060} involves relating the constant term in the asymptotic expansion of the aforementioned Green's function, to the mass of a certain blow-up of the closed pseudohermitian manifold. The blown-up manifold involved is non-compact, and has only one end; the mass mentioned above is defined through the integral of a certain geometric quantity on a sphere at infinity. One key insight of \cite{MR3600060} is that this mass can be given, via a Bochner formula, in terms of a certain integral over the whole blow-up. In order to show that this mass is non-negative, one then has to solve a certain tangential Kohn Laplacian $\boxb$ on this blow-up. Since the blow-up is non-compact, one cannot do this by using the classical $L^2$ theory. (Even in simple examples, the operator $\boxb$, when extended to a closed linear operator from $L^2$ to $L^2$, does not have closed range.) Thus one has to proceed differently. In \cite{MR3366852}, we solved such $\boxb$ using a weighted $L^2$ theory. In the current paper, we present an alternative, and slightly simpler, self-contained approach to the same problem, using an $L^p$ theory instead. 

Our current approach makes systematic use of pseudodifferential calculus. In the work of Nagel and Stein \cite{NaSt}, they developed a theory of non-isotropic pseudodifferential operators. In this paper, we give a rather detailed account of a slightly simplified version of that theory, as was developed in Stein and Yung \cite{MR3228630}; see Sections~\ref{sect:pdo} to~\ref{sect:hboxbsol}. This simplified version is largely similar to the approach of Beals and Greiner \cite{BG88}, except that we do not require the symbols to admit any asymptotic expansions, which makes it slightly easier to use. Using such pseudodifferential calculus, and a variant of the argument in \cite{BG88}, we describe the relative solution operator and the Szeg\H{o} projection for the tangential Kohn Laplacians, on any 3-dimensional abstract compact smooth strongly pseudoconvex CR manifolds for which $\dbarb$ has closed range in $L^2$. We note in passing that the same result was well-known if the CR manifold is the boundary of a strongly pseudoconvex domain $\Omega$ in $\mathbb{C}^2$; see e.g. the exposition of Nagel-Stein \cite{NaSt}. On the other hand, it was crucial in \cite{NaSt} that one has an asymptotic formula for the Szeg\H{o} kernel, before one constructs the relative solution operator to the tangential Kohn Laplacian. Such asymptotics for the Szeg\H{o} projection has been proved by Fefferman \cite{F}, and Boutet de Monvel and Sj\"ostrand \cite{BouSj76}. The proof of Fefferman \cite{F}, for instance, relies on the existence of a biholomorphic map $F$ between a small neighborhood $U$ of a boundary point $p$ of $\Omega \subset \mathbb{C}^2$, and a small neighborhood of a boundary point of the unit ball $B$ in $\mathbb{C}^2$, so that $F(U \cap \Omega)$ is tangent to the boundary of the unit ball to third order at $F(p) \in \partial B$. See also Nagel, Rosay, Stein and Wainger \cite{NRSW2}, where they constructed the Szeg\H{o} kernel of $\partial \Omega$ by making use of knowledge of the Bergman kernel of $\Omega$. These approaches for studying the Szeg\H{o} projection would not work when we merely assume that $\dbarb$ has closed range, which is the case of interest here. So in \cite{BG88} and also what follows, one has to proceed slightly differently; see Section~\ref{sect:hboxbsol}. Nonetheless, it should be emphasized that our approach works only for strongly pseudoconvex CR manifolds. We refer the reader to the much deeper works of e.g. Christ \cite{Ch88I}, \cite{Ch88II}, Fefferman and Kohn \cite{FeKo88}, Machedon \cite{Ma1}, Nagel, Rosay, Stein and Wainger \cite{NRSW2} for results in the weakly pseudoconvex (and finite type) case; in particular, the articles \cite{Ch88I}, \cite{Ch88II} of Christ contain pointwise estimates for the Szeg\H{o} kernel and the relative solution operator to $\dbarb$, on compact 3-dimensional weakly pseudoconvex CR manifolds of finite type for which $\dbarb$ has closed range in $L^2$.

Another difficulty we had to resolve in this paper is to find a way to reduce the solution of $\boxb$ on a non-compact CR manifold obtained from a blow-up, to the solution of a more manageable tangential Kohn Laplacian on a compact CR manifold. When the blown-up is constructed using the modulus of a CR function, this is relatively easy (see Section~\ref{sect:eg}); in general the problem is quite a bit harder, and its resolution forms the core of the current paper (see Sections~\ref{sect:tboxb} and~\ref{sect:hboxb}).

The paper is organised as follows. In Section~\ref{sect:def}, we recall some basic definitions in CR and pseudohermitian geometry. In Section~\ref{sect:CY}, we recall the CR Yamabe problem, and reduce it to the study of asymptotics of a suitable Green's function. In Section~\ref{sect:pm}, we describe the work of Cheng, Malchiodi and Yang \cite{MR3600060}, where they relate the asymptotics of this Green's function, to the mass of a suitable blow-up. In Section~\ref{sect:result}, we state our main result, and indicate how this can be used to complete the proof of the CR positive mass theorem in \cite{MR3600060}. Section~\ref{sect:pdo} contains the description of a non-isotropic pseudodifferential calculus on strongly pseudoconvex CR manifolds, which is a variant of the theory of Nagel and Stein \cite{NaSt} developed in Stein and Yung \cite{MR3228630}. Sections~\ref{sect:hDelb} and \ref{sect:hboxbsol} give the solution of the sublaplacian and certain tangential Kohn Laplacians using the pseudodifferential calculus developed in Section~\ref{sect:pdo}. Section~\ref{sect:eg} describes how a model case of our main theorem can be established; the key there is a certain conformal equivalence. This motivates the proof of our main theorem, which is given in Sections~\ref{sect:tboxb} and~\ref{sect:hboxb}.

\noindent \textbf{Acknowledgments.} The authors would like to thank Elias M. Stein for his constant inspiration through his joint work \cite{MR3228630} with the second author. The authors would also like to thank Charles Fefferman for a very helpful discussion in regard to the material in Section~\ref{sect:hboxbsol}. Yung was partially supported by the General Research Fund CUHK14313716 from the Hong Kong Research Grant Council, and direct grants for research from the Chinese University of Hong Kong (4053220 and 4441651).

\section{Some basic notations in 3-dimensional CR geometry} \label{sect:def}

Suppose $M$ is a 3-dimensional smooth manifold. It is said to be a CR manifold, if there exists a 1-dimensional sub-bundle $L$ of the complexified tangent bundle $\mathbb{C}TM$, with $L \cap \bar{L} = \{0\}$. We usually write $L = T^{(1,0)}M$, and $\bar{L} = T^{(0,1)} M$. Write also $$\hor = \text{Re} (T^{(1,0)} M \oplus T^{(0,1)} M).$$ Then there exists a unique endomorphism $J \colon \hor \to \hor,$ which extends $\mathbb{C}$-bilinearly to an endomorphism of $T^{(1,0)} M \oplus T^{(0,1)} M$, such that $$JZ = iZ$$ for all $Z \in T^{(1,0)}M$, and $$J\Zbar = -i\Zbar$$ for all $\Zbar \in T^{(0,1)}M$; such $J$ is called a complex structure on $M$. Furthermore, $M$ is said to be strongly pseudoconvex, if in a neighborhood of every point $p \in M$, there exists a local section $Z$ of $T^{(1,0)}M$, such that $i[Z,\Zbar] \notin \hor$ throughout that neighborhood. 

Suppose now $M$ is a connected, oriented 3-dimensional strongly pseudoconvex CR manifold. Then since $\hor$ is oriented by $J$, and $M$ is oriented, the subbundle of the cotangent bundle $TM^*$ that annihilates $\hor$ has a global non-vanishing section. Hence there exists a (global) real 1-form $\theta$ on $M$, such that $\hor = \text{kernel}(\theta)$. Since $M$ is strongly pseudoconvex, $\theta$ is then a contact form on $M$, meaning that $\theta \wedge d\theta \ne 0$ everywhere on $M$. Replacing $\theta$ by $-\theta$ if necessary, one may define a Hermitian inner product on $T^{(1,0)}M$, by
$$
\langle W, Z \rangle_{\theta} := -i d\theta(W, \Zbar)
$$
for any $W, Z \in T^{(1,0)}M$ (The convention here is that if $\alpha, \beta$ are 1-forms, then $\alpha \wedge \beta := \alpha \otimes \beta - \beta \otimes \alpha$. So if $X, Y$ are vector fields, then $(\alpha \wedge \beta)(X,Y) = \alpha(X)\beta(Y) - \alpha(Y)\beta(X)$). Similarly, one has a Hermitian inner product on $T^{(0,1)}M$, defined by
\begin{equation} \label{eq:phermetric}
\langle \Wbar,\Zbar \rangle_{\theta} := i d\theta(\Wbar, Z)
\end{equation}
for any $\Wbar$, $\Zbar$ in $T^{(0,1)}M$. (Note that $\langle \Zbar, \Wbar \rangle_{\theta} = \overline{ \langle Z, W \rangle_{\theta}}$.) Such a pair $(M,\theta)$ is then known as a 3-dimensional pseudohermitian manifold (pseudo because we are working with a CR manifold rather than a complex manifold).

Note that by duality, the Hermitian inner product (\ref{eq:phermetric}) on $T^{(0,1)}M$ induces a Hermitian inner product on its dual bundle $\Lambda^{(0,1)}M$. We will also denote, by abuse of notation, this latter Hermitian inner product by $\langle \cdot, \cdot \rangle_{\theta}$. Sections of $\Lambda^{(0,1)}M$ will be called $(0,1)$ forms on $M$.

One can then turn $M$ into a metric space as follows: A curve $\gamma$ on $M$ is said to be horizontal, if $\gamma$ is tangent to $\hor$ at every point. On $\hor$ one has a real inner product given by
$$
g_{\theta}(X,Y) := d\theta(X,JY),
$$ 
and this allows one to measure the length of any horizontal curve on $M$ (The notation here is so that if $X \in \hor$ satisfies $g_{\theta}(X,X) = 1$, and $Y = JX$, then $Z:=\frac{X-iY}{\sqrt{2}}$ satisfies $\langle Z, Z \rangle_{\theta} = 1$). Since any two points on $M$ can be joined by a horizontal curve (Chow's theorem), one can define the distance between two points on $M$, as the infimum of the lengths of all horizontal curves joining the two points; we call this the non-isotropic distance on $(M,\theta)$ (non-isotropic because the directions tangent to $\hor$ play a special role compared to those transverse to $\hor$).

Let $(M,\theta)$ be a 3-dimensional pseudohermitian manifold. Then there exists a unique global real vector field $T$ on $M$, known as the Reeb vector field, such that $\theta(T) = 1$ and $d\theta(T, \cdot) \equiv 0$. We have a direct sum decomposition
$$
\mathbb{C}TM = T^{(1,0)}M \oplus T^{(0,1)}M \oplus \mathbb{C}T,
$$
and for later purposes, we will write the identity map on $\mathbb{C}TM$ as $$I = \pi_+ + \pi_- + \pi_0$$ according to this decomposition. Furthermore, there exists a unique affine connection $\nabla$ on $M$, which, when extended bilinearly over complex scalars, satisfies
\begin{enumerate}[(i)]
\item $\nabla_X Z$ is a section of $T^{(1,0)}M$ whenever $Z$ is a section of $T^{(1,0)}M$ and $X \in \mathbb{C}TM$;
\item $\nabla$ is compatible with the Hermitian inner product $\langle \cdot, \cdot \rangle_{\theta}$, in the sense that $X \langle \Wbar, \Zbar \rangle_{\theta} = \langle \nabla_X \Wbar, \Zbar \rangle_{\theta} + \langle \Wbar, \nabla_{\overline{X}} \Zbar \rangle_{\theta}$ whenever $W, Z$ are sections of $T^{(1,0)}M$, and $X \in \mathbb{C}TM$;
\item $\nabla_X T = 0$ for all $X \in \mathbb{C}TM$;
\item Let $\T$ be the torsion of $\nabla$, i.e. $\T(X,Y) = \nabla_X Y - \nabla_Y X - [X,Y]$. Then $\T(Z,\Wbar)$ is a multiple of $T$, whenever $Z, W \in T^{(1,0)}M$;
\item $\T(T,\Zbar)$ is a section of $T^{(1,0)}M$ whenever $\Zbar$ is a section of $T^{(0,1)}M$.
\end{enumerate}
Such a connection $\nabla$ is called the Tanaka-Webster connection on $(M,\theta)$. Indeed, suppose $\nabla$ is an affine connection on $M$, extended complex bilinearly such that 
\begin{equation} \label{eq:nablaconj}
\nabla_{\Wbar} \Zbar = \overline{\nabla_W Z}, \quad\nabla_{\Wbar} Z = \overline{\nabla_W \Zbar}, \quad \text{and} \quad \nabla_T \Zbar = \overline{\nabla_T Z}
\end{equation}
for any section $Z$ of $T^{(0,1)}M$ and any $W \in T^{(0,1)}M$, such that $\nabla$ satisfies properties (i) to (v) above. 
Condition (iii) shows that 
\begin{equation} \label{eq:nablaXT}
\nabla_X T = 0 \quad \text{for all $X \in \mathbb{C}TM$},
\end{equation}
and hence
$$
\T(T,\Zbar) = \nabla_T \Zbar - [T,\Zbar].
$$
In view of condition (v), this shows that
\begin{equation} \label{eq:nablaTZbar}
\nabla_T \Zbar = \pi_- [T,\Zbar].
\end{equation}
Similarly, let $W, Z$ be sections of $T^{(1,0)}M$. Then since  
$$
\T(W,\Zbar) = \nabla_W \Zbar - \nabla_{\Zbar} W - [W,\Zbar],
$$
conditions (i) and (iv) show that 
\begin{equation} \label{eq:nablaZWbar}
\nabla_W \Zbar = \pi_- [W,\Zbar].
\end{equation}
Condition (ii) then shows that $\nabla_{\Wbar} \Zbar$ is the unique element in $T^{(0,1)}M$ such that the identity
\begin{equation} \label{eq:nablaZbarWbar}
\langle \overline{U}, \nabla_{\Wbar} \Zbar \rangle_{\theta} = W \langle \overline{U}, \Zbar \rangle_{\theta} - \langle\pi_- [W,\overline{U}], \Zbar \rangle_{\theta}
\end{equation}
holds for all sections $\overline{U}$ of $T^{(0,1)}M$. 
The equations (\ref{eq:nablaconj}), (\ref{eq:nablaXT}), (\ref{eq:nablaTZbar}), (\ref{eq:nablaZWbar}) and (\ref{eq:nablaZbarWbar}), determine $\nabla$ uniquely. In addition, one can turn this around, and check that these equations define an affine connection on $M$ such that conditions (i) through (v) are satisfied. For the record, the torsion of $\nabla$ is then given by
$$
\T(Z,\Wbar) = -\pi_0 [Z,\Wbar] = -\theta([Z,\Wbar]) T,
$$
$$
\T(T,\Zbar) = -\pi_+ [T,\Zbar] = - \frac{1}{2} J \circ (\mathcal{L}_T J) (\Zbar);
$$
indeed, to see the last equality, note that from the formula following (\ref{eq:nablaXT}), we have $\pi_0 [T,\Zbar] = 0$, so $[T,\Zbar] = \pi_+ [T,\Zbar] + \pi_- [T,\Zbar],$ from which we obtain $J[T,\Zbar] = i\pi_+ [T,\Zbar] - i \pi_- [T,\Zbar].$ Hence
$$
(\mathcal{L}_T J)(\Zbar) = [T, J\Zbar] - J [T, \Zbar] = -i [T, \Zbar] - i\pi_+ [T,\Zbar] + i \pi_- [T,\Zbar] = -2i \pi_+ [T,\Zbar].
$$
Applying $J$ on both sides, we get $\pi_+ [T,\Zbar] = \frac{1}{2} J \circ (\mathcal{L}_T J) (\Zbar)$, as desired.

Now that we have the Tanaka-Webster connection $\nabla$ of $(M,\theta)$, we can define the corresponding scalar curvature $R_{\theta}$. Indeed, define the curvature operator by
$$
\Omega(X,Y)= \nabla_X \nabla_Y - \nabla_Y \nabla_X - \nabla_{[X,Y]}.
$$
The Tanaka-Webster scalar curvature $R_{\theta}$ is then determined by
$$
\Omega(Z,\Zbar)Z = R_{\theta} \langle \Zbar, \Zbar \rangle_{\theta} Z,
$$
where $\Zbar$ is any non-zero local section of $T^{(0,1)}M$. It is known that the torsion $\T(T, \cdot)$ and the Tanaka-Webster scalar curvature $R_{\theta}$ form a complete set of local invariants for the 3-dimensional pseudohermitian manifold $(M,\theta)$.

The Tanaka-Webster connection on $(M,\theta)$ is also sometimes described in terms of differential forms as follows. Let $\nabla$ be the Tanaka-Webster connection on $(M,\theta)$, extended bilinearly over complex scalars. Let $Z_1$ be a local section of $T^{1,0}M$, so that $Z_{\bar{1}} := \overline{Z_1}$ is a local section of $T^{0,1}M$. Let $T$ denote the Reeb vector field on $(M,\theta)$. Then there exists $\omega_1^1 \in (\mathbb{C}TM)^*$ and $\tau^1 \in (\mathbb{C}TM)^*$, such that
$$
\nabla_X Z_1 = \omega_1^1(X) Z_1 \quad \text{and} \quad \T(T,X) = \tau^1(X) Z_1 + \tau^{\bar{1}}(X) Z_{\bar{1}}
$$
for all $X \in \mathbb{C}TM$, where $\tau^{\bar{1}} := \overline{\tau^1}$. Furthermore, let $\theta^1, \theta^{\bar{1}}, \theta \in (\mathbb{C}TM)^*$ be the dual frame to $Z_1, Z_{\bar{1}}, T$. Then $\omega_1^1$ and $\tau^1$ are the unique elements in $(\mathbb{C}TM)^*$, such that
$$
\begin{cases}
d\theta^1  = \theta^1 \wedge \omega_1^1 + \theta \wedge \tau^1 \\
d h_{1\bar{1}} = \omega_1^1 h_{1 \bar{1}} + h_{1 \bar{1}} \omega_{\bar{1}}^{\bar{1}} \\
\tau^1  = 0 \quad (\text{mod } \theta^{\bar{1}})
\end{cases}
$$
where $h_{1 \bar{1}} := \langle Z_{\bar{1}}, Z_{\bar{1}} \rangle_{\theta}$ (i.e. $d\theta = i h_{1 \bar{1}} \theta^1 \wedge \theta^{\bar{1}}$) and $\omega_{\bar{1}}^{\bar{1}} := \overline{\omega_1^1}$. Sometimes one also writes $\tau^1 = A_{\bar{1}}^1 \theta^{\bar{1}}$, so that the above simplifies to
$$
\begin{cases}
d\theta^1  = \theta^1 \wedge \omega_1^1 + A_{\bar{1}}^1 \theta \wedge \theta^{\bar{1}} \\
d h_{1\bar{1}} = \omega_1^1 h_{1 \bar{1}} + h_{1 \bar{1}} \omega_{\bar{1}}^{\bar{1}};
\end{cases}
$$
it also follows then that
$$
\T(T,Z_{\bar{1}}) = A_{\bar{1}}^1 Z_1 \quad \text{and} \quad \T(T,Z_1) = A_1^{\bar{1}} Z_{\bar{1}},
$$
where $A_1^{\bar{1}} = \overline{A_{\bar{1}}^1}$. The Tanaka-Webster scalar curvature $R_{\theta}$ of $(M,\theta)$ is then determined by
$$
d\omega_1^1 = R_{\theta} h_{1 \bar{1}} \theta^1 \wedge \theta^{\bar{1}} \quad (\text{mod } \theta).
$$ 

As an example, consider for instance the unit sphere $\mathbb{S}^3$ in $\mathbb{C}^2$, defined by $\{\rho = 0\}$ where $\rho(\zeta) := |\zeta|^2 - 1$ if $\zeta = (\zeta^1, \zeta^2) \in \mathbb{C}^2$. It is equipped with the CR structure induced from the complex structure of $\mathbb{C}^2$. Let $$\theta = 2 \textrm{Im}\, \partial \rho = \frac{1}{i} ( \overline{\zeta^1} d\zeta^1 + \overline{\zeta^2} d\zeta^2 -\zeta^1 d\overline{\zeta^1} - \zeta^2 d\overline{\zeta^2}).$$ Then $(\mathbb{S}^3, \theta)$ is a 3-dimensional pseudohermitian manifold. Let $$Z = \frac{1}{\sqrt{2}} ( \overline{\zeta^2} \frac{\partial}{\partial {\zeta^1}} - \overline{\zeta^1} \frac{\partial}{\partial {\zeta^2}}) \quad \text{ and } \quad \Zbar = \frac{1}{\sqrt{2}} ( \zeta^2 \frac{\partial}{\partial {\overline{\zeta^1}}} - \zeta^1 \frac{\partial}{\partial {\overline{\zeta^2}}}).$$ Also let $$T = \textrm{Im} \, (\overline{\zeta^1} \frac{\partial}{\partial \overline{\zeta^1}} + \overline{\zeta^2} \frac{\partial}{\partial \overline{\zeta^2}}) = \frac{1}{2i} (\overline{\zeta^1} \frac{\partial}{\partial \overline{\zeta^1}} + \overline{\zeta^2} \frac{\partial}{\partial \overline{\zeta^2}} - \zeta^1 \frac{\partial}{\partial {\zeta^1}} - {\zeta^2} \frac{\partial}{\partial {\zeta^2}}).$$ Then $Z$ is a global section of $T^{1,0}(\mathbb{S}^3)$, $\Zbar$ is a global section of $T^{0,1}(\mathbb{S}^3)$, $$\theta(T) = 1,$$ and $$[T,Z] = -iZ, \quad [T,\Zbar] = i\Zbar.$$ In particular, $\theta([T,Z]) = \theta([T,\Zbar]) = 0$, so $T$ is the Reeb vector field on $(\mathbb{S}^3,\theta)$. Also, we have $$[Z,\Zbar] = -i T,$$ so in particular $\langle \Zbar, \Zbar \rangle_{\theta} = i \theta([Z,\Zbar]) =  1$. The above commutation relations also show that $$\nabla_Z \Zbar = 0, \quad \nabla_{\Zbar} \Zbar = 0 \quad \text{ and } \quad \nabla_T \Zbar = i\Zbar \quad \text{ with } \quad \T(T, \Zbar) = 0.$$ It follows that
 $$\nabla_Z Z = 0, \quad \nabla_{\Zbar} Z = 0 \quad \text{ and } \quad \nabla_T Z = - i Z \quad \text{ with } \quad \T(T,Z) = 0;$$ 
from $\Omega(Z,\Zbar)Z = -\nabla_{[Z,\Zbar]} Z = i \nabla_T Z = Z$, we see that $R_{\theta} \equiv 1$ on $(\mathbb{S}^3,\theta)$. Letting $Z_1 = Z$, one can check that $$h_{1 \bar{1}} = 1, \quad \omega_1^1 = -i\theta \quad \text{ and } \quad \tau^1 = 0,$$ so $d\omega_1^1 = -i d\theta = \theta^1 \wedge \theta^{\bar{1}}$, which also shows that $R_{\theta} \equiv 1$.

Another example is the Heisenberg group $\mathbb{H}^1$, which is a Lie group diffeomorphic to $\mathbb{C} \times \mathbb{R}$, with group law
$$
(z,t) \cdot (w,s) = (z+w, t+s + 2 \textrm{Im}(z \wbar))
$$
where $(z,t), (w,s) \in \mathbb{C} \times \mathbb{R}$.
The CR structure on $\mathbb{H}^1$ is given so that $T^{0,1}(\mathbb{H}^1)$ is spanned by the left-invariant vector field $\Zbar := \frac{1}{\sqrt{2}} (\frac{\partial}{\partial \overline{z}} - i z \frac{\partial}{\partial t})$. Let $$\theta := dt - i (\overline{z} dz - z d\overline{z}).$$ Then $(\mathbb{H}^1,\theta)$ is a pseudohermitian manifold, with $\langle \Zbar, \Zbar \rangle_{\theta} = 1$. The Reeb vector field is given by $T := \frac{\partial}{\partial t}$, and one can check that
$$
[Z,\Zbar] = -iT, \quad [T,Z] = [T,\Zbar] = 0, 
$$
so
$$
\nabla_Z \Zbar = \nabla_{\Zbar} \Zbar = \nabla_T \Zbar = 0, \quad \text{ with } \quad \T(T,\Zbar) = 0,
$$
from which it follows that $R_{\theta} \equiv 0$ on $(\mathbb{H}^1,\theta)$.

We will now define various differential operators that will play an important role in this paper. Let $(M,\theta)$ be a 3-dimensional pseudohermitian manifold. Then $g_{\theta}$ defines a real inner product on $\hor$, and this induces a real inner product on the dual bundle $\hor^*$ of $\hor$, which we also denote by $g_{\theta}$ by abuse of notation. For a smooth function $u$ on $M$, let $d_b u$ be the element in $\hor^*$, such that
$$
d_b u(X) = Xu \quad \text{for all $X \in \hor$}.
$$
If $X_1, X_2$ is a local frame for $\hor$, and $e^1$, $e^2$ is the dual frame, then this says 
$$
d_b u = (X_1 u) e^1 + (X_2 u) e^2.
$$
One then defines the subelliptic gradient of $u$, denoted by $\nabla_b u$, as the unique vector in $\hor$ such that
$$
g_{\theta}( \nabla_b u, X ) = d_b u(X)
$$
for all $X \in \hor$. The length of $\nabla_b u$ is then
$$
|\nabla_b u| := g_{\theta}(\nabla_b u, \nabla_b u)^{1/2}.
$$ 
Now $\theta \wedge d\theta$ is a volume form on $M$. This allows us to define an $L^2$ inner product on functions:
$$
(u,v)_{\theta} = \int_M u \overline{v} \, \m,
$$
as well as an inner product on $\hor^*$:
$$
(\alpha,\beta)_{\theta} = \int_M g_{\theta}(\alpha,\beta) \m.
$$
The formal adjoint $d_b^*$ of $d_b$ is then defined by
$$
(d_b u, \alpha)_{\theta} = (u, d_b^* \alpha)_{\theta}
$$
for all smooth functions $u$ and smooth sections $\alpha$ of $\hor^*$. The sublaplacian of $(M,\theta)$ is the second order operator defined as
$$
\Delta_b u := d_b^* d_b u
$$
for all smooth functions $u$ on $M$. If $X, Y$ is a local orthonormal frame of $\hor$, and $X^*$, $Y^*$ are their adjoints with respect to the above $L^2$ inner product on functions, then locally
$$
\Delta_b u = ( X^* X + Y^* Y ) u.
$$
We have
$$
\int_M |\nabla_b u|^2 \m = (\Delta_b u, u)_{\theta}
$$
for all compactly supported smooth functions $u$ on $M$.

Now let $\langle \cdot, \cdot \rangle_{\theta}$ be the Hermitian inner product on $T^{(0,1)}M$. We will also denote by $\langle \cdot, \cdot \rangle_{\theta}$ the induced Hermitian inner product on the dual bundle $\Lambda^{(0,1)}M$. For a smooth function $u$ on $M$, let $\dbarb u$ be the element of $\Lambda^{(0,1)}M$, such that 
$$
\dbarb u (\Zbar) = \Zbar u \quad \text{for all $\Zbar \in T^{(0,1)}M$}.
$$
If $\Zbar$ is a local frame for $T^{(0,1)}M$, and $\omegabar$ is the dual frame for $\Lambda^{(0,1)}M$, then this says 
$$
\dbarb u = (\Zbar u) \omegabar.
$$
Now we had an $L^2$ inner product $(u,v)_{\theta}$ on functions. We define, in addition, an inner product on $\Lambda^{(0,1)}M$:
$$
(\alpha,\beta)_{\theta} = \int_M \langle \alpha,\beta \rangle_{\theta} \m.
$$
The formal adjoint $\vartheta_b$ of $\dbarb$ is then defined by
$$
(\dbarb u, \alpha)_{\theta} = (u, \vartheta_b \alpha)_{\theta}
$$
for all smooth functions $u$ and smooth sections $\alpha$ of $\Lambda^{(0,1)}M$. The tangential Kohn Laplacian of $(M,\theta)$ is defined as
$$
\boxb u := \vartheta_b \dbarb u
$$
for all smooth functions $u$ on $M$. (Later we will also need two different non-standard $\boxb$, which involves taking the adjoint of $\dbarb$ with respect to different $L^2$ inner products.) If $\Zbar$ is a local frame of $\Lambda^{(0,1)}M$ with $\langle \Zbar, \Zbar \rangle_{\theta} = 1$, and $\Zbar^*$ is its formal adjoint with respect to the $L^2$ inner product on functions, then locally
$$
\boxb u = \Zbar^* \Zbar u.
$$
Since such $Z$ can be written $Z = \frac{X-iY}{\sqrt{2}}$ where $X$ and $Y$ are in $\hor$ and satisfies $g_{\theta}(X,X) = g_{\theta}(Y,Y) = 1$, we have 
$$
\Delta_b = 2 \re \boxb.
$$
We say that a smooth function $\psi$ on $M$ is a CR function, if $\dbarb \psi = 0$ on $M$.

We will need two conformally invariant operators on $(M,\theta)$, namely the conformal sublaplacian $L_{\theta}$, and the CR Paneitz operator $P_{\theta}$. The former is defined by
$$
L_{\theta} u = 4\Delta_b u + R_{\theta} u.
$$
It describes how the Tanaka-Webster curvature changes under a conformal change of the contact form $\theta$ (see Section~\ref{sect:CY}). Finally, the CR Paneitz operator (c.f. \cite{MR975118}) is the fourth order operator on functions given by
$$
P_{\theta} u = 4 \left( \boxbbar \boxb u + i \Zbar \left( \T(T,Z) u \right) \right).
$$
where $\boxbbar$ is the conjugate of $\boxb$. Here $\Zbar$ is a local frame of $\Lambda^{(0,1)}M$ with $\langle \Zbar, \Zbar \rangle_{\theta} = 1$, and $T$ is the Reeb vector field of $(M,\theta)$ as before. %It describes how the CR $Q$-curvature changes under a conformal change of the contact form (see e.g. \cite{}).

For example, on $(\mathbb{H}^1,\theta)$ where $\theta = dt - i(\zbar dz - z d\zbar)$, if we write $z = x+iy$, and write
\[
X = \frac{1}{2} \left( \frac{\partial}{\partial x} + 2y \frac{\partial}{\partial t} \right), \qquad Y = \frac{1}{2} \left( \frac{\partial}{\partial y} - 2x \frac{\partial}{\partial t} \right), \qquad T = \frac{\partial}{\partial t},
\]
then we have 
\[
\theta = dt + 2x dy - 2y dx,
\]
\[
|\nabla_b f|^2 = |Xf|^2 + |Yf|^2,
\]
\begin{equation} \label{eq:sublap_H1}
L_{\theta} = \Delta_b = -(X^2 + Y^2),
\end{equation}
and writing $$Z = \frac{X-iY}{\sqrt{2}},$$ we have
\[
\boxb = -Z \Zbar = -\frac{1}{2}(X^2 + Y^2 + iT), \qquad
P_{\theta} = 4 \boxbbar \boxb = (X^2 + Y^2)^2 + T^2.
\]

\section{The CR Yamabe problem} \label{sect:CY}

Let $(\hM,\hthe)$ be a closed 3-dimensional pseudohermitian manifold. (Here `closed' means `compact without boundary'.) If $u$ is a positive smooth function on $\hM$, the contact form $\hthe^{\sharp} = u^2 \hthe$ defines another pseudohermitian structure on $\hM$. It is readily verified that the corresponding Hermitian inner product on $T^{(0,1)} \hM$ satisfies
$$
\langle \Zbar, \Wbar \rangle_{\hthe^{\sharp}} = u^2 \langle \Zbar, \Wbar \rangle_{\hthe}
$$ 
for all sections $\Zbar, \Wbar$ of $T^{(0,1)} \hM$, so we think of $\hthe^{\sharp}$ as conformally equivalent to $\hthe$.
The Tanaka-Webster scalar curvature $R_{\hthe^{\sharp}}$ with respect to the new contact form $\hthe^{\sharp}$ is given by the conformal sublaplacian:
\begin{equation} \label{eq:CRYamabe}
L_{\hthe} u = R_{\hthe^{\sharp}} u^3.
\end{equation}
The CR Yamabe problem is then to the problem of prescribing constant Tanaka-Webster scalar curvature via a conformal change of the contact form. In other words, one seeks a positive smooth function $u$ on $\hM$, such that the Tanaka-Webster scalar curvature $R_{\hthe^{\sharp}}$ with respect to the new contact form $\hthe^{\sharp}$ is constant. Hence one seeks solutions to the CR Yamabe equation (\ref{eq:CRYamabe}) with $R_{\hthe^{\sharp}} = $ constant.

The CR Yamabe problem is a variational problem. Define the functional
$$
E_{\hthe}(u) = \frac{ \int_{\hM} (4|\hnabla_b u|^2 + R_{\hthe} u^2) \hthe \wedge d\hthe }{ \left( \int_{\hM} u^4 \hthe \wedge d\hthe \right)^{1/2} } %= \frac{ \int_{\hM} u \, \overline{L_{\hthe} u} \, \hthe \wedge d\hthe }{ \left( \int_{\hM} u^4 \hthe \wedge d\hthe \right)^{1/2} }
$$
for any positive smooth function $u$ on $\hM$. Here $|\hnabla_b u| = g_{\hthe}(\hnabla_b u, \hnabla_b u)^{1/2}$ is the length of the subelliptic gradient of $u$ with respect to $\hthe$. The numerator above can also be written as $\int_{\hM} u \, L_{\hthe} u \, \hthe \wedge d\hthe$. From the transformation rule (\ref{eq:CRYamabe}) of the Tanaka-Webster scalar curvature under a conformal change of contact form, it follows that if $\hthe^{\sharp} := u^2 \hthe$, then $E_{\hthe}(u)$ is also given by
\begin{equation} \label{eq:CRYcostgeom}
E_{\hthe}(u) = \frac{ \int_{\hM} R_{\hthe^{\sharp}} \hthe^{\sharp} \wedge d\hthe^{\sharp} }{ \left[ V(\hM,\hthe^{\sharp}) \right]^{1/2} },
\end{equation}
where $V(\hM,\hthe^{\sharp}) := \int_{\hM} \hthe^{\sharp} \wedge d\hthe^{\sharp}$ is the volume of $\hM$ under the volume form $\hthe^{\sharp} \wedge d\hthe^{\sharp}$. One verifies that any critical point of $E_{\hthe}$ solves the CR Yamabe equation (\ref{eq:CRYamabe}) with $R_{\hthe^{\sharp}} = $ constant. 

Define now the CR Yamabe constant by
$$
Y(\hM,\hthe) = \inf \{ E_{\hthe}(u) \colon u \in C^{\infty}(\hM), u > 0 \textrm{ on $\hM$}\}.
$$
By the interpretation (\ref{eq:CRYcostgeom}) of $E_{\hthe}(u)$, one sees that $Y(\hM,\hthe)$ is a conformal invariant, and depends only on the conformal class $[\hthe]$ of $\hthe$. Jerison and Lee \cite{MR880182} proved that 
\begin{equation} \label{eq:Ymax}
Y(\hM,\hthe) \leq Y(\mathbb{S}^3,\theta_{\text{std}})
\end{equation}
where $\mathbb{S}^3$ is the unit sphere $\{\zeta \in \mathbb{C}^2 \colon |\zeta| = 1\}$ in $\mathbb{C}^2$ and $\theta_{\text{std}} = 2 \im \partial (|\zeta|^2-1)$ is the standard contact form on $\mathbb{S}^3$. In addition, if the inequality is strict, then the infimum defining $Y(\hM,\hthe)$ is attained by some positive smooth function $u$ on $\hM$. This $u$ solves (\ref{eq:CRYamabe}), in which case $u^2 \hthe$ has constant Tanaka-Webster scalar curvature. Since clearly $Y(\mathbb{S}^3,\theta_{\text{std}}) > 0$, to solve the CR Yamabe problem, it suffices to consider the case where $Y(\hM,\hthe) > 0$. 

Let now $(\hM,\hthe)$ be a closed 3-dimensional pseudohermitian manifold for which $Y(\hM,\hthe) > 0$. The conformal sublaplacian $L_{\hthe}$ is then a positive operator. So fix now a point $p \in \hM$. Then one has a Green's function $G$ of $L_{\hthe}$, such that
$$
L_{\hthe} G = 16 \delta_p.
$$ 
Here $\delta_p$ is the delta function at the point $p$. We will expand $G$ around $p$ in the analog of the conformal normal coordinates in CR geometry (which is sometimes called `CR normal coordinates'). Indeed, by first conformally changing the contact form on $\hM$ if necessary, one can show the existence of a coordinate system $(z,t)$ of an open neighborhood of the point $p$ in $\hM$, such that 
\begin{enumerate}[(i)]
\item the point $p$ corresponds to $(z,t) = (0,0)$;
\item one can find a local section $\hZ$ of $T^{1,0} \hM$ near $p$, with $\langle \hZ, \hZ \rangle_{\hthe} = 1$, such that $\hZ$ admits the following expansion near $p$:
\begin{equation} \label{s1-e5b}
\hZ=\frac{1}{\sqrt{2}} \left( \frac{\partial}{\partial z}+i \zbar \frac{\partial}{\partial t} \right) + \O{4} \frac{\partial}{\partial z}+\O{4} \frac{\partial}{\partial \zbar}+\O{5} \frac{\partial}{\partial t};
\end{equation}
\item the Reeb vector field $\hat T$ with respect to $\hthe$ admits an expansion
\begin{equation} \label{s1-e5c}
\hat T=\frac{\partial}{\partial t}+\O{3} \frac{\partial}{\partial z}+\O{3} \frac{\partial}{\partial \zbar}+\O{4} \frac{\partial}{\partial t};
\end{equation}
\end{enumerate} 
see Proposition 6.5 of \cite{MR3600060} for a proof of this fact, using the CR normal coordinates introduced by Jerison and Lee \cite{MR982177}. Here $\O{k}$ denotes the set of smooth functions $f$ on $\hM$ that satisfies $|f(q)| \lesssim \hrho(q)^k$ for $q$ near $p$, where \[\hrho(q)=(|z|^4+t^2)^{\frac{1}{4}}\] if $q$ has coordinates $(z,t)$. Note that for $q$ near $p$, $\hrho(q)$ is comparable to the non-isotropic distance  between $q$ and $p$ on $(\hM,\hthe)$. For later convenience, from now on we will fix a positive smooth extension of $\hrho$ to the whole $\hM \setminus \{p\}$. We will also assume that the contact form $\hthe$ on $\hM$ is already the one so that properties (i) to (iii) above holds; this is no loss in generality in solving the CR Yamabe problem.

Now in CR normal coordinates, the Green's function $G$ of $L_{\hthe}$ admits the following expansion for $q \in \hM$ near $p$:
\begin{equation} \label{eq:Gexpansion}
G(q) = \frac{1}{2\pi \hrho(q)^{2}} + A + \text{error}, \quad \text{error} \in \E{1}.
\end{equation}
Here $A$ is a constant, and for each $k \in \mathbb{R}$, we define $\E{k}$ to be the space of all functions $f$ on $\hM$, that is smooth away from $p$, and satisfies $|\hnabla_b^{\ell} f| \lesssim_{\ell} \hrho^{k-\ell}$ for all $\ell \in \mathbb{N} \cup \{0\}$. The sign of the constant term $A$ is of great importance for the solution of the CR Yamabe problem. In fact, if $A > 0$, then one can construct a suitable test function $u$ on $M$, to show that
$$
E_{\hthe}(u) < Y(\mathbb{S}^3,\theta_{\text{std}}).
$$
($u$ is obtained by gluing $G$ to a standard bubble on the Heisenberg group; see Section 5 of Cheng, Malchiodi and Yang \cite{MR3600060}.) This shows strict inequality holds in (\ref{eq:Ymax}), and hence the infimum defining $Y(\hM,\hthe)$ is attained by some positive smooth function $u$ on $\hM$. Hence the key now is to understand the sign of the constant term $A$ in the expansion (\ref{eq:Gexpansion}) of $G$. This was achieved in the work of Cheng, Malchiodi and Yang  \cite{MR3600060}: 

\begin{thm}[Cheng-Malchiodi-Yang \cite{MR3600060}] \label{thm:A}
Let $(\hM,\hthe)$ be a closed 3-dimensional pseudohermitian manifold with $Y(\hM,\hthe) > 0$. Assume, in addition, that the CR Paneitz operator $P_{\hthe}$ is non-negative on $\hM$, in the sense that
$$
(P_{\hthe}u, u)_{\hthe} \geq 0
$$
for any $u \in C^{\infty}(\hM)$. (We write $P_{\hthe} \geq 0$ in this case.) Let $L_{\hthe}$ be the conformal sublaplacian of $(\hM,\hthe)$. For any point $p \in \hM$, let $G$ be the Green's function of $L_{\hthe}$ with pole $p$, and $A$ be the constant term in the expansion of $G$ in CR normal coordinates centered at $p$. Then 
$$
A \geq 0.
$$
Furthermore, $A = 0$ if and only if $\hM$ is CR equivalent to the standard sphere $\mathbb{S}^3$, and $\hthe$ is conformally equivalent to the standard contact form $\theta_{\text{std}}$ on $\mathbb{S}^3$.
\end{thm}

Hence the following corollary follows:

\begin{cor}[Cheng-Malchiodi-Yang \cite{MR3600060}]
Let $(\hM,\hthe)$ be a closed 3-dimensional pseudohermitian manifold with $Y(\hM,\hthe) > 0$ and $P_{\hthe} \geq 0$. Then the infimum defining the CR Yamabe constant $Y(\hM,\hthe)$ is attained, and hence the CR Yamabe equation (\ref{eq:CRYamabe}) has a positive smooth solution.
\end{cor}

Indeed, the case when $A > 0$ follows from the discussion immediately preceding Theorem~\ref{thm:A}, while the case $A = 0$ follows from the last part of Theorem~\ref{thm:A} and a result of Jerison and Lee \cite{MR924699} (see also Frank and Lieb \cite{MR2925386} for a simpler argument). It follows that in either case, the Tanaka-Webster scalar curvature of $(\hM,\hthe)$ can be made a positive constant by conformally changing the contact form $\hthe$. See also Gamara \cite{MR1831872}, and Gamara and Yacoub \cite{MR1867895}, for an alternative approach to the CR Yamabe problem. 

The proof of Theorem~\ref{thm:A} in \cite{MR3600060} consists of two steps. First, they relate the constant $A$ to the mass of a certain blow-up of the closed manifold $(\hM,\hthe)$. Next, they determine the sign of the mass using a suitable Bochner formula. The assumption that $P_{\hthe} \geq 0$ will be used twice in this second step. We will outline all these in the next section.

\section{The CR positive mass theorem} \label{sect:pm}

Let $(\hM,\hthe)$ be a closed 3-dimensional pseudohermitian manifold with $Y(\hM,\hthe) > 0$. Let $L_{\hthe}$ be the conformal sublaplacian on $(\hM, \hthe)$, and $G$ be the Green's function of $L_{\hthe}$ with pole at some point $p \in \hM$. In the last section, we saw that the key to understanding the CR Yamabe problem on $(\hM,\hthe)$ is to get a handle on the constant term $A$ in the expansion (\ref{eq:Gexpansion}) of $G$ under CR normal coordinates near $p$; see Theorem~\ref{thm:A} for a precise statement of what is needed. In this section, we describe how Theorem~\ref{thm:A} is proved in the work of Cheng, Malchiodi and Yang \cite{MR3600060}.

First, assume that we have already conformally changed the contact form $\hthe$ on $\hM$, so that the expansion of the Green's function $G$ of $L_{\hthe}$ admits the expansion in (\ref{eq:Gexpansion}) in CR normal coordinates. Let 
$$
M := \hM \setminus \{p\}, \quad \theta := G^2 \hthe.
$$
We impose on $M$ the CR structure it inherits from $\hM$ by restriction. Then $(M,\theta)$ is a non-compact 3-dimensional pseudohermitian manifold. 

Now recall the transformation rule (\ref{eq:CRYamabe}) of the Tanaka-Webster scalar curvature under a conformal change of contact form. Since $L_{\hthe} G = 0$ on $M = \hM \setminus \{p\}$, the Tanaka-Webster scalar curvature $R_{\theta}$ of $(M,\theta)$ is identically zero on $M$. 

We are going to think of $M$ as a blow up of $\hM$. In particular, let $U \subset M$ be a small deleted neighborhood of $p$ in $\hM$. We will introduce coordinates $(z_*, t_*)$ on $U$, so that $|z_*|+|t_*| \to +\infty$ as the point corresponding to $(z_*,t_*)$ approaches $p$ in $\hM$. The coordinates $(z_*,t_*)$ are just obtained from the CR normal coordinates $(z,t)$ by an inversion: we define
$$
z_* := \frac{z}{t+i|z|^2}, \quad t_* := -\frac{t}{t^2+|z|^4}.
$$
With this coordinates $(z_*,t_*)$, we think of $U$ as a neighborhood of `infinity' of $M$: indeed if $q \in U$ has coordinates $(z_*,t_*)$ in this coordinate chart, we define $\rho(q) := (|z_*|^4 + t_*^2)^{1/4}$. Then $\rho(q) = 1/\hrho(q)$, so $\rho(q) \to +\infty$ as $q \to p$. 

The non-compact pseudohermitian manifold $M$ is sometimes said to have one `end'; by what we observed earlier, this end is actually `flat' with respect to the contact form $\theta$ on $M$, meaning that $R_{\theta}$ is zero there. One can also show that $(M,\theta)$ is asymptotically flat, in the sense of Definition 2.1 of Cheng, Malchiodi and Yang \cite{MR3600060}. This allows one to define the mass of $(M,\theta)$ as in \cite{MR3600060}: let $Z$ be a local section of $T^{(1,0)}M$ on the open set $U$ with
$$
Z = - G  \frac{\rho^6}{(t_*+i|z_*|^2)^3} \hZ
$$
In particular, then $\langle Z, Z \rangle_{\theta} = 1$. Let $\nabla$ be the Tanaka-Webster connection associated to $(M,\theta)$, and let $\omega_1^1$ be the connection 1-form, determined by
$$
\nabla_X Z = \omega_1^1(X) Z
$$
for all $X \in \mathbb{C}TM$. Then the mass of $(M,\theta)$ is defined as
$$
m(M,\theta) :=\lim_{\Lambda \to +\infty} i \int_{\{\rho=\Lambda \}} \omega^1_1 \wedge \theta.
$$
Cheng, Malchiodi and Yang established the following proposition:

\begin{prop}[Cheng-Malchiodi-Yang \cite{MR3600060}]
The constant $A$ in the asymptotic expansion of the Green's function $G$ of $L_{\hthe}$ satisfies
$$
A = \frac{1}{48\pi^2} m(M,\theta),
$$
where $m(M,\theta)$ is the mass of the blow-up defined as above.
\end{prop}

Furthermore, they showed that there exists a specific $\tb \in \E{-1}$, with $\boxb \tb \in \E{4}$ (for the definitions of $\E{-1}$ and $\E{4}$, see the discussion immediately following equation (\ref{eq:Gexpansion})), such that 
\begin{equation} \label{eq:masstoboxb}
\frac{2}{3} m(M,\theta)= -\int_M |\boxb \tb|^2 \m + 2 \int_M |\tb_{,\bar{1} \bar{1}}|^2 \m + \frac{1}{2} \int_M \overline{\tb} \cdot P_{\theta} \tb \, \m,
\end{equation}
where $\tb_{,\bar{1} \bar{1}}$ is some derivative of the function $\tb$, and $P_{\theta}$ is the CR Paneitz operator of $(M,\theta)$.
(Note $R_{\theta} = 0$ in our current set-up, so the term involving $R_{\theta}$ in the corresponding identity of mass in~\cite{MR3600060} is not present above.) 
This allowed Cheng, Malchiodi and Yang to prove the following theorem:

\begin{thm}[Cheng-Malchiodi-Yang \cite{MR3600060}] \label{thm:CRposmass}
Let $(\hM,\hthe)$ be a closed 3-dimensional pseudohermitian manifold for which $Y(\hM,\hthe) > 0$. Assume in addition that $P_{\hthe} \geq 0$. Let $(M,\theta)$ be the blow-up of $(\hM,\hthe)$ as constructed above. Then 
$$
m(M,\theta) \geq 0.
$$
Furthermore, equality holds if and only if $\hM$ is CR equivalent to the standard sphere $\mathbb{S}^3$, and $\hthe$ is conformally equivalent to the standard contact form $\theta_{\text{std}}$ on $\mathbb{S}^3$.
\end{thm}

Indeed, it was shown that (\ref{eq:masstoboxb}) holds for any $\beta$ in place of $\tb$, as long as $\tb  -\beta \in \E{1+\delta}$, and $\boxb \beta \in \E{3+\delta}$ for some $\delta > 0$. Thus, if we could find such a $\beta$ in the \emph{kernel} of $\boxb$, then we have
\begin{equation} \label{eq:masstoboxb2}
\frac{2}{3} m(M,\theta)= 2 \int_M |\beta_{,\bar{1} \bar{1}}|^2 \m + \frac{1}{2} \int_M \overline{\beta} \cdot P_{\theta} \beta \, \m.
\end{equation}
The assumption that the CR Paneitz operator $P_{\hthe}$ is non-negative will be used to control the last term here: indeed the conformal invariance of the CR Paneitz operator gives that 
$$
\int_M \overline{v} \cdot P_{\theta} v \, \m = \int_{\hM} \overline{v} \cdot P_{\hthe} v \, \hthe \wedge d\hthe
$$
whenever $v \in \E{1}$. Thus writing $\beta$ as a certain main term plus some error $v \in \E{1}$, we can write the last term of (\ref{eq:masstoboxb2}) as 
$$
\frac{1}{2} \int_{\hM} \overline{v} \cdot  P_{\hthe} v \, \hthe \wedge d\hthe - \frac{2}{3} m(M,\theta).
$$
The sign of $m(M,\theta)$ works out right: in view of (\ref{eq:masstoboxb2}), we now have 
$$
\frac{4}{3} m(M,\theta) = 2 \int_M |\beta_{,\bar{1} \bar{1}}|^2 \m + \frac{1}{2} \int_{\hM} \overline{v} \cdot  P_{\hthe} v \, \hthe \wedge d\hthe.
$$
The assumption that $P_{\hthe} \geq 0$ then allows one to conclude that $m(M,\theta) \geq 0$; with some further work, when the mass $m(M,\theta) = 0$, one can construct a bijection from $\mathbb{H}^1$ to $M$ preserving the CR structure, and the bijection carries the standard contact form on $\mathbb{H}^1$ to $\theta$ on $M$. 

The above finishes the proof of Theorem~\ref{thm:A}, modulo the construction of the desired function $\beta$. To construct such a $\beta$, we will make use of the embeddability of the CR manifold $\hM$ into $\mathbb{C}^K$ for some (possibly large) $K$. This is possible according to the following result of Chanillo, Chiu and Yang \cite{MR2999315}:

\begin{thm}[Chanillo-Chiu-Yang \cite{MR2999315}] \label{thm:embed}
If $(\hM,\hthe)$ is a closed 3-dimensional pseudohermitian manifold with $Y(\hM,\hthe) > 0$ and $P_{\hthe} \geq 0$, then $\hM$ is embeddable in $\mathbb{C}^K$ for some positive integer $K$.
\end{thm}

Note how the assumption $P_{\hthe} \geq 0$ enters again in the application of this theorem. 

Theorem~\ref{thm:embed} in turn allows one to invoke the following result of J.J. Kohn \cite{Koh86}:

\begin{thm}[Kohn \cite{Koh86}] \label{thm:embed_closed_range}
Let $\hM$ be a compact strongly pseudoconvex CR manifold. If $\hM$ is embeddable in $\mathbb{C}^K$ for some positive integer $K$, then $\hdbarb$ (which we will define in Section~\ref{sect:hboxb} as a suitable $L^2$ closure of the tangential Cauchy-Riemann operator on $\hM$) has closed range in $L^2$.
\end{thm}

(The converse is also true by the estimates in Kohn \cite{Koh85} and the construction of Boutet de Monvel \cite{MR0409893}.)

From this point on, a construction of the desired $\beta$ has already appeared in an earlier work of the authors~\cite{MR3366852}, using weighted $L^2$ theory. Our goal here is to give an alternative construction of such $\beta$, using $L^p$ theory instead.

Since our result holds in more general situations (in particular, our result does not depend on the non-negatively of the CR Paneitz operator, and holds as long as the closed CR manifold $\hM$ is embeddable in some $\mathbb{C}^K$), we will set things up afresh in the next section, and then describe our main result.

\section{Statement of main result} \label{sect:result}

Let $\hM$ be a closed 3-dimensional pseudohermitian manifold that is embeddable in some $\mathbb{C}^K$, and fix a point $p \in \hM$. We choose CR normal coordinates $(z,t)$ near $p$, and let $\hthe$ be the corresponding contact form. Write $\hrho = (|z|^4 + t^2)^{1/4}$ in a neighborhood of $(0,0)$, and extend it to be a smooth positive function on all of $M$. For each $k \in \mathbb{R}$, we define, as before, $\E{k}$ to be the space of all functions $f$ on $\hM$, that is smooth away from $p$, and satisfies $|\hnabla_b^{\ell} f| \lesssim_{\ell} \hrho^{k-\ell}$ for all $\ell \in \mathbb{N} \cup \{0\}$ (hereafter $A \lesssim_{\ell} B$ means that $A \leq C_{\ell} B$ for some constant depending on $\ell$). We assume we are given a function $G \in C^{\infty}(\hM \setminus \{p\})$ with the following asymptotic expansion, where $A$ is a constant:
\begin{equation} \label{eq:Gexpansion2}
G = \frac{1}{2\pi \hrho^2} + A + e, \quad e \in \E{1}.
\end{equation}
(In particular, we do not require $G$ to be the Green's function of a conformal sublaplacian.) Let $M = \hM \setminus \{p\}$, $\theta = G^2 \hthe$. Define the following inner product on functions:
$$
(u, v)_{\theta} = \int_M u \overline{v} \, \theta \wedge d \theta,
$$
and define following inner product on $(0,1)$ forms:
$$
(\alpha,\beta)_{\theta} = \int_M \langle \alpha, \beta \rangle_{\theta} \m.
$$
Let $\dbarb$ be the operator acting on functions in $C^{\infty}(M)$, and $\vartheta_b$ be its formal adjoint under the inner products of $L^2(M)$ and $L^2_{(0,1)}(M)$. In other words, $\vartheta_b$ is the unique differential operator acting on smooth $(0,1)$ forms on $M$ such that $$(\dbarb u, \alpha)_{\theta} = (u, \vartheta_b \alpha)_{\theta}$$ for all $u \in C^{\infty}_c(M)$ and all $\alpha \in \Lambda^{(0,1)} M$. Let $\boxb = \vartheta_b \dbarb$ acting on functions $C^{\infty}(M)$. Also let $\chi_0$ be a smooth function, which is identically 1 near $p$, and is supported in a sufficiently small neighborhood of $p$. Write
$$
\beta_0(z,t) = \chi_0(z,t) \frac{i\zbar}{|z|^2-it} \in \E{-1}.
$$
It is known that
$$
\boxb \beta_0 \in \E{3}.
$$ 
Suppose there exists 
$
\beta_1 \in \E{1}
$ 
such that 
$$
\tb := \beta_0  + \beta_1 
$$ 
satisfies 
$$
\boxb \tb \in \E{3+\delta} \quad \text{for some $0 < \delta < 1$}.
$$
(Such $\beta_1$ is known to exist on all 3-dimensional asymptotically flat pseudohermitian manifolds; see \cite{MR3600060}.)
Let 
$$f := \boxb \tb \in \E{3+\delta}.$$
Our main theorem then states:
\begin{thm} \label{thm:goal}
There exists a function $u$ such that 
\begin{equation} \label{eq:boxb}
u \in \E{1+\delta} \quad \text{with} \quad \boxb u = f.
\end{equation}
\end{thm}

To prove the theorem, we need to be able to solve a certain tangential Kohn Laplacian on $\hM$. The relative solution operators and the Szeg\H{o} projections to this tangential Kohn Laplacian will be constructed inside a certain algebra of non-isotropic pseudodifferential operators on $\hM$. We will describe this algebra of pseudodifferential operators in Section~\ref{sect:pdo}, and use it to solve the sublaplacian and certain tangential Kohn Laplacians on $\hM$ in Sections~\ref{sect:hDelb} and \ref{sect:hboxbsol} respectively. Then in Section~\ref{sect:eg} we prove Theorem~\ref{thm:goal} in a special case. This will motivate the proof of the theorem in the general case, which we give in Sections~\ref{sect:tboxb} and \ref{sect:hboxb}.

\section{A class of pseudodifferential operators} \label{sect:pdo}

The algebra of non-isotropic pseudodifferential operators we describe below is similar to the ones constructed in the monographs of Nagel and Stein \cite{NaSt} and Beals and Greiner \cite{BG88}. The formulation we use here is slightly simpler, and is adapted from Stein and Yung~\cite{MR3228630}.  

The following algebra of non-isotropic pseudodifferential operators is defined whenever one has a smooth codimension 1  distribution $\mathcal{D}$ of tangent subspaces to a Euclidean space $\mathbb{R}^N$, or more generally to a closed manifold $\hM$ without boundary. In applications, we will always take this distribution to be the contact distribution $\hor$ on $\hM$. But since we will not need this distribution to be the contact distribution for most of what follows, we will not assume this distribution $= \hor$ until stated otherwise.

First, let $\mathcal{D}$ be a smooth codimension 1 distribution on some $\mathbb{R}^N$ ($N$ will be equal to 3 in our applications, but for now it will just be any integer $\geq 2$). We pick $N$ linearly independent vector fields $X_1, \dots, X_N$, so that the first $N-1$ vectors span $\mathcal{D}$, and 
$$
X_i = \sum_{j=1}^N A_i^j(x) \frac{\partial}{\partial x^j}, \quad 1 \leq i \leq N.
$$
By saying that $\mathcal{D}$ is smooth, we always assume that all coefficients $A_i^j(x)$ are $C^{\infty}$ smooth, and has uniformly bounded $C^k$ norm for all $k$; we also assume the existence of some constant $c > 0$, such that
$$
|\det(A_i^j(x))| \geq c
$$
for all $x \in \mathbb{R}^N$. We will use this frame of vector fields to define our class of non-isotropic pseudodifferential operators on $\mathbb{R}^N$, but ultimately the class of pseudodifferential operators we obtain will depend only on $\mathcal{D}$, and not on the choice of this frame.

We now define a variable coefficient linear map on the cotangent spaces, namely
$$
M_x \colon T_x^*(\mathbb{R}^N) \to T_x^*(\mathbb{R}^N),
$$
by
\begin{equation} \label{eq:Mxdef}
M_x \xi = \sum_{i=1}^N \left( \sum_{j=1}^N A_i^j(x) \xi_j \right) dx^i \quad \text{if $\xi = \sum_{j=1}^N \xi_j dx^j$}. 
\end{equation}
%We will also need a variable semi-norm on the cotangent bundle of $\mathbb{R}^N$, given by
%$$ \rho_x(\xi) = \left( \sum_{i=1}^{N-1} |(M_x \xi)_i|^2 \right)^{1/2} \quad \text{if $M_x \xi = \sum_{i=1}^N (M_x \xi)_i dx^i$} $$
%(note only the first $N-1$ components of $M_x \xi$ are involved). 

Let $n \in \mathbb{R}$. If $a(x,\xi)$ is a function on the cotangent bundle of $\mathbb{R}^N$, such that 
$$
a(x,\xi) = a_0(x,M_x \xi),
$$
for some symbol $a_0$ satisfying
\begin{equation} \label{eq:a0diff}
|\partial_x^I \partial_{\xi}^{\alpha} a_0(x,\xi)| \lesssim_{I,\alpha} (1+\|\xi\|)^{n-\|\alpha\|},
\end{equation}
where
$$
\|\xi\| := |\xi'| + |\xi_N|^{1/2} \quad \text{if $\xi = (\xi',\xi_N)$}
$$
and
$$
\|\alpha\| := |\alpha'| + 2 \alpha_N \quad \text{if $\alpha = (\alpha', \alpha_N)$},
$$
then $a(x,\xi)$ is called a (non-isotropic) symbol of order $n$ on $\mathbb{R}^N$ adapted to $\mathcal{D}$, written $a \in S^n_{\mathcal{D}}$.

To every $a \in S^n_{\mathcal{D}}$ we associate a pseudodifferential operator $T$ on $\mathbb{R}^N$,  such that 
$$
Tf(x) = \int_{\mathbb{R}^N} a(x,\xi) \widehat{f}(\xi) e^{2\pi i x \cdot \xi} d\xi,
$$ 
where $\widehat{f}$ is the Fourier transform of $f$, defined by
$$
\widehat{f}(\xi) = \int_{\mathbb{R}^N} f(x) e^{-2\pi i x \cdot \xi} dx.
$$
Such a pseudodifferential operator is said to be of non-isotropic order $n$ and adapted to $\mathcal{D}$; we denote the class of all such pseudodifferential operators by $\Psi^n_{\mathcal{D}}(\mathbb{R}^N)$. 

It can be shown that the class of operators $\Psi^n_{\mathcal{D}}(\mathbb{R}^N)$ depends only on the distribution $\mathcal{D}$, and is otherwise independent of the choice of the frame $X_1, \dots, X_N$ we made earlier. This is because if $\tilde{X}_1, \dots, \tilde{X}_N$ is another frame so that $\tilde{X}_1, \dots, \tilde{X}_{N-1}$ span $\mathcal{D}$, and $\tilde{M}_x$ is the associated variable coefficient linear map on the cotangent space of $\mathbb{R}^N$, then $\tilde{M}_x M_x^{-1}$ preserves the subspace spanned by $dx^N$, which implies that $\|\tilde{M}_x M_x^{-1} \xi\| \simeq \|\xi\|$ for all $\xi \in \mathbb{R}^N$. One can also show that the class $\Psi^n_{\mathcal{D}}(\mathbb{R}^N)$ is independent of the choice of coordinates $x$ on $\mathbb{R}^N$. Hence one can transplant this onto a smooth manifold, and obtain a class of non-isotropic pseudodifferential operators on a closed smooth manifold if one is given a codimension 1 distribution of tangent subspaces on the manifold.

More precisely, let $\hM$ be a closed smooth manifold of real dimension $N$, and $\hD$ be a smooth distribution of codimension 1 tangent subspaces of $\hM$. A coordinate chart $x \colon U \to \mathbb{R}^N$ is said to be admissible, if $U$ is a contractible open set in $\hM$, $x$ is a diffeomorphism of $U$ onto the unit ball in $\mathbb{R}^N$, and there exists a smooth frame of tangent vectors $X_1, \dots, X_N$ on $U$, such that $\hD$ is spanned by $X_1, \dots, X_{N-1}$ at every point of $U$. If $x \colon U \to \mathbb{R}^N$ is an admissible coordinate chart, one can find a smooth distribution $\mathcal{D}$ of codimension 1 tangent subspaces on $\mathbb{R}^N$, that agrees with $dx(\hD)$ on $x(U)$ and satisfy all our previous assumptions about a codimension 1 distribution of tangent subspaces on $\mathbb{R}^N$ (in fact, $\mathcal{D}$ can be chosen to be constantly equal to say the nullspace of $dx^N$ outside a ball of radius 2). A linear operator $S \colon C^{\infty}(\hM) \to C^{\infty}(\hM)$ is said to be a pseudodifferential operator of non-isotropic order $n$ adapted to $\hD$, written $S \in \Psi^n_{\hD}(\hM)$, if the following are true: 
\begin{enumerate}[(a)]
\item For any admissible coordinate chart $x \colon U \to \mathbb{R}^N$ and any $\chi_1, \chi_2 \in C^{\infty}_c(U)$, the operator $\chi_1 S \chi_2 \in \Psi^n_{\mathcal{D}}(\mathbb{R}^N)$, where $\mathcal{D}$ is the extension of $dx(\hD)$ from $x(U)$ to $\mathbb{R}^N$ as described above;
\item For any $\varsigma_1, \varsigma_2 \in C^{\infty}_c(\hM)$ with disjoint supports, there exists a smooth kernel $s(x,y) \in C^{\infty}(\hM \times \hM)$, such that $$\varsigma_1 S \varsigma_2 f(x) = \int_{\hM} s(x,y) f(y) dy$$ for all $f \in C^{\infty}(\hM)$.
\end{enumerate}
In practice, to show that $S \in \Psi^n_{\hD}(\hM)$, condition (b) usually follows from pseudolocality, and is pretty automatic. In addition, when we check condition (a), we only need to cover $\hM$ by finitely many admissible coordinate charts, and check condition (a) on those; this is because the class $\Psi^n_{\mathcal{D}}(\mathbb{R}^N)$ is invariant under smooth changes of coordinates.

Having said that, we still need to know how to check condition (a) on some admissible coordinate chart. This is often done via establishing \emph{kernel estimates}, and \emph{cancellation conditions} in the case of operators of non-negative orders. To describe this, let's return to the situation where a smooth codimension 1 distribution $\mathcal{D}$ is given on $\mathbb{R}^N$. 
%Let $x \colon U \to \mathbb{R}^N$ be an admissible coordinate chart on $\hM$, and identify $U$ as part of $\mathbb{R}^N$ via this chart. Let $\mathcal{D}$ be the extension of $\hD$ to $\mathbb{R}^N$, and 
Let
\begin{equation} \label{eq:XiAij}
X_i = \sum_{j=1}^N A^j_i(x) \frac{\partial}{\partial x^j}, \quad 1 \leq j \leq N
\end{equation}
so that $X_1, \dots, X_{N-1}$ span $\mathcal{D}$. Let $(B^k_j)$ be the inverse of the matrix $(A^j_i)$; i.e. let
$$
\sum_{j=1}^N A^j_i(x) B^k_j(x) = \delta^k_i \quad \text{for every $x \in \mathbb{R}^N$}.
$$
We define a variable coefficient linear map $L_x \colon \mathbb{R}^N \to \mathbb{R}^N$, so that the $k$-th component of $L_x(u)$ is $\sum_{j=1}^N B^k_j(x) u^j$ for $1 \leq k \leq N$. 
We also define $\Theta_0 \colon \mathbb{R}^N \times \mathbb{R}^N \to \mathbb{R}^N$, so that 
\begin{equation} \label{eq:Theta0}
\Theta_0(x,y) = L_x(x-y),
\end{equation}
and set
$$
Q := N + 1.
$$
We then have the following theorem:

\begin{thm} \label{thm:equiv_smoothing} 
Suppose $-Q < n < 0$. Then $T \in \Psi^n_{\mathcal{D}}(\mathbb{R}^N)$, if and only if there exists a tempered distribution $k_0(x,u)$ on $\mathbb{R}^N \times \mathbb{R}^N$, that is smooth on $\mathbb{R}^N \times (\mathbb{R}^N \setminus \{0\})$, and satisfies
\begin{equation} \label{eq:kernelest}
|\partial_x^I \partial_u^{\gamma} k_0(x,u)| \lesssim_{I,\gamma,M} \|u\|^{-Q-n-\|\gamma\|-M}
\end{equation}
for all $\gamma$ and $I$, and all $M \geq 0$, such that 
\begin{equation} \label{eq:Tkerrep}
Tf(x) = \int_{\mathbb{R}^N} k_0(x,\Theta_0(x,y)) f(y) dy \quad \text{for all $x \in \mathbb{R}^N$}.
\end{equation}
\end{thm}

We remark that if $|x-y| \leq 1$, then $|\Theta_0(x,y)^N|$, and hence the norm $\|\Theta_0(x,y)\|$, remains comparable upon choosing a different frame $X_1, \dots, X_N$, as long as the first $N-1$ vectors still span $\mathcal{D}$; here $\Theta_0(x,y)^N$ is the $N$-th component of $\Theta_0(x,y) \in \mathbb{R}^N$. So the theorem is a statement that is independent of the choice of frames $X_1, \dots, X_N$.

The situation for $\Psi^0_{\mathcal{D}}(\mathbb{R}^N)$ is a bit more involved; one must take into account certain cancellation conditions. To formulate this, we say a function $\phi(u)$ is a normalized bump function, if $\phi(u)$ is smooth with compact support on the unit ball $\{\|u\| \leq 1\}$, and 
$$
\|\partial_u^{\gamma} \phi(u)\|_{L^{\infty}} \lesssim_{\gamma} 1
$$
for all multiindices $\gamma$; we write $\phi_r(u)$ for its non-isotropic dilation 
$$
\phi_r(u) = \phi(r^{-1} u', r^{-2} u_N)
$$
for $r \in (0,1)$.

\begin{thm} \label{thm:equiv_nonneg_order}
Suppose $n \geq 0$. Then $T \in \Psi^n_{\mathcal{D}}(\mathbb{R}^N)$, if and only if there exists a tempered distribution $k_0(x,u)$ on $\mathbb{R}^N \times \mathbb{R}^N$, such that the following holds:
\begin{enumerate}[(a)]
\item $k_0(x,u)$ is smooth on $\mathbb{R}^N \times (\mathbb{R}^N \setminus \{0\})$, and satisfies the kernel estimates (\ref{eq:kernelest});
\item for every normalized bump function $\phi(u)$ and every $r \in (0,1)$, we have
$$
\sup_{x \in \mathbb{R}^N} \left| \langle \partial_x^I k_0(x,u), \phi_r(u) \rangle \right| \lesssim_I r^{-n}
$$
for all multiindices $I$;
\item the representation formula (\ref{eq:Tkerrep}) holds, in the sense that
$$
\langle Tf, g \rangle = \langle k_0(x,u), f(x-L_x^{-1}(u)) g(x) \rangle
$$
for all Schwartz functions $f$ and $g$.
\end{enumerate}
\end{thm}

Here $\langle \cdot, \cdot \rangle$ denotes the pairing of a tempered distribution and a Schwartz function, either on $\mathbb{R}^N$ or on $\mathbb{R}^{2N}$.

We will need a small generalization of Theorems~\ref{thm:equiv_smoothing} and \ref{thm:equiv_nonneg_order}. Suppose $\Theta \colon \{(x,y) \in \mathbb{R}^N \times \mathbb{R}^N \colon |x-y| < 1\} \to \mathbb{R}^N$ is a $C^{\infty}$ map defined on the unit tubular neighborhood of the diagonal on $\mathbb{R}^N$, with $\Theta(x,x) = 0$ for all $x \in \mathbb{R}^N$ and $\|\Theta(x,y)\|_{C^k} \lesssim 1$ for all $k \in \mathbb{N}$. We write 
\begin{equation} \label{eq:tildeB_def}
\mathcal{B}^k_j(x) = -\left. \frac{\partial}{\partial y^j} \right|_{y = x} \Theta(x,y)^k,
\end{equation}
and assume in addition that $|\det(\mathcal{B}^k_j(x))|$ has a uniform positive lower bound. Such a map $\Theta$ is said to be compatible with our given distribution $\mathcal{D}$, if and only if $\mathcal{D}$ is the nullspace of $\sum_{j=1}^N \mathcal{B}^N_j(x) dx^j$ at every point $x \in \mathbb{R}^N$. For instance, the $\Theta_0$ we constructed earlier is an example of such a $\Theta$. Note that in general, such a $\Theta$ can be written as
\begin{equation} \label{eq:Theta}
\Theta(x,y) = L_{x,x-y}(x-y)
\end{equation}
where $L_{x,u} \colon \mathbb{R}^N \to \mathbb{R}^N$ is an invertible linear map that varies smoothly with $x, u \in \mathbb{R}^N$ and $|u| < 1$ (see, e.g. p.99 of Nagel-Stein \cite{NaSt}).
We have:

\begin{thm} \label{thm:equiv_general}
Suppose $n > -Q$, and that $k_0(x,u)$ is a tempered distribution on $\mathbb{R}^N \times \mathbb{R}^N$ that satisfies the condition (a) of Theorem~\ref{thm:equiv_nonneg_order}. If $n \geq 0$, assume further that $k_0(x,u)$ satisfies the condition (b) of Theorem~\ref{thm:equiv_nonneg_order}. Then for any $\Theta$ that is compatible with $\mathcal{D}$, and any smooth cut-off function $\chi(x,y)$ supported on the unit tubular neighborhood of the diagonal on $\mathbb{R}^N$, the map $T$ defined by
$$
Tf(x) = \int_{\mathbb{R}^N} \chi(x,y) k_0(x,\Theta(x,y)) f(y) dy
$$ 
is in $\Psi^n_{\mathcal{D}}(\mathbb{R}^N)$, where the last integral should be interpreted as
$$
\langle k_0(x,u), \chi(u) f(x-L_{x,u}^{-1} u) \rangle
$$
for any Schwartz function $f$ on $\mathbb{R}^N$.
\end{thm}

This theorem will allow us to construct examples of operators in $\Psi^n_{\hD}(\hM)$.

An observation that will often make life easier for us is that the non-isotropic pseudodifferential operators we introduced are automatically (isotropically) of type $(1/2,1/2)$ in the sense of H\"ormander: if $a \in S^n_{\mathcal{D}}$, then automatically we have
$$
|\partial_x^I \partial_{\xi}^{\alpha} a(x,\xi)| \lesssim_{I,\alpha} (1+|\xi|)^{|n|+\frac{|I|-|\alpha|}{2}}.
$$ 
In particular, we have
\begin{prop} \label{prop:pseudolocal}
If $T \in \Psi^n_{\mathcal{D}}(\mathbb{R}^N)$ for some $n \in \mathbb{R}$, then $T$ is pseudolocal, in the sense that $T$ preserve $C^{\infty}(U)$ for any open set $U \subset \mathbb{R}^N$. As a result, the same is true for the class $\Psi^n_{\hD}(\hM)$.
\end{prop}

We will need to be able to compose the non-isotropic pseudodifferential operators we introduced, via the following theorem:

\begin{thm} \label{thm:compose}
If $T_1 \in \Psi^{n_1}_{\mathcal{D}}(\mathbb{R}^N)$ and $T_2 \in \Psi^{n_2}_{\mathcal{D}}(\mathbb{R}^N)$ for some $n_1, n_2 \in \mathbb{R}$, then $$T_1 \circ T_2 \in \Psi^{n_1+n_2}_{\mathcal{D}}(\mathbb{R}^N).$$ As a result, if $S_1 \in \Psi^{n_1}_{\hD}(\hM)$ and $S_2 \in \Psi^{n_2}_{\hD}(\hM)$ for some $n_1, n_2 \in \mathbb{R}$, then $$S_1 \circ S_2 \in \Psi^{n_1+n_2}_{\hD}(\hM).$$
\end{thm}

Note that in general there is no simple asymptotic expansion formula for the symbol of $T_1 \circ T_2$ in terms of the symbols of $T_1$ and $T_2$ if both $T_1$ and $T_2$ are non-isotropic pseudodifferential operators as in the previous theorem. But there is indeed an asymptotic expansion formula if we compose an isotropic pseudodifferential operator of type $(1,0)$, with a non-isotropic pseudodifferential operator of the kind we introduced above; see \cite{MR3228630} for a more extensive discussion of what happens when one composes these two kinds of operators. We will not need the full strength of the theory developed in \cite{MR3228630}; what we need here is only the following simple special commutation relation:

\begin{thm} \label{thm:comm}
Suppose $-Q + 1 < n < 0$. If $\eta \in C^{\infty}_c(\mathbb{R}^N)$ and $T \in \Psi^n_{\mathcal{D}}(\mathbb{R}^N)$, then $$[\eta, T] \in \Psi^{n-1}_{\mathcal{D}}(\mathbb{R}^N).$$ As a result, if $\eta \in C^{\infty}(\hM)$ and $S \in \Psi^n_{\hD}(\hM)$, then $$[\eta, S] \in \Psi^{n-1}_{\hD}(\hM).$$
\end{thm}

Here by abuse of notation, we write $\eta$ also for the multiplication by $\eta$.

In addition, we will also need the $L^p$ boundedness of the non-isotropic pseudodifferential operators of order 0:

\begin{thm} \label{thm:Lpbdd}
If $T \in \Psi^0_{\mathcal{D}}(\mathbb{R}^N)$, then $T$ extends to a bounded linear operator $$T \colon L^p(\mathbb{R}^N) \to L^p(\mathbb{R}^N)$$ for all $1 < p < \infty$. As a result, if $S \in \Psi^0_{\hD}(\hM)$, then $S$ extends to a bounded linear operator $$ S \colon L^p(\hM) \to L^p(\hM)$$ for all $1 < p < \infty$.
\end{thm}

Finally, we will also need to be able to form asymptotic sums of symbols:

\begin{thm} \label{thm:asym_sum}
Suppose $n \in \mathbb{R}$ and $T_j \in \Psi^{n-j}_{\mathcal{D}}(\mathbb{R}^N)$ for $j = 0, 1, 2, \dots$. Then there exists $T \in \Psi^n_{\mathcal{D}}(\mathbb{R}^N)$ such that $$T - (T_0 + \dots + T_k) \in \Psi^{n-k-1}_{\mathcal{D}}(\mathbb{R}^N)$$ for all $k \in \mathbb{N}$. As a result, if $S_j \in \Psi^{n-j}_{\hD}(\hM)$ for $j = 0, 1, 2, \dots$. Then there exists $S \in \Psi^n_{\hD}(\hM)$ such that $$S - (S_0 + \dots + S_k) \in \Psi^{n-k-1}_{\hD}(\hM)$$ for all $k \in \mathbb{N}$.
\end{thm}

It will often be convenient to write $T \in \Psi^{-\infty}_{\mathcal{D}}(\mathbb{R}^N)$ if $T \in \Psi^{-k}_{\mathcal{D}}(\mathbb{R}^N)$ for all $k \geq 0$. Similarly we write $S \in \Psi^{-\infty}_{\hD}(\hM)$ if $S \in \Psi^{-k}_{\hD}(\hM)$ for all $k \geq 0$. Note that $\Psi^{-\infty}_{\mathcal{D}}(\mathbb{R}^N)$ agrees with the class of isotropic infinitely smoothing pseudodifferential operators of type $(1,0)$; indeed $a(x,\xi)$ is a symbol of an operator in $\Psi^{-\infty}_{\mathcal{D}}(\mathbb{R}^N)$, if and only if 
$$
|\partial_x^I \partial_{\xi}^{\alpha} a(x,\xi)| \lesssim_{I,\alpha,k} (1+|\xi|)^{-k}
$$ 
for all $k \geq 0$ and all multiindices $I$, $\alpha$.
Thus $S \in \Psi^{-\infty}_{\hD}(\hM)$, if and only if there exists $s(x,y) \in C^{\infty}(\hM \times \hM)$ such that $Sf(x) = \int_{\hM} f(y) s(x,y) \hthe \wedge d\hthe(y)$ for all $f$.

We will briefly describe the proofs of the above theorems in what follows.

\begin{proof}[Proof of Theorem~\ref{thm:equiv_smoothing}]
The theorem essentially follows from the simple fact that if $-Q < n < 0$ and $a_0(\xi)$ is a function on $\mathbb{R}^N$, then it satisfies the differential inequalities
$$
|\partial_{\xi}^{\alpha} a_0(\xi)| \lesssim_{\alpha} (1+\|\xi\|)^{n-\|\alpha\|}
$$
for all multiindices $\alpha$, if and only if its Fourier transform $k_0(u)$ satisfies the differential inequalities
$$
|\partial_u^{\gamma} k_0(u)| \lesssim_{\gamma,M} \|u\|^{n-\|\gamma\|-M}
$$
for all multiindices $\gamma$ and all $M \geq 0$.

Indeed, suppose $\mathcal{D}$ is a smooth codimension 1 distribution on $\mathbb{R}^N$. Let $X_1, \dots, X_N$ be a frame of the tangent bundle on $\mathbb{R}^N$, so that the first $N-1$ of them span $\mathcal{D}$, and let $X_i$ be given by some coefficients $(A_i^j)$ as in (\ref{eq:XiAij}) for $1 \leq i \leq N$. Let $L_x \colon \mathbb{R}^N \to \mathbb{R}^N$ be a variable coefficient linear map, so that the $k$-th component of $L_x(u)$ is $\sum_{j=1}^N B^k_j(x) u^j$ for $1 \leq k \leq N$; here $(B_j^k)$ is the inverse of the matrix $(A_i^j)$. The inverse transpose $L_x^{-t}$ of $L_x$ is precisely the $M_x$ we defined in (\ref{eq:Mxdef}), and we have $\Theta_0(x,y) = L_x(x-y)$ as in (\ref{eq:Theta0}).

Now let $k_0(x,u)$ be an integral kernel satisfying (\ref{eq:kernelest}), and $T$ be the operator defined by (\ref{eq:Tkerrep}). Then $k_0(x,u)$ is an integrable function of $u$ for every $x \in \mathbb{R}^N$. Writing $a_0(x,\xi)$ for the Fourier transform of $k_0(x,u)$ in the $u$ variable, so that
$$
k_0(x,u) = \int_{\mathbb{R}^N} a_0(x,\xi) e^{2\pi i u \cdot \xi} d\xi,
$$
we have
\begin{align*}
k_0(x,\Theta_0(x,y)) 
&= \int_{\mathbb{R}^N} a_0(x,\xi) e^{2\pi i L_x(x-y) \cdot \xi} d\xi \\
&= \det(M_x) \int_{\mathbb{R}^N} a_0(x,M_x \xi) e^{2\pi i (x-y) \cdot \xi} d\xi 
\end{align*}
where in the last line we have changed variable $\xi \mapsto M_x \xi$. Plugging this back in the definition of $Tf(x)$ in (\ref{eq:Tkerrep}), this shows that
$$
Tf(x) = \int_{\mathbb{R}^N} \det(M_x) a_0(x,M_x \xi) \widehat{f}(\xi) e^{2\pi i x \cdot \xi} d\xi.
$$
Since $\det(M_x)$ is smooth and bounded along with all its derivatives in $x$, it remains to show that $a_0(x,\xi)$ satisfies the bound (\ref{eq:a0diff}). But this follows from the fact we stated above, since $k_0(x,u)$ satisfies (\ref{eq:kernelest}) for some $-Q < n < 0$. This proves one implication in Theorem~\ref{thm:equiv_smoothing}; the converse can be proved similarly.
\end{proof}

\begin{proof}[Proof of Theorem~\ref{thm:equiv_nonneg_order}]
Suppose now $n \geq 0$ and $a_0(\xi)$ is a function on $\mathbb{R}^N$. We now need the following characterization of the Fourier transform of $a_0$: it turns out $a_0(\xi)$ satisfies the differential inequalities
$$
|\partial_{\xi}^{\alpha} a_0(\xi)| \lesssim_{\alpha} (1+\|\xi\|)^{n-\|\alpha\|}
$$
for all multiindices $\alpha$, if and only if its Fourier transform $k_0(u)$ satisfies the differential inequalities
$$
|\partial_u^{\gamma} k_0(u)| \lesssim_{\gamma,M} \|u\|^{n-\|\gamma\|-M}
$$
for all multiindices $\gamma$ and all $M \geq 0$, and the cancellation conditions
$$
|\langle k_0(u), \phi_r(u) \rangle| \lesssim r^{-n}
$$
for all normalized bump function $\phi$ and all $r \in (0,1)$. This fact can be proved, for instance, by adapting the proof of Proposition 3 in \cite[Chapter VI.4.5]{Ste93}. Once we have this fact, then the proof of Theorem~\ref{thm:equiv_smoothing} above can be adapted to give a proof of Theorem~\ref{thm:equiv_nonneg_order}.
\end{proof}

\begin{proof}[Proof of Theorem~\ref{thm:equiv_general}]
Let $a_0(x,\xi)$ be the Fourier transform of $k_0(x,u)$ in the $u$-variable; in particular $a_0(x,\xi)$ satisfies condition (\ref{eq:a0diff}). Then if we express the Schwartz function $f$ in the definition of $Tf(x)$ in terms of $\widehat{f}$, we have (at least formally)
$$
Tf(x) = \int_{\mathbb{R}^N} \tilde{a}_0(x,M_x \xi) \widehat{f}(\xi) e^{2 \pi i x \cdot \xi} d\xi
$$
where
$$
\tilde{a}_0(x,M_x \xi) = \int_{\mathbb{R}^N} \chi(x,y) k_0(x,\Theta(x,y)) e^{-2\pi i (x-y) \cdot \xi} dy,
$$
i.e.
$$
\tilde{a}_0(x,\xi) = \int_{\mathbb{R}^N} \chi(x,y) k_0(x,\Theta(x,y)) e^{-2\pi i (x-y) \cdot L_x^t \xi} dy.
$$
Now write $\Theta(x,y) = L_{x,x-y}(x-y)$ as in (\ref{eq:Theta}). Then expressing $k_0(x,u)$ in terms of $a_0(x,\zeta)$ via the Fourier inversion formula, we have
$$
\tilde{a}_0(x,\xi) = \int_{\mathbb{R}^N} \int_{\mathbb{R}^N} \chi(x,y) a_0(x,\zeta) e^{2\pi i L_{x,x-y}(x-y) \cdot \zeta} e^{-2\pi i (x-y) \cdot L_x^t \xi} d\zeta dy.
$$
We write $L_{x,x-y}(x-y) \cdot \zeta = (x-y) \cdot L_{x,x-y}^{-t} \zeta$ and make a change of variable $\zeta \mapsto L_{x,x-y}^{-t} L_x^t \zeta$. Then
$$
\tilde{a}_0(x,\xi) = \int_{\mathbb{R}^N} \int_{\mathbb{R}^N} \chi(x,y) \det(L_{x,x-y})^{-1} \det(L_x) a_0(x,L_{x,x-y}^{-t} L_x^t \zeta) e^{-2\pi i (x-y) \cdot L_x^t (\xi-\zeta)} d\zeta dy.
$$
Now we write $(x-y) \cdot L_x^t(\xi-\zeta) = L_x (x-y) \cdot (\xi-\zeta)$, and make another change of variable $y \mapsto x-L_x^{-1}(x-y)$. Then
$$
\tilde{a}_0(x,\xi) = \int_{\mathbb{R}^N} \int_{\mathbb{R}^N} \tilde{\chi}(x,y) a_0(x, L_{x,L_x^{-1}(x-y)}^{-t} L_x^t \zeta) e^{-2\pi i (x-y) \cdot (\xi-\zeta)} d\zeta dy
$$
where $\tilde{\chi}(x,y)$ is the smooth function $\chi(x,x-L_x^{-1}(x-y)) \det(L_{x,L_x^{-1}(x-y)})^{-1}$. Note that 
$
L_{x,L_x^{-1} u}^{-t} L_x^t 
$
is the identity map $I$ when $u = 0$. Thus we may write
$$
L_{x,L_x^{-1}(x-y)}^{-t} L_x^t = I + (x-y) \tilde{L}_{x,y}
$$ 
for a vector of variable coefficient linear maps $\tilde{L}_{x,y}$ on $\mathbb{R}^N$; more precisely, there are variable coefficient linear maps $\tilde{L}_{x,y}^{(j)}$, $1 \leq j \leq N$, such that
$$
L_{x,L_x^{-1}(x-y)}^{-t} L_x^t = I + \sum_{j=1}^N (x^j-y^j) \tilde{L}_{x,y}^{(j)}.
$$ 
Thus
\begin{equation} \label{eq:specific_atilde0}
\tilde{a}_0(x,\xi) = \int_{\mathbb{R}^N} \int_{\mathbb{R}^N} \tilde{\chi}(x,y) a_0(x, [I + (x-y) \tilde{L}_{x,y}] \zeta) e^{-2\pi i (x-y) \cdot (\xi-\zeta)} d\zeta dy
\end{equation}
In other words, $\tilde{a}_0(x,\xi)$ is the symbol associated to the compound symbol 
\begin{equation} \label{eq:specific_c}
c(x,y,\xi) = \tilde{\chi}(x,y) a_0(x, [I + (x-y) \tilde{L}_{x,y}] \xi).
\end{equation}
The key now is to show that $\tilde{a}_0(x,\xi)$ also satisfies the differential inequalities (\ref{eq:a0diff}) in place of $a_0(x,\xi)$. To do so, we introduce a variable coefficient norm function $\rho_{x,y}(\xi)$ on $\mathbb{R}^N$.

More precisely, for $x, y, \xi \in \mathbb{R}^N$, let $$\rho_{x,y}(\xi) = \|[I + (x-y) \tilde{L}_{x,y}] \xi\|,$$ so that
$$
\left|\rho_{x,y}(\xi) - \|\xi\| \right| \lesssim |x-y| |\xi| \quad \text{whenever $|\xi| \geq 1$}.
$$
In particular, for $|x-y|$ sufficiently small, we have
\begin{equation} \label{eq:rhoxy1}
1+\rho_{x,y}(\xi) \lesssim (1+\|\xi\|) (1+|x-y||\xi|^{\frac{1}{2}})
\end{equation}
and
\begin{equation} \label{eq:rhoxy2}
\frac{1}{1+\rho_{x,y}(\xi)} \lesssim \frac{1+|x-y||\xi|^{\frac{1}{2}}}{1+\|\xi\|}.
\end{equation}
Next, we observe that $c(x,y,\xi)$ defined by (\ref{eq:specific_c}) satisfies
\begin{equation} \label{eq:c_diffineq}
|\partial_x^I \partial_y^J \partial_{\xi}^{\alpha} c(x,y,\xi)| \lesssim (1+|\xi|)^{\frac{|I|+|J|-|\alpha|}{2}} (1+\rho_{x,y}(\xi))^n.
\end{equation}
From (\ref{eq:specific_atilde0}), and a further change of variables $y \mapsto x-w$, $\zeta \mapsto \xi-\zeta$, we see that
\begin{equation} \label{eq:tildea0_def_c}
\tilde{a}_0(x,\xi) = \int_{\mathbb{R}^N} \int_{\mathbb{R}^N} c(x,x-w,\xi-\zeta) e^{-2\pi i w \cdot \zeta} d\zeta dy.
\end{equation}
Using (\ref{eq:c_diffineq}) only, we will prove that the $\tilde{a}_0(x,\xi)$ given by (\ref{eq:tildea0_def_c}) satisfies 
\begin{equation} \label{eq:tildea0_bdd0}
|\tilde{a}_0(x,\xi)| \lesssim (1+\|\xi\|)^n.
\end{equation}
The idea is to Taylor expand in $\zeta$ in (\ref{eq:tildea0_def_c}). Indeed, let $\eta \in C^{\infty}_c([-1,1])$ that is identically 1 on $(-1/2,1/2)$. Then $\tilde{a}_0(x,\xi)$ is equal to
$$
\int_{\mathbb{R}^N} \int_{\mathbb{R}^N} \eta \left( \frac{|\zeta|}{|\xi|/2} \right) c(x,x-w,\xi-\zeta) e^{-2\pi i w \cdot \zeta} dw d\zeta
$$
up to an error that is rapidly decreasing in $\xi$. Now write 
$$
e^{-2\pi i w \cdot \zeta} = \left( \frac{I-|\xi|^{-1} \Delta_w - |\xi| \Delta_{\zeta}}{1+4\pi^2 |\xi| |w|^2 + 4\pi^2 |\xi|^{-1} |\zeta|^2} \right)^M e^{-2\pi i w \cdot \zeta}
$$
for some large positive integer $M$ and integrate by parts. On the support of this integral, $|\zeta| \leq |\xi|/2$, so in particular $|\xi - \zeta| \simeq |\xi|$. Using this and (\ref{eq:c_diffineq}), we have
$$
|\tilde{a}_0(x,\xi)| \lesssim \int_{|\zeta| \leq |\xi|/2} \int_{\mathbb{R}^N} \frac{(1+\rho_{x,x-w}(\xi-\zeta))^n}{(1+ |\xi| |w|^2 + |\xi|^{-1} |\zeta|^2)^M} dw d\zeta.
$$
To estimate the numerator in the above integral, if $n \geq 0$, we use (\ref{eq:rhoxy1}); if $n < 0$, we use (\ref{eq:rhoxy2}). Thus using $|\xi - \zeta| \simeq |\xi|$ again on the support of the integral, we get
$$
|\tilde{a}_0(x,\xi)| \lesssim (1+\|\xi\|)^n \int_{|\zeta| \leq |\xi|/2} \int_{\mathbb{R}^N} \frac{(1+|w||\xi|^{1/2})^{|n|}}{(1+ |\xi| |w|^2 + |\xi|^{-1} |\zeta|^2)^M} dw d\zeta,
$$
which is $\lesssim (1+\|\xi\|)^n$ if $M$ is large enough. This establishes (\ref{eq:tildea0_bdd0}). We also need to bound the derivatives of $\tilde{a}_0(x,\xi)$; but from the explicit expression (\ref{eq:specific_c}) for $c(x,y,\xi)$, we see that
$$
c(x,x-w,\xi-\zeta) = \tilde{\chi}(x,x-w) a_0(x, [I+w \tilde{L}_{x,x-w}] (\xi-\zeta))
$$
so for any multiindices $I$ and $\alpha$,
$$
\partial_x^{I} \partial_{\xi}^{\alpha} [c(x,x-w,\xi-\zeta)] = \sum_{|\sigma| \leq |I|} w^{\sigma} c_{\sigma}^{(I,\alpha)}(x,x-w,\xi-\zeta)
$$
for some compound symbols $c_{\sigma}^{(I,\alpha)}$, where the sum is over all multiindices $\sigma$ of length at most $|I|$ (this sum arises because for each $\partial_x$ that hits $[I+w \tilde{L}_{x,x-w}] (\xi-\zeta)$, we get a power of $w$). From (\ref{eq:tildea0_def_c}), we then have 
$$
\partial_x^I \partial_{\xi}^{\alpha} \tilde{a}_0(x,\xi) 
= \int_{\mathbb{R}^N} \int_{\mathbb{R}^N} \sum_{|\sigma| \leq |I|} w^{\sigma} c_{\sigma}^{(I,\alpha)}(x,x-w,\xi-\zeta) e^{-2\pi i w \cdot \zeta} dw d\zeta.
$$
Writing $w^{\sigma} e^{-2\pi i w \cdot \zeta} = (-2\pi i)^{-|\sigma|} \partial_{\zeta}^{\sigma} e^{-2\pi i w \cdot \zeta}$ and integrating by parts, we see that
$$
\partial_x^I \partial_{\xi}^{\alpha} \tilde{a}_0(x,\xi) 
= \sum_{|\sigma| \leq |I|} (2\pi i)^{-|\sigma|} \int_{\mathbb{R}^N} \int_{\mathbb{R}^N} [\partial_{\xi}^{\sigma} c_{\sigma}^{(I,\alpha)}](x,x-w,\xi-\zeta) e^{-2\pi i w \cdot \zeta} dw d\zeta.
$$
But for every multiindex $\sigma$, $\partial_{\xi}^{\sigma} c_{\sigma}^{(I,\alpha)}(x,y,\xi)$ satisfies (\ref{eq:c_diffineq}) in place of $c(x,y,\xi)$, with $n$ replaced by $n-\|\alpha\|$ there. Thus by the same derivation of the inequality (\ref{eq:tildea0_bdd0}), we see that 
\begin{equation} \label{eq:tildea0_diff}
| \partial_x^I \partial_{\xi}^{\alpha} \tilde{a}_0(x,\xi)  | \lesssim (1+\|\xi\|)^{n-\|\alpha\|}.
\end{equation}
This establishes the full differential inequality (\ref{eq:a0diff}) for $\tilde{a}_0(x,\xi)$ in place of $a_0(x,\xi)$.

As a result, defining $\tilde{k}_0(x,u)$ to be the inverse Fourier transform of $\tilde{a}_0(x,\xi)$ in the $\xi$ variable, we have $\tilde{k}_0(x,u)$ satisfying the conditions of Theorem~\ref{thm:equiv_smoothing} or Theorem~\ref{thm:equiv_nonneg_order} depending on whether $-Q < n < 0$ or $n \geq 0$ (c.f. the facts cited in the proofs of Theorems~\ref{thm:equiv_smoothing} and \ref{thm:equiv_nonneg_order}). Now
$$
Tf(x) = \int_{\mathbb{R}^N} \tilde{k}_0(x,\Theta_0(x,y)) f(y) dy, 
$$
where $\Theta_0(x,y)$ is defined so that its $k$-th component is given by $$\sum_{j=1}^N \mathcal{B}^k_j(x) (x^j-y^j),$$ the coefficients $\tilde{B}^k_j(x)$ given by (\ref{eq:tildeB_def}). Thus by Theorem~\ref{thm:equiv_smoothing} or Theorem~\ref{thm:equiv_nonneg_order}, we have $T \in \Psi^n_{\mathcal{D}}(\mathbb{R}^N)$ as desired.
\end{proof}

\begin{proof}[Proof of Theorem~\ref{thm:compose}]
Suppose $T_j \in \Psi^{n_i}_{\mathcal{D}}(\mathbb{R}^N)$ for $j=1,2$. Then we may write
$$
T_j f(x) = \int_{\mathbb{R}^N} a_{j,0}(x,M_x \xi) \widehat{f}(\xi) e^{2\pi i x \cdot \xi} d\xi,
$$
where $M_x$ is a variable coefficient linear map associated to $\mathcal{D}$ as in (\ref{eq:Mxdef}) and $a_{j,0}(x,\xi)$ satisfies 
$$
|\partial_x^I \partial_{\xi}^{\alpha} a_{j,0}(x,\xi)| \lesssim_{I,\alpha} (1+\|\xi\|)^{n_j-\|\alpha\|}, \quad j = 1, 2.
$$ 
Then
$$
T_1 \circ T_2 f(x) = \int_{\mathbb{R}^N} a_0(x,M_x \xi) \widehat{f}(\xi) e^{2\pi i x \cdot \xi} d\xi,
$$
if
$$
a_0(x,M_x \xi) = \int_{\mathbb{R}^N} \int_{\mathbb{R}^N} a_{1,0}(x, M_x(\xi-\zeta)) a_{2,0}(x-w, M_{x-w} \xi) e^{-2\pi i w \cdot \zeta} dw d\zeta,
$$
i.e.
$$
a_0(x,\xi) = \int_{\mathbb{R}^N} a_{1,0}(x, \xi - L_x^{-t} \zeta) a_{2,0}(x-w, L_{x-w}^{-t} L_x^t \xi) e^{-2\pi i w \cdot \zeta} dw d\zeta,
$$
if we let $L_x = M_x^{-t}$. By making the change of variable $\zeta \mapsto L_x^t \zeta$ and $w \mapsto L_x^{-1} w$, we have
$$
a_0(x,\xi) = \int_{\mathbb{R}^N} a_{1,0}(x, \xi - \zeta) a_{2,0}(x-L_x^{-1} w, [I+w \tilde{L}_{x,x-w}] \xi) e^{-2\pi i w \cdot \zeta} dw d\zeta.
$$
if we write $L_{x-L_x^{-1} w}^{-t} L_x^t = I+w \tilde{L}_{x,x-w}$ (which is possible since $L_{x-L_x^{-1} w}^{-t} L_x^t = I$ when $w = 0$). It remains to show that $a_0(x,\xi)$ also satisfies (\ref{eq:a0diff}) with $n = n_1 + n_2$.
This follows from a similar calculation that would establish the desired bound (\ref{eq:tildea0_diff}) for $\tilde{a}_0(x,\xi)$ in the proof of Theorem~\ref{thm:equiv_general}.
\end{proof}

\begin{proof}[Proof of Theorem~\ref{thm:comm}]
Let $\eta \in C^{\infty}_c(\mathbb{R}^N)$. Then there exists a variable coefficient linear map $\Gamma_{x,u}$ on $\mathbb{R}^N$, whose coefficients varies smoothly with $x$ and $u$, so that
$$
\eta(x) - \eta(y) = \Gamma_{x,\Theta_0(x,y)} \Theta_0(x,y)
$$
for all $x, y \in \mathbb{R}^N$, where $\Theta_0(x,y)$ is as in Theorem~\ref{thm:equiv_smoothing}. Now for $T \in \Psi^n_{\mathcal{D}}(\mathbb{R}^N)$, let $k_0(x,u)$ be the integral kernel given by Theorem~\ref{thm:equiv_smoothing}. Then
$$
[\eta, T]f(x) = \int_{\mathbb{R}^N} \tilde{k}_0(x,\Theta_0(x,y)) f(y) dy.
$$
where $$\tilde{k}_0(x,u):= k_0(x,u) \Gamma_{x,u} u.$$ Since $\tilde{k}_0(x,u)$ satisfies (\ref{eq:kernelest}) with $n$ replaced by $n-1$, by invoking Theorem~\ref{thm:equiv_smoothing} again, we see that $[\eta,T] \in \Psi^{n-1}_{\mathcal{D}}(\mathbb{R}^N)$.
\end{proof}

\begin{proof}[Proof of Theorem~\ref{thm:Lpbdd}]
The key is to show that any $T \in \Psi^0_{\mathcal{D}}(\mathbb{R}^N)$ is a Calder\'{o}n-Zygmund operator with respect to a quasi-distance $d$ determined by the distribution $\mathcal{D}$. More precisely, suppose $X_1, \dots, X_N$ is a frame of the tangent bundle on $\mathbb{R}^N$, so that the first $N-1$ of them span $\mathcal{D}$, and let $X_i$ be given by some coefficients $(A_i^j)$ as in (\ref{eq:XiAij}) for $1 \leq i \leq N$. Let $(B^k_j)$ be the inverse of the matrix $(A^j_i)$ as before; these are precisely the coefficients of the dual frame $\theta^1, \dots, \theta^N$ to $X_1,\dots,X_N$, via 
$$
\theta^k = \sum_{j=1}^N B^k_j(x) dx^j \quad \text{for $1 \leq k \leq N$}.
$$
In particular, $\theta^N$ is the annihilator of $\mathcal{D}$, so the coefficients $B^N_j(x)$ are determined up to a multiple by $\mathcal{D}$ on $\mathbb{R}^N$, for $1 \leq j \leq N$. Now for $x, y \in \mathbb{R}^N$, let 
$$
d(x,y) \simeq |x-y| + \left| \sum_{j=1}^N B^N_j(x) (x^j-y^j) \right|^{1/2}.
$$
This defines a quasi-distance $d$ on $\mathbb{R}^N$, that depends only on the distribution $\mathcal{D}$, and that satisfies
$$
d(x,y) \simeq d(y,x) \quad \text{for any $x, y \in \mathbb{R}^N$}
$$
with
$$
d(x,z) \lesssim d(x,y) + d(y,z) \quad \text{for any $x,y,z \in \mathbb{R}^N$}. 
$$
For $x,y \in \mathbb{R}^N$ with $d(x,y) \leq 1$, the volume of the ball $\{z \in \mathbb{R}^N \colon d(x,z) \leq d(x,y)\}$ is comparable to $d(x,y)^Q$.
To prove the theorem, the key is to show that there exists a kernel $K(x,y)$, smooth away from the diagonal on $\mathbb{R}^N \times \mathbb{R}^N$, such that
$$
Tf(x) = \int_{\mathbb{R}^N} K(x,y) f(y) dy
$$
for any $f \in C^{\infty}_c(\mathbb{R}^N)$ and any $x$ not in the support of $f$; and such that
\begin{equation} \label{eq:K_ker_est_diff}
|X_{\gamma} K(x,y)| \lesssim d(x,y)^{-Q-\|\gamma\|}
\end{equation}
for any $\gamma = (\gamma_1, \dots, \gamma_k) \in \{1,\dots,N\}^k$, where $X_{\gamma}$ denotes $X_{\gamma_1} \dots X_{\gamma_k}$ and each $X_{\gamma_j}$ can act on either the $x$ or the $y$ variable, and $\|\gamma\|$ is the number of indices $\gamma_j$ such that $\gamma_j \ne N$, plus twice the number of indices $\gamma_j$ such that $\gamma_j = N$. Indeed, this follows from the case $n = 0$ of Theorem~\ref{thm:equiv_nonneg_order}, once we take $K(x,y) = k_0(x,\Theta_0(x,y))$. The kernel estimates (\ref{eq:K_ker_est_diff}) in turn imply that
$$
|K(x,y_1)-K(x,y_2)| \lesssim \frac{d(y_1,y_2)}{d(x,y_1)^{Q+1}} \quad \text{if $d(y_1,y_2) < \delta_0 d(x,y_1)$},
$$
and
$$
|K(x_1,y)-K(x_2,y)| \lesssim \frac{d(x_1,x_2)}{d(x_1,y)^{Q+1}} \quad \text{if $d(x_1,x_2) < \delta_0 d(x_1,y)$}
$$
where $\delta_0 \in (0,1)$ is a sufficiently small absolute constant.
Hence by Calder\'{o}n-Zygmund theory on spaces of homogeneous types (see e.g. \cite[Chapter I]{Ste93}), to prove that $T$ is bounded on $L^p(\mathbb{R}^N)$ for all $1 < p < \infty$, it suffices to show that $T$ is bounded on $L^2(\mathbb{R}^N)$; but the latter follows since any operator in $\Psi^0_{\mathcal{D}}(\mathbb{R}^N)$ is automatically in H\"ormander's symbol class of type $(1/2,1/2)$. This completes our proof of Theorem~\ref{thm:Lpbdd}.
\end{proof}

\begin{proof}[Proof of Theorem~\ref{thm:asym_sum}]
Suppose $T_j \in \Psi^{n-j}_{\mathcal{D}}(\mathbb{R}^N)$ for $j = 0,1,2,\dots$. Write 
$$
T_j f(x) = \int_{\mathbb{R}^N} a_{j,0}(x,M_x \xi) \widehat{f}(\xi) e^{-2\pi i x \cdot \xi} d\xi
$$
where $M_x$ is given by (\ref{eq:Mxdef}) and $a_{j,0}(x,\xi)$ satisfy
$$
|\partial_x^I \partial_{\xi}^{\alpha} a_{j,0}(x,\xi)| \leq C_{j,I,\alpha} (1+\|\xi\|)^{n-j-\|\alpha\|}
$$
for all $j \geq 0$ and multiindices $I$ and $\alpha$. Let 
$$
a_0(x,\xi) = \sum_{j=0}^{\infty} a_{j,0}(x,\xi) (1-\chi)\left( \frac{\|\xi\|}{N_j} \right)
$$
where $\chi \in C^{\infty}_c([-2,2])$ is identically 1 on $[-1,1]$, and $N_j$ is sufficiently large (say $$N_j = 1+\sup_{0 \leq k \leq j} \sup_{|I|,|\alpha| \leq j} C_{k,I,\alpha}$$ will do). Then one can check that $a_0(x,\xi)$ satisfies the differential inequalities (\ref{eq:a0diff}). Letting 
$$
Tf(x) = \int_{\mathbb{R}^N} a_0(x,M_x \xi) \widehat{f}(\xi) e^{-2\pi i x \cdot \xi} d\xi,
$$ 
one can check that $T-(T_0+\dots+T_k) \in \Psi^{n-k-1}_{\mathcal{D}}(\mathbb{R}^N)$ for all $k \in \mathbb{N}$. Hence we have Theorem~\ref{thm:asym_sum}.
\end{proof}

\section{Solution of $\hDelb$} \label{sect:hDelb}

We now return to our compact pseudohermitian manifold $(\hM,\hthe)$. We will use the class of non-isotropic pseudodifferential operators we introduced in the last section to solve the sublaplacian $\hDelb$ on $(\hM,\hthe)$; henceforth we will take $\hD$ to be the contact distribution $\hor$ on $\hM$, and all pseudodifferential operators on $\hM$ will be adapted to this $\hD$. As a result, we obtain a corollary, about how one commutes $\hnabla_b$ through a non-isotropic pseudodifferential operator. We will need this corollary in Section~\ref{sect:hboxbsol}.

First we will need to introduce a normal coordinate system near every point of $\hM$, following Folland-Stein \cite[Section 14]{FoSt} (see also Greiner-Stein \cite[Chapter 4]{GrSt}). Here we will only need the CR structure on $\hM$; the pseudohermitian structure on $\hM$ will come in only later. Let $U$ be a sufficiently small open subset of $\hM$. If $y \in U$ and $\hat{W}$ is a local section of $T\hM$ on $U$, then for $\varepsilon > 0$ sufficiently small and for $s \in [-\varepsilon,\varepsilon]$, one can define $y \exp(s \hat{W}) \in \hM$, so that the map $s \mapsto y \exp(s \hat{W})$ is an integral curve to $\hat{W}$ starting from $y$. This defines the exponential map $y \exp(\hat{W})$ for every $ y \in U$ and every $\hat{W} \in T_y \hM$ in a sufficiently small neighborhood of $0$. Now let $\hX, \hY$ be a local frame of the contact distribution $\hor$ on $U$, and let $\hT$ be a local section of $T\hM$ on $U$ so that $\hT$ is transverse to $\hor$. Then for $y \in U$ and for $u = (u_1,u_2,u_3)$ in a sufficiently small neighborhood of $0$ in $\mathbb{H}^1$, the map $$u \mapsto y \exp(2 u_1 \hX + 2 u_2 \hY + u_3 \hT)$$ is a diffeomorphism of a small neighborhood of $0$ in $\mathbb{H}^1$, onto a small neighborhood of $y$ in $\hM$. Indeed, by shrinking $U$ if necessary, we may assume that for every $x, y \in U$, there exists a unique $u \in \mathbb{H}^1$ such that $x = y \exp(2 u_1 \hX + 2 u_2 \hY + u_3 \hT)$. Such $u$ will be denoted $\Theta(x,y)$; for every $y \in U$, the map $x \in U \mapsto u = \Theta(x,y) \in \mathbb{H}^1$ provides a coordinate chart on $U$, and the important fact here is that in the $u$ coordinates, 
$$
\hX = X + O^1 \frac{\partial}{\partial u_1} + O^1 \frac{\partial}{\partial u_2} + O^2 \frac{\partial}{\partial u_3},
$$
$$
\hY = Y + O^1 \frac{\partial}{\partial u_1} + O^1 \frac{\partial}{\partial u_2} + O^2 \frac{\partial}{\partial u_3},
$$
$$
\hT = T + O^1 \frac{\partial}{\partial u_1} + O^1 \frac{\partial}{\partial u_2} + O^1 \frac{\partial}{\partial u_3},
$$
where
$$
X = \frac{1}{2} \left( \frac{\partial}{\partial u_1} + 2 u_2 \frac{\partial}{\partial u_3} \right), \qquad Y = \frac{1}{2} \left( \frac{\partial}{\partial u_2} - 2 u_1 \frac{\partial}{\partial u_3} \right), \qquad T = \frac{\partial}{\partial u_3},
$$
and henceforth $O^k$ represents a smooth function on $U \times U$ that is $\lesssim \|
\Theta(x,y)\|^k$; here $\|u\| := |u_1| + |u_2| + |u_3|^{1/2}$. In particular, each $O^k$ can be written as a function of the form $\eta(x,\Theta(x,y))$, where $\eta(x,u)$ satisfies $$|\partial_x^I \partial_u^{\alpha} \eta(x,u)| \lesssim_{I,\alpha} \|u\|^{k-\|\alpha\|}.$$ Note also that
$$
x = y \exp( \Theta(x,y) \cdot \hat{\Xi})
$$
where $\hat{\Xi} := (2\hX, 2\hY, \hT)$.
In particular, differentiating both sides with respect to $y$, and evaluating at $y = x$, we see that $\Theta(x,y)$ is compatible with our distribution $\hD$, in the sense described just before Theorem~\ref{thm:equiv_general}.

Now we specialize to the case when $\hM$ is compact and endowed with a pseudohermitian structure $\hthe$. We cover $\hM$ by finitely many sufficiently small open sets $U_i$'s, and on each $U_i$ we pick a local orthonormal frame $\hX, \hY$ to the contact distribution $\hor$, so that as we have seen in Section~\ref{sect:def}, the sublaplacian $\hDelb$ on $(\hM,\hthe)$ is given by $$\hDelb = \hX^* \hX + \hY^* \hY$$ on $U$. Now recall, from (\ref{eq:sublap_H1}), that $\Delta_b$ on $(\mathbb{H}^1,\theta)$ is given by
$$
\Delta_b = -(X^2+Y^2)
$$
if $\theta = du_3 + 2 u_1 du_2 - 2 u_2 du_1$. Hence on $U$, we have
$$
\hDelb = \Delta_b + O^1 \left( \frac{\partial}{\partial u'} + O^1 \frac{\partial}{\partial u_3} \right)^2 + O^0 \left( \frac{\partial}{\partial u'} + O^1 \frac{\partial}{\partial u_3} \right)
$$
where we wrote $u' = (u_1, u_2)$. But if 
$$K(u) = \frac{1}{2\pi (u_1^4 + u_2^4 + u_3^2)^{1/2}},$$ then 
$$
\Delta_b K(u) = \delta_0(u)
$$
(see e.g. Chapter XIII.2.3.4 of \cite{Ste93}). 
Hence if $K^{(0)}$ is the integral operator with integral kernel
$$
\tilde{\chi}(x) K(\Theta(x,y)) \chi(y)
$$
where $\chi, \tilde{\chi} \in C^{\infty}_c(U_i)$ and $\tilde{\chi} \equiv 1$ on an open set containing the support of $\chi$, then by Theorem~\ref{thm:equiv_general}, $K^{(0)}$ is a non-isotropic pseudodifferential operator in $\Psi^{-2}_{\hD}(\hM)$; furthermore, $\hDelb K^{(0)}$ is $\chi(y)$ times the identity operator, modulo an operator in $\Psi^{-1}_{\hD}(\hM)$. Patching together the operators from the different coordinate charts $U_i$, we obtain a non-isotropic pseudodifferential operator in $\Psi^{-2}_{\hD}(\hM)$, which we by abuse of notation still denote by $K^{(0)}$, such that
$$
\hDelb K^{(0)} = I - R
$$
where $R \in \Psi^{-1}_{\hD}(\hM)$; indeed, to be more precise, we would take a partition of unity $\{\chi_i\}$ subordinate to the open cover $\{U_i\}$ of $\hM$, and pick $\tilde{\chi}_i \in C^{\infty}_c(U_i)$ that is identically 1 on an open set containing the support of $\chi_i$, such that if $\Theta_i$ is the Folland-Stein coordinates on $U_i$, then the patched-up $K^{(0)}$ is the pseudodifferential operator with integral kernel $\sum_i \tilde{\chi}_i(x) K(\Theta_i(x,y)) \chi_i(y)$. By Theorem~\ref{thm:asym_sum}, there exists $E \in \Psi^0_{\hD}(\hM)$ such that 
$$
E - (I + R + R^2 + \dots + R^k) \in \Psi^{-(k+1)}_{\hD}(\hM)
$$ 
for all $k \geq 0$. Thus 
$$
\hDelb K^{(0)} E = I + R_{-\infty},
$$
where $R_{-\infty} \in \Psi^{-\infty}_{\hD}(\hM)$. The integral kernel of $R_{-\infty}$ is then smooth on $\hM \times \hM$; in other words, there exists $r_{-\infty}(x,y) \in C^{\infty}(\hM \times \hM)$ such that
$$
R_{-\infty} f(x) = \int_{\hM} f(y) r_{-\infty}(x,y) \hthe \wedge d\hthe(y).
$$
By microlocalization, one may find a function $e_{-\infty}(x,y) \in C^{\infty}(\hM \times \hM)$, such that
$$
(\hDelb)_x e_{-\infty}(x,y) = r_{-\infty}(x,y) - c(y),
$$
where $c(y)$ is given by $$c(y) := \frac{1}{\text{Vol}(\hM)} \int_{\hM} r_{-\infty}(x,y) \hthe \wedge d\hthe(x),$$ with $\text{Vol}(\hM) = \int_{\hM} \hthe \wedge d\hthe(x).$ Writing $E_{-\infty} \in \Psi^{-\infty}_{\hD}(\hM)$ for the pseudodifferential operator with integral kernel $e_{-\infty}(x,y)$, we have
\begin{equation} \label{eq:hDelbsol_hatC}
\hDelb [K^{(0)} E + E_{-\infty}] + \hat{C} = I
\end{equation}
where $\hat{C} f(x) = \int_{\hM} c(y) f(y) \hthe \wedge d\hthe(y)$. This gives a decomposition of any $L^2$ function on $\hM$ into a sum of two parts, one of which is in the image of $\hDelb$, the other of which is in the kernel of $\hDelb$ (= the set of all constant functions on $\hM$). By uniqueness of such a decomposition, we see that 
\begin{equation} \label{eq:hatC_simplified} 
\hat{C} f(x) = \frac{1}{\text{Vol}(\hM)} \int_{\hM} f(y) \hthe \wedge d\hthe(y),
\end{equation} for any $f \in L^2(\hM)$, which gives $c(y) = \frac{1}{\text{Vol}(\hM)}$ is constant independent of $y$. (Alternatively, just apply (\ref{eq:hDelbsol_hatC}) to an arbitrary $f \in L^2(\hM)$ and average both sides over $\hM$; then since $\hat{C} f(x)$ is a constant function on $\hM$, we obtain again (\ref{eq:hatC_simplified}), which gives us the same conclusion about $c(y)$.)
Letting $$\hat{K} = (I-\hat{C}) [K^{(0)} E + E_{-\infty}] \in \Psi^{-2}_{\hD}(\hM),$$ we see that
\begin{equation} \label{eq:sol_hDelb}
\hDelb \hat{K} + \hat{C} = I,
\end{equation}
and that the image of $\hat{K}$ is orthogonal to the kernel of $\hDelb$. Thus $\hat{K}$ is the canonical solution to $\hDelb$, and this completes our solution to $\hDelb$ on $(\hM,\hthe)$.

\begin{cor} \label{cor:nablab_comm}
Suppose $k \in \mathbb{R}$, $T_k \in \Psi^k_{\hD}(\hM)$. Let $\hat{W}$ be a global section of $\hor$ on $\hM$ such that $g_{\hthe}(\hat{W},\hat{W}) \leq 1$ on $\hM$. Then for each sufficiently small open set $U$ on $\hM$, if $\hX, \hY$ is a local orthonormal frame to $\hor$ on $U$, then there exist $T_k', T_k'' \in \Psi^k_{\hD}(\hM)$ and $T_{-\infty} \in \Psi^{-\infty}_{\hD}(\hM)$, such that 
$$\hat{W} T_k = T_k' \hX + T_k'' \hY + T_{-\infty}$$ when acting on functions with compact support in $U$. We usually abbreviate this by saying that $$\hnabla_b T_k = T_k' \hnabla_b + T_{-\infty};$$ furthermore, $T_{-\infty}$ can be taken to be zero if $T_k 1 = 0$.
\end{cor}

\begin{proof}
Taking the adjoint of (\ref{eq:sol_hDelb}), we obtain $$I = \hat{K} \hDelb + \hat{C}.$$ Applying $\hat{W} T_k$ on the left of both sides, we have 
$$
\hat{W} T_k = (\hat{W} \hat{K} \hX^*) \hX + (\hat{W} \hat{K} \hY^*) \hY + \hat{W} T_k \hat{C}.
$$  
It remains to observe that $T_k' := \hat{W} \hat{K} \hX^*$ and $T_k'' := \hat{W} \hat{K} \hY^*$ are both in $\Psi^k_{\hD}(\hM)$ by Theorem~\ref{thm:compose}, and that $T_{-\infty}:=\hat{W} T_k \hat{C} \in \Psi^{-\infty}_{\hD}(\hM)$, by Proposition~\ref{prop:pseudolocal} and the characterization of $\Psi^{-\infty}_{\hD}(\hM)$ immediately after Theorem~\ref{thm:asym_sum}. We also note that if $T_k 1 = 0$, then $T_k \hat{C} = 0$, and hence $T_{-\infty} = 0$.
\end{proof}

\section{Solution of $\hboxb$} \label{sect:hboxbsol}

Let $\hM$ be a closed 3-dimensional strongly pseudoconvex CR manifold that is embeddable in some $\mathbb{C}^K$. The goal of this section is to solve a certain tangential Kohn Laplacian on $\hM$. If $\hM$ is the boundary of a strongly pseudoconvex domain in $\mathbb{C}^2$, then one may follow the approach in e.g. Nagel and Stein \cite{NaSt}; on the other hand, the following approach works whenever $\hM$ is embeddable in some $\mathbb{C}^K$ (where $K$ can be bigger than 2; equivalently, it works whenever $\hdbarb$ has closed range in $L^2(\hM)$ (c.f. Theorem~\ref{thm:embed_closed_range})). The approach presented here is a variant of the one in Beals and Greiner \cite{BG88}.

Recall that in Section~\ref{sect:def}, when $\hM$ is equipped with a pseudohermitian structure $\hthe$, we defined the tangential Kohn Laplacian on $(\hM,\hthe)$, so that its action on $C^{\infty}$ functions on $\hM$ is given by $\vartheta_b \dbarb$, where $\vartheta_b$ is the formal adjoint of $\dbarb$ that satisfies
$$
\int_{\hM} \langle \dbarb u, \alpha \rangle_{\hthe} \, \hthe \wedge d\hthe = \int_{\hM} u \, \overline{\vartheta_b \alpha} \, \hthe \wedge d\hthe
$$
for all smooth functions $u$ and smooth $(0,1)$-forms $\alpha$ on $\hM$. Note that $\hthe \wedge d\hthe$ is a particular smooth measure on $\hM$. In general, we could have defined $\vartheta_b$ by taking adjoints with respect to a different smooth measure on $\hM$, in which case we would get a different tangential Kohn Laplacian; this is often useful in practice, as in the proof of our main Theorem~\ref{thm:goal} in the general case. Thus in this section, we solve all tangential Kohn Laplacians on $\hM$ that arise this way.

More precisely, let $\hm_0$ and $\hm_1$ be two smooth measures on $\hM$, and $\langle \cdot, \cdot \rangle$ be any smoothly varying pointwise hermitian inner product on $T^{0,1} \hM$. We use $\hm_0$ to define an $L^2$ space of functions (denoted $L^2(\hm_0)$):
$$
(u, v) = \int_{\hM} u \, \overline{v} \, \hm_0,
$$
and $\hm_1$ to define an $L^2$ space of $(0,1)$ forms (denoted $L^2_{(0,1)}(\hm_1)$):
$$
(\alpha,\beta) = \int_{\hM} \langle \alpha, \beta \rangle_{\hthe} \, \hm_1.
$$
Let $\vartheta_b$ be the formal adjoint of $\dbarb$, that satisfies
$$
(\dbarb u, \alpha) = (u, \vartheta_b \alpha)
$$
for all smooth functions $u$ and smooth $(0,1)$-forms $\alpha$ on $\hM$. Let $\hboxb = \vartheta_b \dbarb$ on smooth functions on $\hM$. Under the assumption that $\hM$ is embeddable in some $\mathbb{C}^K$, we will construct, in our algebra of non-isotropic pseudodifferential operators on $\hM$, the relative solution operator and the Szeg\H{o} projection to $\hboxb$. Indeed, we will show that $\hN \in \Psi^{-2}_{\hD}(\hM)$, and $\hS \in \Psi^0_{\hD}(\hM)$, where as before, $\hD$ is the contact distribution $\hor$ on $\hM$.

First we need an $L^2$ theory. Let 
$$
\hdbarb \colon L^2(\hm_0) \to L^2_{(0,1)}(\hm_1)
$$ 
be the Hilbert space closure of $\hdbarb$ acting on functions in $C^{\infty}(\hM)$. Let 
$$
\hdbarb^* \colon L^2_{(0,1)}(\hm_1) \to L^2(\hm_0)
$$ 
be its Hilbert space adjoint. Define also a densely defined operator 
$$
\hboxb \colon L^2(\hm_0) \to L^2(\hm_0),
$$ 
with domain given by
\begin{align*}
\text{Dom}(\hboxb) := \{ & u \in \text{Dom}[\hdbarb \colon L^2(\hm_0) \to L^2_{(0,1)}(\hm_1)],  \\ 
& \quad \hdbarb u \in \text{Dom}[\hdbarb^* \colon L^2_{(0,1)}(\hm_1) \to L^2(\hm_0)] \}
\end{align*}
Then from the embeddability of $\hM$, one can prove, using Kohn's microlocalization procedure, the following standard fact (c.f. Theorem~\ref{thm:embed_closed_range}):

\begin{prop} \label{prop:Fact 1} The operator $\hdbarb \colon L^2(\hm_0) \to L^2_{(0,1)}(\hm_1)$ has closed range; hence so does $\hdbarb^* \colon L^2_{(0,1)}(\hm_1) \to L^2(\hm_0)$ and $\hboxb \colon L^2(\hm_0) \to L^2(\hm_0)$.
\end{prop}

This allows us to define the relative solution operator
\begin{equation} \label{eq:hNdef1}
\hN \colon L^2(\hm_0) \to \text{Dom}(\hboxb) \subseteq L^2(\hm_0)
\end{equation}
and the orthogonal projection onto the kernel of $\hboxb$ (which is also the kernel of $\hdbarb$) in $L^2(\hm_0)$:
$$
\hS \colon L^2(\hm_0) \to \text{Dom}(\hboxb) \subseteq L^2(\hm_0)
$$
Then 
\begin{equation} \label{eq:hNdef2}
\hboxb \hN + \hS = I \quad \text{on $L^2(\hm_0)$}.
\end{equation}

We now turn to the $L^p$ theory for $\hN$, $\hS$ and $\hboxb$. Recall that $\hD$ is the contact distribution on $\hM$, and in Section~\ref{sect:pdo} we constructed non-isotropic pseudodifferential operators adapted to $\hD$. We then have the following proposition:
 
\begin{prop} \label{prop:pdo}
$\hS$ and $\hN$ are non-isotropic pseudodifferential operators on $\hM$, of order 0 and $-2$ respectively; in other words, $\hS \in \Psi^0_{\hD}(\hM)$ and $\hN \in \Psi^{-2}_{\hD}(\hM)$.
\end{prop}

To prove Proposition~\ref{prop:pdo}, first recall the Folland-Stein normal coordinates that we defined in Section~\ref{sect:hDelb}. Indeed, for any sufficiently small open set $U \subset \hM$, we will pick a local section $\hZbar$ of $T^{0,1}(\hM)$ on $U$, such that $\langle \hZbar, \hZbar \rangle = 1$ on $U$, and write $\hT$ for $i[\hZ, \hZbar]$ on $U$; then we have a map $\Theta \colon U \times U \to \mathbb{H}^1$, such that 
$$
x = y \exp( \Theta(x,y) \cdot \Xi )
$$
for all $x, y \in U$, where $\Xi := (2 \hX, 2 \hY, \hT)$, with $\hX, \hY$ determined by $\hZbar = \frac{\hX + i \hY}{\sqrt{2}}$. Differentiating both sides with respect to $y$, and evaluating at $y = x$, we see that $\Theta(x,y)$ is compatible with our distribution $\hD$, in the sense described just before Theorem~\ref{thm:equiv_general}. Now write $$ u = (w,s) = \Theta(x,y) \in \mathbb{H}^1 = \mathbb{C} \times \mathbb{R}.$$ Then for any $y \in U$, in the $(w,s)$ coordinates, we have
$$
\hZbar = \Zbar + O^1 \frac{\partial}{\partial w} + O^1 \frac{\partial}{\partial \wbar} + O^2 \frac{\partial}{\partial s},
$$ 
and
$$
\hT = T + O^1 \frac{\partial}{\partial w} + O^1 \frac{\partial}{\partial \wbar} + O^1 \frac{\partial}{\partial s}.
$$
where 
$$
\Zbar := \frac{1}{\sqrt{2}} \left( \frac{\partial}{\partial \wbar} - i w \frac{\partial}{\partial s} \right), \qquad T := \frac{\partial}{\partial s} 
$$
are the corresponding vector fields on the Heisenberg group $(\mathbb{H}^1,\theta)$, with $\theta = ds - i(\wbar dw - w d\overline{w})$ (see Section~\ref{sect:pdo} for the definition of $O^k$).
Hence on $U$, the formal adjoint of $\hZbar \colon L^2(\hm_0) \to L^2_{(0,1)}(\hm_1)$ satisfies
$$
\hZbar^* = \frac{1}{\sqrt{2}} \left( \frac{\partial}{\partial w} + i \wbar \frac{\partial}{\partial s} \right) + O^1 \frac{\partial}{\partial w} + O^1 \frac{\partial}{\partial \wbar} + O^2 \frac{\partial}{\partial s} + O^0.
$$
This shows
$$
\hboxb = \boxb + O^1 \left( \frac{\partial}{\partial u'} + O^1 \frac{\partial}{\partial s} \right)^2 + O^0 \left( \frac{\partial}{\partial u'} + O^1 \frac{\partial}{\partial s} \right)
$$
where
$$
\boxb := -Z \Zbar
$$
is the tangential Kohn Laplacian on $(\mathbb{H}^1,\theta)$, and $\frac{\partial}{\partial u'}$ represents linear combinations of $\frac{\partial}{\partial w}$ and $\frac{\partial}{\partial \wbar}$. Now let $N$ and $\Pi$ be the convolution kernels of the relative solution operator, and the Szeg\H{o} projection, to $\boxb$ on $(\mathbb{H}^1,\theta)$. In other words,
$$
N(w,s) = \frac{1}{\pi^2} \log \left( \frac{|w|^2 - i s}{|w|^2 + i s} \right) \frac{1}{|w|^2 - i s}
$$ 
($\log$ being the principal branch of the logarithm) and
$$
\Pi(w,s) = -\frac{i}{\pi^2} \frac{\partial}{\partial s} \left( \frac{1}{|w|^2-is} \right) = \frac{1}{\pi^2 (|w|^2 - i s)^2}.
$$
Hence if $N^{(0)}$ and $\Pi^{(0)}$ are the non-isotropic pseudodifferential operators with kernels $\chi(x) N(\Theta(x,y)) \chi(y)$ and $\chi(x) \Pi(\Theta(x,y)) \chi(y)$ respectively, where $\chi \in C^{\infty}_c(U)$ is supported on our sufficiently small open set $U$, then by Theorem~\ref{thm:equiv_general}, $N^{(0)} \in \Psi^{-2}_{\hD}(\hM)$ and $\Pi^{(0)} \in \Psi^{0}_{\hD}(\hM)$; furthermore, from the expansion of $\hboxb$ as above, we see that
$\hboxb N^{(0)} + \Pi^{(0)}$ is the identity operator, modulo an operator in $\Psi^{-1}_{\hD}(\hM)$. Now let's cover $\hM$ by finitely many of these sufficiently small open sets $U$'s, and patch the operators $N^{(0)}$ and $\Pi^{(0)}$ from these different open sets together; by abuse of notation, let's still call the resulting operators $N^{(0)}$ and $\Pi^{(0)}$. More precisely, we would take a partition of unity $\{\chi_i\}$ subordinate to the open cover $\{U_i\}$ of $\hM$, and pick $\tilde{\chi}_i \in C^{\infty}_c(U_i)$ that is identically 1 on an open set containing the support of $\chi_i$, such that if $\Theta_i$ is the Folland-Stein coordinates on $U_i$, then the patched-up $N^{(0)}$ is the pseudodifferential operator with integral kernel $\sum_i \tilde{\chi}_i(x) N(\Theta_i(x,y)) \chi_i(y)$, and similarly for the patched up $\Pi^{(0)}$. We then obtain $N^{(0)} \in \Psi^{-2}_{\hD}(\hM)$ and $\Pi^{(0)} \in \Psi^{0}_{\hD}(\hM)$, such that
$$
\hboxb N^{(0)} + \Pi^{(0)} = I - R^{(0)}
$$
where $R^{(0)} \in \Psi^{-1}_{\hD}(\hM)$; furthermore,
$$
\hboxb \Pi^{(0)} \in \Psi^0_{\hD}(\hM).
$$
Unfortunately this is not good enough for iterations, if we want to solve for $\hN$ and $\hS$ simultaneously. The key now is to construct a perturbation of $\Pi^{(0)}$, so that it is exactly annihilated by $\hboxb$. We proceed as follows. 

A polynomial of $w, \wbar$ and $s$ is said to be homogeneous of degree $k$, if it is a linear combination of terms of the form $\wbar^{\alpha} w^{\beta} s^{\gamma}$ where $\alpha + \beta + 2 \gamma = k$. 

\begin{lem} \label{lem:solve_poly}
If $p(w,\wbar,s)$ is a homogeneous polynomial of degree $k \geq 2$, then there exists a homogeneous polynomial $q(w,\wbar,s)$ of degree $k + 1$ such that 
$$
\Zbar q(w,\wbar,s) = p(w,\wbar,s)
$$
with 
$$
|\re q(w,\wbar,s)| \lesssim |w|^2 (|w| + |s|).
$$
Furthermore, the coefficients of $q$ are linear combinations of the coefficients of $p$.
\end{lem}

\begin{proof}
Let $p(w,\wbar,s) = \wbar^{\alpha} w^{\beta} s^{\gamma}$ where $\alpha + \beta + 2 \gamma = k$. If $\alpha + \beta > 0$, then we set
\begin{align*}
q(w,\wbar,s) 
&= \sqrt{2} \left( \frac{\wbar^{\alpha+1} w^{\beta} s^{\gamma}}{\alpha+1} + \frac{i\gamma \wbar^{\alpha+2} w^{\beta+1} s^{\gamma-1}}{(\alpha+1)(\alpha+2)} + \dots + \frac{i^{\gamma} \gamma! \wbar^{\alpha+\gamma+1} w^{\beta + \gamma} s^0}{(\alpha+1) \dots (\alpha + \gamma + 1)} \right) \\
&= \sqrt{2} \sum_{\ell=0}^{\gamma} \frac{i^{\ell} \gamma! \alpha! \wbar^{\alpha+\ell+1} w^{\beta + \ell} s^{\gamma-\ell}}{(\gamma-\ell)! (\alpha+\ell+1)!};
\end{align*}
then by telescoping, $\Zbar q(w,\wbar,s) = \wbar^{\alpha} w^{\beta} s^{\gamma}$, and since $\alpha + \beta \geq 1$ and $k \geq 2$, we have
$$
|q(w,\wbar,s)| \lesssim |w|^2 \sum_{\ell = 0}^{\gamma} |w|^{2\ell+\alpha+\beta-1} s^{\gamma - \ell} \lesssim |w|^2 (|w|+|s|),
$$
yielding the desired bound for $|\re q(w,\wbar,s)|$. This does not work when $\alpha = \beta = 0$; in that case, we set
\begin{align*}
q(w,\wbar,s) 
&= \sqrt{2} \left( \frac{\wbar^1 w^0 s^{\gamma}}{1!} + \frac{i\gamma \wbar^2 w^1 s^{\gamma-1}}{2!} + \dots + \frac{i^{\gamma} \gamma! \wbar^{\gamma+1} w^{\gamma} s^0}{(\gamma + 1)!} \right) \\
&\quad - \sqrt{2} \left( \frac{\wbar^0 w^1 s^{\gamma}}{0!} + \frac{i\gamma \wbar^1 w^2 s^{\gamma-1}}{1!} + \dots + \frac{i^{\gamma} \gamma! \wbar^{\gamma} w^{\gamma+1} s^0}{\gamma!} \right) \\
&= \sqrt{2} \sum_{\ell=0}^{\gamma} \left( \frac{i^{\ell} \gamma!  \wbar^{\ell+1} w^{\ell} s^{\gamma-\ell}}{(\gamma-\ell)! (\ell+1)!} + \frac{i^{\ell} \gamma! \wbar^{\ell} w^{\ell+1} s^{\gamma-\ell}}{(\gamma-\ell)! \ell!} \right);
\end{align*}
we then get, by telescoping, that $\Zbar q(w,\wbar,s) = s^{\gamma}$, and since the term corresponding to $\ell = 0$ in the definition of $q(w,\wbar,s)$ is purely imaginary, we have
$$
|\re  q(w,\wbar,s)| \lesssim \sum_{\ell=1}^{\gamma} |w|^{2\ell+1} s^{\gamma-\ell} \lesssim |w|^3,
$$ 
which is $\lesssim |w|^2 (|w| + |s|)$ as desired.
\end{proof}

Next, let $U$ be a sufficiently small open set of $\hM$ as above, so that we have a local section $\hZbar$ of $T^{0,1}(\hM)$ on $U$ with $\langle \hZbar, \hZbar \rangle = 1$ on $U$, and so that $\Theta(x,y)$ is defined for all $x,y \in U$. Then we have the following lemma:

\begin{lem} \label{lem:psi_inf_order}
There exists a $C^{\infty}$ function $\psi_0(x,u)$, defined for $x \in U$ and $u$ in a neighborhood of $0$ on $\mathbb{H}^1$, with $\re \psi_0(x,u) \geq 0$ for all such $x$ and $u$, such that for every $y \in U$, we have 
$$
\psi_0(x,\Theta(x,y)) = \pi (|w|^2 - is) + O^3, \quad (w,s) = \Theta(x,y),
$$
and that $\hZbar_x \psi_0(x,\Theta(x,y))$ vanishes to infinite order at $y$.
\end{lem}

Henceforth $\hZbar_x$ refers to the derivative $\hZbar$ acting on the $x$ variable (when $y$ is fixed).

\begin{proof}
Let $q_2(x,u) = \pi(|w|^2 - is)$ where $u = (w,s)$. Then $\hZbar_x q_2(x,\Theta(x,y)) \in O^2$, hence by Lemma~\ref{lem:solve_poly}, there exists a homogeneous polynomial $q_3(x,u)$ of degree 3 in $u$, varying smoothly with $x \in U$, such that $$|\re q_3(x,u)| \leq C_3 |w|^2 (|w|+|s|) \quad \text{on $U \times \{\|u\| \leq 1\}$},$$ and $$\hZbar_x (q_2 + q_3)(x,\Theta(x,y)) \in O^4.$$ Iterating, for any $k \geq 3$, there exists a homogeneous polynomial $q_k(x,u)$ of degree $k$ in $u$, varying smoothly with $x \in U$, such that $$|\re q_k(x,u)| \leq C_k |w|^2 (|w|+|s|) \quad \text{on $U \times \{\|u\| \leq 1\}$},$$ and $$\hZbar_x \sum_{\ell=2}^k q_{\ell}(x,\Theta(x,y)) \in O^{k+1}.$$ If $\chi(u) \in C^{\infty}_c[-2,2]$ is identically 1 on $[-1,1]$ and $\{\varepsilon_{\ell}\}$ is a positive sequence that tends rapidly to 0 ($\varepsilon_{\ell} \leq \min \left\{(10 C_{\ell})^{-1}, \left( \sup_{k \leq \ell} \|q_k(x,u)\|_{C^{\ell}(U \times \{\|u\| \leq 1\})} \right)^{-1} \right\}$ will more than suffice), then letting
$$
\psi_0(x,u) := \sum_{\ell=2}^{\infty} q_{\ell}(x,u) \chi(\varepsilon_{\ell}^{-1} \|u\|),
$$
we have $\psi_0(x,u) \in C^{\infty}(U \times \{\|u\| \leq 1\})$, $$|\re \psi_0(x,u)| \geq 
\pi|w|^2 - \sum_{\ell=3}^k |\re q_{\ell}(x,u)| \geq |w|^2 \geq 0,$$ $$
\psi_0(x,\Theta(x,y)) = \pi(|w|^2 - is) + O^3, \quad (w,s) = \Theta(x,y),
$$
and for every $k \in \mathbb{N}$, we have $\hZbar_x \psi_0(x,\Theta(x,y)) \in O^{k+1}$. This proves the lemma.
\end{proof}

We remark that in the above lemma, if we are willing to shrink the open set $U$, then instead of requiring $\hZbar_x \psi_0(x,\Theta(x,y))$ to vanish to infinite order at every $y \in U$, we could have arranged so that it is zero for every $x, y \in U$. Indeed, if $U'$ is a relatively compact open subset of $U$, and $\tilde{\chi} \in C^{\infty}_c(U)$ is identically 1 on an open set containing the closure of $U'$, then for every $y \in U'$, the function $$(\hboxb)_x \frac{\tilde{\chi}(x)}{\psi_0(x,\Theta(x,y))}$$ can be extended as a smooth function of $x$ on $\hM$, such that $$\hS_x (\hboxb)_x \frac{\tilde{\chi}(x)}{\psi_0(x,\Theta(x,y))} = 0$$ on $\hM$; the latter follows since $$(\hboxb)_x \frac{\tilde{\chi}(x)}{\psi_0(x,\Theta(x,y)) + \varepsilon} \to (\hboxb)_x \frac{\tilde{\chi}(x)}{\psi_0(x,\Theta(x,y))}$$ in $L^2(\hm_0)$, and that $$\hS_x (\hboxb)_x \frac{\tilde{\chi}(x)}{\psi_0(x,\Theta(x,y)) + \varepsilon} = 0$$ on $\hM$ for all $\varepsilon > 0$. (Here we used that $\re \psi_0(x,\Theta(x,y)) \geq 0$ to guarantee that $\frac{\tilde{\chi}(x)}{\psi_0(x,\Theta(x,y)) + \varepsilon}$ is smooth on $\hM$, and that the aforementioned $L^2$ convergence holds.) As a result, by the $C^{\infty}$ regularity of $\hboxb$ (which can be obtained using Kohn's microlocalization technique), there exists a smooth function $e(x,y)$, such that 
$$(\hboxb)_x \frac{\tilde{\chi}(x)}{\psi_0(x,\Theta(x,y))} = (\hboxb)_x e(x,y)$$
for every $x \in \hM$ and $y \in U'$. We may arrange so that $\re e(x,y) < 0$ by adding a constant, and we may write $e(x,y)$ as $e_0(x,\Theta(x,y))$ if $x \in U$. Then setting 
$$
\psi_{00}(x,u) = \frac{1}{\frac{\tilde{\chi}(x)}{\psi_0(x,u)} - e_0(x,u)},
$$
one can check that $\psi_{00}(x,u)$ is smooth as $x$ varies over $U'$ and $u$ varies over a sufficiently small neighborhood of $0$ in $\mathbb{H}^1$, $\re \psi_{00}(x,u) \geq 0$ for all such $x$ and $u$,
$$
\psi_{00}(x,\Theta(x,y)) = \pi(|w|^2 - i s)  + O^3, \quad (w,s) = \Theta(x,y)
$$ 
whenever $x,y \in U'$, and in addition we have
$$
\hZbar_x \psi_{00}(x,\Theta(x,y)) = 0
$$
for every $x, y \in U'$.

Now cover $\hM$ by finitely many sufficiently small open sets $U_i$'s, and let $\chi_i, \tilde{\chi}_i \in C^{\infty}_c(U_i)$ be such that $\tilde{\chi}_i = 1$ on the support of $\chi_i$. Let $\Theta_i$ be the Folland-Stein normal coordinates on $U_i$, and $\psi_{0,i}$ be the function constructed in Lemma~\ref{lem:psi_inf_order} when $U = U_i$. Let $\Pi^{(00)}$ be the non-isotropic pseudodifferential operator in $\Psi^0_{\hD}(\hM)$ with integral kernel 
$$
\sum_i \tilde{\chi}_i(x) \frac{1}{ \psi_{0,i}(x,\Theta_i(x,y))^2 } \chi_i(y).
$$
Then applying Theorem~\ref{thm:equiv_general}, we have
$$
\Pi^{(0)} - \Pi^{(00)} \in \Psi^{-1}_{\hD}(\hM).
$$
Furthermore,
$$
\hboxb \Pi^{(00)} \in \Psi^{-\infty}_{\hD}(\hM);
$$
indeed, the integral kernel of $\hboxb \Pi^{(00)}$ is $(\hboxb)_x \sum_i \tilde{\chi}_i(x) \frac{1}{ \psi_{0,i}(x,\Theta_i(x,y))^2 } \chi_i(y)$, which is $C^{\infty}$ in $x$ for every $y \in \hM$. Now observe that
$$
\hS_x (\hboxb)_x \sum_i \tilde{\chi}_i(x) \frac{1}{ \psi_{0,i}(x,\Theta_i(x,y))^2 } \chi_i(y)
= 0
$$ 
for every $y \in \hM$; this holds because $$(\hboxb)_x \sum_i \tilde{\chi}_i(x) \frac{1}{ (\psi_{0,i}(x,\Theta_i(x,y)) + \varepsilon)^2 } \chi_i(y) \to (\hboxb)_x \sum_i \tilde{\chi}_i(x) \frac{1}{ \psi_{0,i}(x,\Theta_i(x,y))^2 } \chi_i(y)$$ in $L^2_x$ as $\varepsilon \to 0^+$, and $\sum_i \tilde{\chi}_i(x) \frac{1}{ (\psi_{0,i}(x,\Theta_i(x,y)) + \varepsilon)^2 } \chi_i(y)$ is smooth in $x$, the denominators having positive real parts everywhere.
Using the $C^{\infty}$ regularity of $\hboxb$, one can find a $C^{\infty}$ function $e(x,y)$ on $\hM \times \hM$, such that
$$
(\hboxb)_x \sum_i \tilde{\chi}_i(x) \frac{1}{ \psi_{0,i}(x,\Theta_i(x,y))^2 } \chi_i(y)
= (\hboxb)_x e(x,y)
$$ 
for every $x, y \in \hM$; then adjusting the kernel of $\Pi^{(00)}$ by $e(x,y)$, we get a pseudodifferential operator $\Pi^{(000)} \in \Psi^0_{\hD}(\hM)$ such that
$$
\hboxb \Pi^{(000)} = 0,
$$
and
$$
\Pi^{(00)} - \Pi^{(000)} \in \Psi^{-\infty}_{\hD}(\hM).
$$
It follows that
$$
\hboxb N^{(0)} + \Pi^{(000)} = I - R
$$
where $R \in \Psi^{-1}_{\hD}(\hM)$. Now using Theorem~\ref{thm:asym_sum}, there exists a non-isotropic pseudodifferential operator $E$ in $\Psi^0_{\hD}(\hM)$, such that 
$$
E - (I + R + R^2 + \dots + R^k) \in \Psi^{-(k+1)}_{\hD}(\hM)
$$
for every $k \in \mathbb{N}$. Then
$$
\hboxb N^{(0)} E + \Pi^{(000)} E = I + R_{-\infty}
$$
where $R_{-\infty} \in \Psi^{-\infty}_{\hD}(\hM)$. Thus by $C^{\infty}$ regularity of $\hboxb$ again, one can adjust the kernels of $N^{(0)} E$ and $\Pi^{(000)} E$ by a $C^{\infty}$ function, and obtain non-isotropic pseudodifferential operators $N_0$ and $\Pi_0$, in $\Psi^{-2}_{\hD}(\hM)$ and $\Psi^0_{\hD}(\hM)$ respectively, such that
$$
\begin{cases}
\hboxb N_0 + \Pi_0 = I \\
\hboxb \Pi_0 = 0.
\end{cases}
$$
It remains to observe that the true Szeg\H{o} projection $\hS$ on $\hM$, and the true relative solution operator $\hN$ of $\hboxb$ on $\hM$, are given precisely by $\Pi_0$ and $(I-\Pi_0) N_0$ respectively. This finishes the proof of Proposition~\ref{prop:pdo}.

One importance of Proposition~\ref{prop:pdo} is the following. In view of Theorem~\ref{thm:Lpbdd}, we see that $\hN$ and $\hS$ admits the following (continuous) extensions:
\begin{equation} \label{eq:hSLp}
\hS \colon L^p(\hm_0) \to L^p(\hm_0) \quad \text{for all $p \in (1,\infty)$},
\end{equation}
\begin{equation} \label{eq:hN_L43}
\hN \colon L^{4/3}(\hm_0) \to L^4(\hm_0),
\end{equation}
\begin{equation} \label{eq:hN_D_L43}
\hnabla_b \hN \colon L^{4/3}(\hm_0) \to L^2(\hm_0).
\end{equation}
Indeed, by Proposition~\ref{prop:pdo}, $\hnabla_b^2 \hN \in \Psi^0_{\hD}(\hM)$, so $\hN$ maps $L^{4/3}(\hm_0)$ into the non-isotropic Sobolev space $NL^{2,4/3}(\hm_0)$ (the space of functions whose $\hnabla_b^2$ is in $L^{4/3}(\hm_0)$), which by Sobolev embedding is embedded into $NL^{1,2}(\hm_0)$ and $L^4(\hm_0)$, yielding (\ref{eq:hN_L43}) and (\ref{eq:hN_D_L43}).
Also, by Theorem~\ref{thm:compose}, $\hnabla_b \hS \hN \in \Psi^{-1}_{\hD}(\hM)$, so by Theorem~\ref{thm:Lpbdd}, we have
\begin{equation} \label{eq:hShN_D_L43}
\hnabla_b \hS \hN \colon L^{4/3}(\hm_0) \to L^2(\hm_0).
\end{equation}
Finally, $\hS$ and $\hN$ are pseudolocal on $\hM$, i.e. if $f \in C^{\infty}(U)$ for some open set $U \subseteq \hM$, then $\hS f$ and $\hN f$ are also in $C^{\infty}(U)$.

We also need the following key proposition:

\begin{prop} \label{prop:regularity_in_E}
For any $0 < \delta < 1$, $\hS$ maps $\E{-1+\delta}$ into $\E{-1+\delta}$, and $\hN$ maps $\E{-3+\delta}$ into $\E{-1+\delta}$. (We do not even need continuity of these maps).
\end{prop}

Proposition~\ref{prop:regularity_in_E} can be proved by using the kernel estimates, as well as the cancellation conditions, of $\hS$ and $\hN$. Indeed, let $f \in \E{-\gamma}$ for some $\gamma \in (0,4)$. Let $T \in \Psi^{-n}_{\hD}(\hM)$ for some $n \in [0,\gamma)$. We will prove $Tf \in \E{-\gamma+n}$. Let $x \in \hM \setminus \{p\}$ be sufficiently close to $p$. Write $r = \hat{d}(x,p)/4$ where $\hat{d}$ is the metric on $\hM$ induced by $\hthe$. Pick now $\chi_1 \in C^{\infty}_c(B_{2r}(x))$ such that $\chi_1 \equiv 1$ on $B_r(x)$, $\chi_2 \in C^{\infty}_c(B_{2r}(p))$ such that $\chi_2 \equiv 1$ on $B_r(p)$ and $\chi_3 \in C^{\infty}_c(B_{100r}(p))$ such that $\chi_3 \equiv 1$ on $B_{50r}(p)$. Decompose 
\begin{align*}
f &= \chi_1 f + \chi_2 f + \chi_3 (1-\chi_1 - \chi_2) f + (1-\chi_3) (1- \chi_1 - \chi_2) f \\
&= f_1 + f_2 + f_3 + f_4.
\end{align*}
Then $Tf_4 \in C^{\infty}(\hM)$, and in particular $|\hnabla_b^k Tf_4 (x)|$ is bounded by a constant independent of $x$, for all $k \in \mathbb{N}$. We need to show that $|\hnabla_b^k Tf_i(x)| \lesssim r^{-\gamma+n-k}$ for $i=1,2,3$, $k \in \mathbb{N}$. Now writing $K(x,y)$ for the integral kernel of $T$, i.e. letting
$$
Tf(x) = \int_{\hM} f(y) K(x,y) \hthe \wedge d\hthe(y),
$$
we have, from Theorem~\ref{thm:equiv_smoothing} or \ref{thm:equiv_nonneg_order}, that
\begin{equation} \label{eq:K_ker_est}
|(\hnabla_b)_x^k K(x,y)| \lesssim \hat{d}(x,y)^{-Q+n-k}
\end{equation}
for all $k \in \mathbb{N}$. Also, if $n = 0$, then for any normalized bump function $\varphi$ on $B_r(x)$, we have
$$
\|\hnabla_b^k T\varphi \|_{L^{\infty}(B_r(x))} \lesssim r^{-k} \quad \text{for all $k \in \mathbb{N}$};
$$
here $\varphi$ is said to be a normalized bump function $\varphi$ on $B_r(x)$, if $\varphi \in C^{\infty}_c(B_r(x))$, and $\|\hnabla_b^k \varphi \|_{L^{\infty}(B_r(x))} \lesssim r^{-k}$ for all $k \in \mathbb{N}$. Note that $f_1$ is $r^{-\gamma}$ times a normalized bump function on $B_r(x)$. Hence if $n = 0$, then by the above cancellation property, we have
$$
|\hnabla_b^k Tf_1(x)| \lesssim r^{-\gamma-k};
$$
on the other hand, if $n \in (-4,0)$, we will commute $\hnabla_b^k$ through $T$ using Theorem~\ref{thm:comm}, and then use kernel estimates (\ref{eq:K_ker_est}) to bound $|\hnabla_b^k Tf_1(x)|$. The main term that arise is then
$$
\left| \int_{\hM} K(x,y) \hnabla_b^k f_1(y) \hthe \wedge d\hthe(y) \right| \lesssim \int_{\hat{d}(y,x) \leq 2r} \hat{d}(x,y)^{-Q+n} r^{-\gamma-k} \hthe \wedge d\hthe(y) = r^{-\gamma+n-k}.
$$
The estimates on $|\hnabla_b^k Tf_2(x)|$ and $|\hnabla_b^k Tf_3(x)|$ make use of kernel estimates only: for instance, 
\begin{align*}
|\hnabla_b^k Tf_2(x)|
&\lesssim \int_{\hat{d}(y,p) \leq 2r} |(\hnabla_b)^k_x K(x,y)| |f_2(y)| \hthe \wedge d\hthe(y) \\
&\lesssim r^{-Q+n-k} \int_{\hat{d}(y,p) \leq 2r} \hat{d}(y,p)^{-\gamma} \hthe \wedge d\hthe(y) \\
&\lesssim r^{-\gamma+n-k}.
\end{align*}
Similarly,
\begin{align*}
|\hnabla_b^k Tf_3(x)|
&\lesssim \int_{\substack{\hat{d}(y,p) \leq 100r \\ \hat{d}(y,p) \geq r, \hat{d}(y,x) \geq r}} |(\hnabla_b)^k_x K(x,y)| |f_3(y)| \hthe \wedge d\hthe(y) \\
&\lesssim r^Q r^{-Q+n-k} r^{-\gamma} \\
&= r^{-\gamma+n-k}.
\end{align*}
This completes the proof of Proposition~\ref{prop:regularity_in_E}.

\section{A model case of our main theorem} \label{sect:eg}

We will now discuss the proof of our main Theorem~\ref{thm:goal}. There in the assumption we had a function $G \in C^{\infty}(\hM \setminus \{p\})$, which admits an expansion as in (\ref{eq:Gexpansion2}). In this section, we will prove Theorem~\ref{thm:goal} under an additional assumption, namely that 
$$
G = |\psi|^{-1}
$$
where $\psi$ is a smooth function on $\hM$ that satisfies $\hdbarb \psi = 0$, $\re \psi \geq 0$ on $\hM$ and  $\psi\neq0$ on $\hM \setminus \{p\}$. This assumption is satisfied, for instance, when $(\hM,\hthe) = (\mathbb{S}^3, \theta_{\text{std}})$ with $\theta_{\text{std}} = 2 \im \partial (|\zeta|^2 - 1)$ is the standard contact form on $\mathbb{S}^3$, $p = (0,-1) \in \mathbb{S}^2 \subset \mathbb{C}^2$, and $G$ is the Green's function of the conformal sublaplacian of $(\mathbb{S}^3,\theta_{\text{std}})$ with pole at $p$; indeed then $G$ is a multiple of $|\psi|^{-1}$ where $\psi(\zeta) = 1+\zeta^2$ if $\zeta = (\zeta^1, \zeta^2) \in \mathbb{S}^3 \subset \mathbb{C}^2$, and clearly $\psi$ is a CR function on $\mathbb{S}^3$ with non-negative real part. But we should also point out that this assumption on $G$ is rather rigid; it is rarely satisfied. Nevertheless, the proof of Theorem~\ref{thm:goal} under this additional assumption on $G$ will shed some light for the general case, so we single it out in this section.

So in this section, we assume, in addition to the assumptions in Section~\ref{sect:result}, that $G = |\psi|^{-1}$ for some CR function $\psi$ on $\hM$ with $\re \psi \geq 0$ on $\hM$. We let $M = \hM \setminus \{p\}$, $\theta = G^2 \hthe$, and define $\boxb = \vartheta_b \dbarb$ on $C^{\infty}$ functions on $M$ as in Section~\ref{sect:result}. Now we let $\hboxb$ on $\hM$ be the tangential Kohn Laplacian on $(\hM,\hthe)$ defined in Section~\ref{sect:def}; in other words, it is the tangential Kohn Laplacian defined in Section~\ref{sect:hboxbsol}, where one chose $\langle \cdot, \cdot \rangle = \langle \cdot, \cdot \rangle_{\hthe} = $ the pointwise Hermitian inner product on $(0,1)$ forms on $\hM$ given by the pseudohermitian structure $\hthe$, and chose $\hm_0 = \hm_1 = \hthe \wedge d\hthe$. Thus $\hboxb = \hat{\vartheta}_b \dbarb$, where $\hat{\vartheta}_b$ is the formal adjoint of $\dbarb \colon L^2(\hm_0) \to L^2_{(0,1)}(\hm_1)$. Observe that for all $C^{\infty}$ (0,1) forms $\alpha$ on $M$, we have
$$
\vartheta_b \alpha = |\psi|^2 \psi \hat{\vartheta}_b (\psi^{-1} \alpha);
$$
in fact, for all $u \in C^{\infty}_c(M)$ and all $C^{\infty}_c$ $(0,1)$ form $\alpha$ on $M$, we have
\begin{align*}
(\dbarb u, \alpha)_{\theta} 
=&\int_M \langle \dbarb u, \alpha \rangle_{\theta} \, \m \\
=& \int_{\hM} \langle \dbarb u, \alpha \rangle_{\hthe} |\psi|^{-2} \, \hthe \wedge d \hthe \\
=& \int_{\hM} \langle \dbarb (\psi^{-1} u), \psi^{-1} \alpha \rangle_{\hthe} \, \hthe \wedge d \hthe \\
=& \int_{\hM} \psi^{-1} u \cdot \overline{ \hat{\vartheta}_b (\psi^{-1} \alpha) } \, \hthe \wedge d \hthe \\
=& (u, |\psi|^2 \psi \hat{\vartheta}_b (\psi^{-1} \alpha)  )_{\theta}.
\end{align*}
Since $\boxb = \vartheta_b \dbarb$, and since $\dbarb$ commutes with $\psi^{-1}$, this shows
$$
\boxb u = |\psi|^2 \psi \hboxb (\psi^{-1} u) \quad \text{for $u \in C^{\infty}(M)$}.
$$
Thus to solve $\boxb u = f$, it suffices to solve
\begin{equation} \label{eq:model}
\hboxb U = |\psi|^{-2} \psi^{-1} f
\end{equation}
and set $u = \psi U$. This we have basically accomplished in Section~\ref{sect:hboxbsol} above; indeed, there we constructed relative solution operator $\hN$ and Szeg\H{o} projection $\hS$ for $\hboxb$, and proved that $\hN$ maps $\E{-3+\delta}$ to $\E{-1+\delta}$ for all $0 < \delta < 1$. For the $f$ as in the statement of Theorem~\ref{thm:goal}, one can prove, following the proof of Proposition~\ref{prop3} below, that $$\hS (|\psi|^{-2} \psi^{-1} f) = 0.$$ Since $|\psi|^{-2} \psi^{-1} f \in \E{-3+\delta}$ for some $0 < \delta < 1$, it follows that $$U := \hN (|\psi|^{-2} \psi^{-1} f)$$ is a solution to (\ref{eq:model}) in $\E{-1+\delta}$, and hence $$u := \psi U$$ is a solution to $\boxb u = f$ in $\E{1+\delta}$. This completes the proof of Theorem~\ref{thm:goal}, under our additional assumption on $G$.

\section{Reduction to solution of $\tboxb$} \label{sect:tboxb}

We now return to the proof of Theorem~\ref{thm:goal} in the general case, without the additional assumption on $G$. It turns out that one can perform a similar reduction, where the solution to $\boxb$ on $M$ is reduced to the solution of a certain $\tboxb$ operator on $\hM$; unfortunately, this $\tboxb$ does not fall under the scope covered in Section~\ref{sect:hboxbsol}; in particular, this $\tboxb$ involve the adjoint of $\dbarb$ on $\hM$ with respect to measures that are not smooth across the point $p$. Thus a further reduction has to be carried out, where one reduces the solution of $\tboxb$ to the solution of a $\hboxb$, that falls under the scope covered in Section~\ref{sect:hboxbsol}. We will carry out the reduction of Theorem~\ref{thm:goal} to the solution of $\tboxb$ in this section. In the next section, we will reduce the solution of $\tboxb$ to a properly defined $\hboxb$ on $\hM$.

To begin with, first we claim that in the general case, one can still construct a CR function $\psi$ on $\hM$, with $\re \psi \geq 0$ on $\hM$, such that in CR normal coordinates $(z,t)$ near $p$, we have
$$
\psi = 2\pi (|z|^2 - it) + R, \quad R \in \O{4}.
$$
In particular,
$$
|\psi|^2 G^2 = 1 + \E{2}.
$$ 
This has been established in Theorem 4.4 of \cite{MR3366852}; alternatively, one can construct such a $\psi$ by using our Lemma~\ref{lem:psi_inf_order}, and the argument in the remark following it. Indeed, since we have better asymptotics (\ref{s1-e5b}) for $\hZ$ in CR normal coordinates near $p$, following the construction in Lemma~\ref{lem:psi_inf_order} (with $y := p$), we can construct some $\psi_0$ near $p$ so that 
$$
\psi_0 = \pi (|z|^2 - it) + \O{6}.
$$ 
Then following the construction as in the remark following Lemma~\ref{lem:psi_inf_order}, we can further construct a global CR function  $\psi_{00}$ on $\hM$, with $\re \psi_{00} \geq 0$ on $\hM$ and $\psi_{00}\neq0$ on $\hM \setminus \{p\}$, such that
$$
\psi_{00} = \pi(|z|^2 - it) + \O{4}.
$$
Hence we may simply take $\psi = 2\psi_{00}$.

Now define $a \in \E{2}$ by
$$
a := |\psi|^{-2} G^{-2} - 1 \quad \text{near $p$},
$$
and define $\chi \in C^{\infty}_c(B_{2 \varepsilon_0}(p))$ that is identically 1 on $B_{\varepsilon_0}(p)$, where $B_r(p)$ denotes a ball of radius $r$ centered at $p$ with respect to the metric determined by $\hthe$. Here $\varepsilon_0$ is a sufficiently small positive constant to be determined; all theorems below only hold if $\varepsilon_0$ is sufficiently small (see (\ref{eq:epsilon0_choice}) below for the choice of $\varepsilon_0$). Define two (possibly non-smooth) measures
$$\tm_0 = (1 + \chi a)^{-1} \hthe \wedge d\hthe,$$
$$\tm_1 = |\psi|^2 G^2 \hthe \wedge d\hthe.$$ 
(Note that they agree on $B_{\varepsilon_0}(p)$.) We use $\tm_0$ to define an $L^2$ space of functions (denoted $L^2(\tm_0)$):
$$
(u, v)_{\tm_0} = \int_{\hM} u \overline{v} \, \tm_0,
$$
and $\tm_1$ to define an $L^2$ space of $(0,1)$ forms (denoted $L^2_{(0,1)}(\tm_1)$):
$$
(\alpha,\beta)_{\tm_1} = \int_{\hM} \langle \alpha, \beta \rangle_{\hthe} \tm_1.
$$
(Note we use $\hthe$ to measure the pointwise inner product of $(0,1)$ forms.) Let $\tdbarb$ act on $C^{\infty}$ functions on $M$, and $\tvdbarb$ be its formal adjoint under the inner products of $L^2(\tm_0)$ and $L^2_{(0,1)}(\tm_1)$. In other words, $\tvdbarb$ is the unique differential operator acting on $C^{\infty}$ $(0,1)$ forms on $M$ such that $$(\tdbarb u, \alpha)_{\tm_1} = (u, \tvdbarb \alpha)_{\tm_0}$$ for all $u \in C^{\infty}_c(M)$ and all $C^{\infty}_c$ $(0,1)$ form $\alpha$ on $M$. Let $\tboxb = \tvdbarb \tdbarb$ act on $C^{\infty}$ functions on $M$. Then  
\begin{equation} \label{eq:boxbtboxb}
\boxb u = (1+\chi a)^{-1} \psibar^{-1} G^{-4} \tboxb(\psi^{-1} u) \quad \text{for $u \in C^{\infty}(M)$}.
\end{equation}
This is because for all $C^{\infty}$ $(0,1)$ forms $\alpha$ on $M$, we have 
\begin{equation} \label{eq:dbarb*tvdbarb}
\vartheta_b \alpha = (1+\chi a)^{-1} G^{-4} \psibar^{-1} \tvdbarb (\psi^{-1} \alpha);
\end{equation}
in fact, for all $u \in C^{\infty}_c(M)$ and all $C^{\infty}_c$ $(0,1)$ form $\alpha$ on $M$, we have
\begin{align*}
(\dbarb u, \alpha)_{\theta} 
=&\int_M \langle \dbarb u, \alpha \rangle_{\theta} \, \m \\
=& \int_{\hM} \langle \tdbarb u, \alpha \rangle_{\hthe} G^2 \, \hthe \wedge d \hthe \\
=& \int_{\hM} \langle \tdbarb u, \alpha \rangle_{\hthe} |\psi|^{-2} \, \tm_1 \\
=& \int_{\hM} \langle \tdbarb (\psi^{-1} u), \psi^{-1} \alpha \rangle_{\hthe} \, \tm_1 \\
=& \int_{\hM} \psi^{-1} u \cdot \overline{ \tvdbarb (\psi^{-1} \alpha) } \, \tm_0 \\
=& \int_{M} u \cdot \overline{ \psibar^{-1} \tvdbarb (\psi^{-1} \alpha) (1+\chi a)^{-1} G^{-4}} \, \m \\
=& (u, \psibar^{-1} \tvdbarb (\psi^{-1} \alpha) (1+\chi a)^{-1} G^{-4})_{\theta}.
\end{align*}
Since $\boxb = \vartheta_b \dbarb$, and $\psi^{-1}$ commutes with $\dbarb$, (\ref{eq:boxbtboxb}) follows from (\ref{eq:dbarb*tvdbarb}).
Thus solving $\boxb u  = f$ amounts to finding a solution, in $C^{\infty}(M)$, to the equation
$$
\tboxb(\psi^{-1} u) = \tf,
$$
where
$$
\tf := (1+\chi a) \psibar G^4 f \in \E{-3+\delta}.
$$
We will show that 
\begin{thm} \label{thm:tboxb}
There exists a function $\tilde{u} \in C^{\infty}(M)$ such that
\begin{equation} \label{eq:tboxb}
\tilde{u} \in \E{-1+\delta} \quad \text{with} \quad \tboxb \tilde{u} = \tf.
\end{equation}
\end{thm}
Assuming this, then $$u := \psi \tilde{u} \in \E{1+\delta}$$ is a solution to (\ref{eq:boxb}), and our main result, namely Theorem~\ref{thm:goal}, follows.

To prove Theorem~\ref{thm:tboxb}, we proceed via the $L^p$ theory for $\tboxb$. First we need the $L^2$ theory. First we extend $\tdbarb$ to
$$
\tdbarb \colon L^2(\tm_0) \to L^2_{(0,1)}(\tm_1)
$$ 
by taking the Hilbert space closure of $\tdbarb$ acting on functions in $C^{\infty}(\hM)$. In other words, we define 
$$
u \in \text{Dom}(\tdbarb \colon L^2(\tm_0) \to L^2_{(0,1)}(\tm_1)),
$$ 
if and only if there exists a sequence $u_j \in C^{\infty}(\hM)$ such that 
$$u_j \to u \text{ in $L^2(\tm_0)$}, \quad \text{and} \quad \tdbarb u_j \to \alpha \text{ in $L^2_{(0,1)}(\tm_1)$}$$ 
for some $\alpha \in L^2_{(0,1)}(\tm_1)$. In that case $\alpha$ is uniquely determined by $u$, and we define $\tdbarb u = \alpha$. 

Next, let 
$$
\tdbarb^* \colon L^2_{(0,1)}(\tm_1) \to L^2(\tm_0)
$$ 
the Hilbert space closure of $\tvdbarb$ acting on functions in $C^{\infty}_{(0,1)}(\hM)$. In other words, we define 
$$
\alpha \in \text{Dom}(\tdbarb^* \colon L^2_{(0,1)}(\tm_1) \to L^2(\tm_0)),
$$ 
if and only if there exists a sequence $\alpha_j \in C^{\infty}_{(0,1)}(\hM)$ such that 
$$\alpha_j \to \alpha \text{ in $L^2_{(0,1)}(\tm_1)$}, \quad \text{and} \quad \tvdbarb \alpha_j \to u \text{ in $L^2(\tm_0)$}$$ 
for some $u \in L^2(\tm_0)$. In that case $u$ is uniquely determined by $\alpha$, and we define $\tdbarb^* \alpha = u$.

Furthermore, we define a densely defined operator 
$$
\tboxb \colon L^2(\tm_0) \to L^2(\tm_0),
$$ 
with domain given by
\begin{align*}
\text{Dom}(\tboxb) := \{ & u \in \text{Dom}[\tdbarb \colon L^2(\tm_0) \to L^2_{(0,1)}(\tm_1)],  \\ 
& \quad \tdbarb u \in \text{Dom}[\tdbarb^* \colon L^2_{(0,1)}(\tm_1) \to L^2(\tm_0)] \}
\end{align*}
If $u \in \text{Dom}(\tboxb)$, we define $\tboxb u = \tdbarb^* \tdbarb u$.

\begin{prop} \label{prop1}
The operator $\tdbarb \colon L^2(\tm_0) \to L^2_{(0,1)}(\tm_1)$ has closed range, and so does $\tboxb \colon L^2(\tm_0) \to L^2(\tm_0)$. 
\end{prop}

The proof of this proposition will be deferred to the next section.

Proposition~\ref{prop1} allows us to define the relative solution operator
$$
\tN \colon L^2(\tm_0) \to \text{Dom}(\tboxb) \subseteq L^2(\tm_0)
$$
and the orthogonal projection onto the kernel of $\tboxb$ in $L^2(\tm_0)$:
$$
\tS \colon L^2(\tm_0) \to \text{Dom}(\tboxb) \subseteq L^2(\tm_0)
$$
Then 
$$
\tboxb \tN + \tS = I \quad \text{on $L^2(\tm_0)$}.
$$

We now turn to the $L^p$ theory for $\tN$, $\tS$ and $\tboxb$. We will prove the following proposition in the next section:

\begin{prop} \label{prop2} 
$\tN$ and $\tS$ admits the following (continuous) extensions:
$$
\tS \colon L^p(\tm_0) \to L^p(\tm_0) \quad \text{for all $p \in (1,\infty)$},
$$
$$
\tN \colon L^{4/3}(\tm_0) \to L^4(\tm_0).
$$
$$
\hnabla_b \tN \colon L^{4/3}(\tm_0) \to L^2(\tm_0).
$$
Also, if $F \in C^{\infty}_c(M)$, then $\tS F$ and $\tN F$ are in $C^{\infty}(M)$.
\end{prop}

%Now we define
%$$
%\tdbarb \colon L^4(\tm_0) \to L^2_{(0,1)}(\tm_1)
%$$
%as the Banach space closure of $\tdbarb$ on functions in $C^{\infty}_c(M)$ in the graph norm $L^4 \times L^2_{(0,1)}$ (this is just the restriction of $\tdbarb \colon L^2(\tm_0) \to L^2_{(0,1)}(\tm_1)$ to $L^4(\tm_0)$). Define also
%$$
%\tdbarb^* \colon L^2_{(0,1)}(\tm_1) \to L^{4/3}(\tm_0)
%$$
%to be the Banach space closure of $\tvdbarb$ on $C^{\infty}_c$ $(0,1)$ forms on $M$ under the graph norm $L^2 \times L^{4/3}$, and 
%$$
%\tboxb \colon L^4(\tm_0) \to L^{4/3}(\tm_0)
%$$
%be the composition $\tdbarb^* \tdbarb$ with domain
%\begin{align*}
%\text{Dom}(\tboxb) := \{ & u \in \text{Dom}[\tdbarb \colon L^4(\tm_0) \to L^2_{(0,1)}(\tm_1)],  \\ 
%& \quad \tdbarb u \in \text{Dom}[\tdbarb^* \colon L^2_{(0,1)}(\tm_1) \to L^{4/3}(\tm_0)] \}.\end{align*}

Now we define 
$$
\tboxb \colon L^4(\tm_0) \to L^{4/3}(\tm_0)
$$
to be the Banach space closure of $\tboxb$ acting on $C^{\infty}_c(M)$ under the graph norm $L^4 \times L^{4/3}$. In other words, we define 
$$
u \in \text{Dom}(\tboxb \colon L^4(\tm_0) \to L^{4/3}(\tm_0)),
$$ 
if and only if there exists a sequence $u_j \in C^{\infty}_c(M)$ such that 
$$u_j \to u \text{ in $L^4(\tm_0)$}, \quad \text{and} \quad \tboxb u_j \to F \text{ in $L^{4/3}(\tm_0)$}$$ 
for some $F \in L^{4/3}(\tm_0)$. In that case $F$ is uniquely determined by $u$, and we define $\tboxb u = F$. 

Using Proposition~\ref{prop2}, we will show that
$$
\tN \colon L^{4/3}(\tm_0) \to \text{Dom}[\tboxb \colon L^4(\tm_0) \to L^{4/3}(\tm_0)] \subseteq L^4(\tm_0),
$$
with
\begin{equation} \label{eq:tboxbLp}
\tboxb \tN + \tS = I \quad \text{on $L^{4/3}(\tm_0)$}.
\end{equation}
To see this, we argue as follows. Given $F \in L^{4/3}(\tm_0)$, Proposition~\ref{prop2} guarantees that $\tS F \in L^{4/3}(\tm_0)$, $\tN F \in L^4(\tm_0)$, and $\hnabla_b \tN F \in L^2(\tm_0)$. Thus we can pick a sequence of cut-off functions $\phi_j \in C^{\infty}_c(M)$, that is identically equal to 1 except near $p$, and vanishes in a small neighborhood of $p$, such that if $\chi_j$ is the characteristic function of the support of $\hnabla_b \phi_j$, then
$$
\phi_j (I-\tS) F \to (I-\tS) F \quad \text{in $L^{4/3}(\tm_0)$},
$$
$$
\chi_j \tN F \to 0 \quad \text{in $L^4(\tm_0)$},
$$
and
$$
\chi_j \hnabla_b \tN F \to 0 \quad \text{in $L^2(\tm_0)$}.
$$
Let now $F_j \in C^{\infty}_c(\hM)$ be a sequence of functions with 
$$
F_j \to F \quad \text{in $L^{4/3}(\tm_0)$}.
$$
Then $u_j := \phi_j \tN F_j \in C^{\infty}_c(M)$, 
$$
u_j \to \tN F \quad \text{in $L^4(\tm_0)$},
$$
and
$$
\tboxb u_j = \phi_j \tboxb \tN F_j + \text{error},
$$
where the error is supported in the support of $\chi_j$. In fact, if $\hobar$ is a frame of $(0,1)$ vector of unit length near $p$, and $\hat{\Zbar}$ is its dual, then 
\begin{align*}
|\text{error}| 
&\leq |\tboxb \phi_j| |\tN F_j| + \chi_j |\hZ \phi_j \cdot i_{\hobar} \tdbarb \tN F_j| + |\hZ \tN F_j \cdot i_{\hobar} \tdbarb \phi_j| \\
&\leq |\tboxb \phi_j| |\tN F_j - \tN F| + \chi_j |\tboxb \phi_j| |\tN F| + |\hnabla_b \phi_j| |\hnabla_b \tN F_j - \hnabla_b \tN F| + \chi_j  |\hnabla_b \phi_j| |\hnabla_b \tN F|.
\end{align*}
Now we show $\|\text{error}\|_{L^{4/3}(\tm_0)} \to 0$ as $j \to \infty$: in fact,
$$
\|\tboxb \phi_j\|_{L^2(\tm_0)} \leq C, \quad \|\hnabla_b \phi_j\|_{L^4(\tm_0)} \leq C,
$$
and 
$$
\|\tN F_j - \tN F\|_{L^4(\tm_0)} \to 0, \quad \|\hnabla_b \tN F_j - \hnabla_b \tN F\|_{L^2(\tm_0)} \to 0
$$
as $j \to \infty$. Furthermore, 
$$
\|\chi_j \tN F\|_{L^4(\tm_0)} \to 0, \quad \|\chi_j \hnabla_b \tN F\|_{L^2(\tm_0)} \to 0
$$
by our choice of $\phi_j$. Thus $\|\text{error}\|_{L^{4/3}(\tm_0)} \to 0$ as $j \to \infty$ as desired.

Now
$$
\phi_j \tboxb \tN F_j = \phi_j (I - \Pi) F_j,
$$
(this holds because $F_j \in C^{\infty}_c(M) \subseteq L^2(\tm_0)$), so $$\phi_j \tboxb \tN F_j \to (I-\Pi) F \quad \text{in $L^{4/3}(\tm_0)$}$$ by an argument similar to the one above. Altogether, it follows that
$$
\tboxb u_j \to (I-\tS) F \quad \text{in $L^{4/3}(\tm_0)$}.
$$
Thus $\tN F \in \text{Dom}(\tboxb \colon L^4(\tm_0) \to L^{4/3}(\tm_0))$, and (\ref{eq:tboxbLp}) follows.

We now return to the equation (\ref{eq:tboxb}) we want to solve, namely
$$
\tboxb \tilde{u} = \tf.
$$ 
Recall 
$$
\tf = (1+\chi a) \psibar G^4 f \in \E{-3+\delta}.
$$
In particular, $\tf \in L^{4/3}(\tm_0)$. Thus we may apply the identity (\ref{eq:tboxbLp}) above. What we need is the following claim:

\begin{prop} \label{prop3} 
$$\tS \tf = 0.$$
\end{prop}

We will come back and prove this proposition in the next section. Assuming this proposition, then 
$$
\tilde{u} := \tN \tf \in L^4(\tm_0)
$$ 
solves (\ref{eq:tboxb}), in the sense that it is in the domain of $\tboxb \colon L^4(\tm_0) \to L^{4/3}(\tm_0)$, and that its image under this operator is $\tilde{f}$.

Finally we invoke the following proposition (again to be proved in the next section):

\begin{prop} \label{prop4} 
$\tN$ maps $\E{-3+\delta}$ into $\E{-1+\delta}$ for all $0 < \delta < 1$ (we do not even need continuity of this map).
\end{prop}

Then since $\tf \in \E{-3+\delta}$ for all $0 < \delta < 1$, we have 
$$\tilde{u} \in \E{-1+\delta}$$ for all $0 < \delta < 1$. In particular, $\tilde{u} \in C^{\infty}(M) \cap L^4(\tm_0)$, $\hnabla_b \tilde{u}$ (defined classically) is in $L^2(\tm_0)$, and $\tboxb \tilde{u}$ (defined classically) is in $L^{4/3}(\tm_0)$. Thus the image of $\tilde{u}$ under $\tboxb \colon L^4(\tm_0) \to L^{4/3}(\tm_0)$ is equal to the classically defined $\tboxb \tilde{u}$ almost everywhere. Together with what we have shown earlier, this shows that $\tilde{u}$ is a classical solution to (\ref{eq:tboxb}). This proves Theorem~\ref{thm:tboxb}, modulo the Propositions~\ref{prop1}, \ref{prop2}, \ref{prop3} and \ref{prop4} we have stated.

\section{Reduction to solution of $\hboxb$} \label{sect:hboxb}

In the previous section, we reduced the proof of our main Theorem~\ref{thm:goal} to the proofs of Propositions~\ref{prop1}, \ref{prop2}, \ref{prop3} and \ref{prop4}, where one has to solve the $\tboxb$ operator on $\hM$. The difficulty lies in the fact that $\tboxb$ is defined as $\tvdbarb \tdbarb$, where $\tvdbarb$ is an adjoint taken with respect to two measures $\tm_0$, $\tm_1$ that are in general not smooth across the point $p$. In this section, we will first construct another tangential Kohn Laplacian $\hboxb$, where the adjoint is taken with respect to smooth measures. Then we will reduce the solution to $\tboxb$ (i.e. the proofs of Propositions~\ref{prop1}, \ref{prop2}, \ref{prop3} and \ref{prop4}) to the solution of $\hboxb$, which we already understood from Section~\ref{sect:hboxbsol}. This will complete the proof of Theorem~\ref{thm:goal} in the general case.

First define two measures
$$\hm_0 := \hthe \wedge d\hthe,$$
$$\hm_1 := (1+\chi a) \tm_1.$$
Note that then 
$$\hm_0 := (1+\chi a) \tm_0$$ 
as well. In particular, $\hm_0 = \hm_1$ near $p$, and is smooth over there. Define also $\langle \cdot, \cdot \rangle = \langle \cdot, \cdot \rangle_{\hthe}$, the smoothly varying pointwise inner product on $(0,1)$ forms on $\hM$ given by $\hthe$. Then we may define
$$
\hdbarb \colon L^2(\hm_0) \to L^2_{(0,1)}(\hm_1)
$$ 
as a Hilbert space closure of $\hdbarb$ acting on functions in $C^{\infty}(\hM)$, its Hilbert space adjoint
$$
\hdbarb^* \colon L^2_{(0,1)}(\hm_1) \to L^2(\hm_0),
$$
and a tangential Kohn Laplacian 
$$
\hboxb \colon L^2(\hm_0) \to L^2(\hm_0),
$$ 
as in Section~\ref{sect:hboxbsol}. By Proposition~\ref{prop:Fact 1}, the operator $\hdbarb \colon L^2(\hm_0) \to L^2_{(0,1)}(\hm_1)$ has closed range, hence so does $\hboxb \colon L^2(\hm_0) \to L^2(\hm_0)$; by Proposition~\ref{prop:pdo}, the relative solution operator
$$
\hN \colon L^2(\hm_0) \to \text{Dom}(\hboxb) \subseteq L^2(\hm_0)
$$ 
and the Szeg\H{o} projection 
$$
\hS \colon L^2(\hm_0) \to \text{Dom}(\hboxb) \subseteq L^2(\hm_0)
$$
to $\hboxb$ are in $\Psi^{-2}_{\hD}(\hM)$ and $\Psi^0_{\hD}(\hM)$ respectively, where $\hD$ is the contact distribution $\hor$ on $\hM$. We also have $\hS \colon \E{-1+\delta} \to \E{-1+\delta}$ and $\hN \colon \E{-3+\delta} \to \E{-1+\delta}$ for all $0 < \delta < 1$, by Proposition~\ref{prop:regularity_in_E}. With these, we now turn to the proof of Propositions~\ref{prop1}, \ref{prop2}, \ref{prop3} and \ref{prop4} we stated in the previous section; the key will ultimately be establishing a relation between $\tboxb$ and $\hboxb$ (c.f. (\ref{eq:compare_tboxb1}) and (\ref{eq:compare_tboxb2}) below). 

\begin{proof}[Proof of Proposition~\ref{prop1}]
First note that $\tdbarb \colon L^2(\tm_0) \to L^2_{(0,1)}(\tm_1)$ and $\hdbarb \colon L^2(\hm_0) \to L^2_{(0,1)}(\hm_1)$ are identical as operators by definition. In other words, these operators have the same domain of definition, and for $u$ in this common domain,
$$\tdbarb u = \hdbarb u.$$
By Proposition~\ref{prop:Fact 1}, $\hdbarb \colon L^2(\hm_0) \to L^2_{(0,1)}(\hm_1)$ has closed range. Since convergence in $L^2_{(0,1)}(\tm_1)$ is equivalent to convergence in $L^2_{(0,1)}(\hm_1)$, it follows that $\tdbarb \colon L^2(\tm_0) \to L^2_{(0,1)}(\tm_1)$ has closed range as well, proving the first part of Proposition~\ref{prop1}.

To proceed further, we need the following lemma about $\tdbarb^*$: 
\begin{lem} \label{lem:tdbarb*adjoint}
For all $$u \in \text{Dom} [\tdbarb \colon L^2(\tm_0) \to L^2_{(0,1)}(\tm_1)], \quad \alpha \in \text{Dom} [\tdbarb^* \colon L^2_{(0,1)}(\tm_1) \to L^2(\tm_0)],$$ we have
$$
(\tdbarb u, \alpha)_{\tm_1} = (u, \tdbarb^* \alpha)_{\tm_0}.
$$
\end{lem}

Assume this for the moment. Then for all $u \in \text{Dom} [\tboxb \colon L^2(\tm_0) \to L^2(\tm_0)]$, we have
\begin{equation}  \label{eq:tdbarb^2u=utboxb}
(\tdbarb u, \tdbarb u)_{\tm_1} = (u, \tboxb u)_{\tm_0}.
\end{equation}
Thus the kernels of $\tboxb \colon L^2(\tm_0) \to L^2(\tm_0)$ and $\tdbarb \colon L^2(\tm_0) \to L^2_{(0,1)}(\tm_1)$ are identical. Let's call the common kernel $\mathcal{K}$. It is a closed subspace of $L^2(\tm_0)$ since it is the kernel of a closed linear operator.

Now given $u \in \text{Dom}[\tboxb \colon L^2(\tm_0) \to L^2(\tm_0)]$ with $u$ orthogonal to the kernel of $\mathcal{K}$, we have
$$
\|u\|_{L^2(\tm_0)}^2 \leq C \|\tdbarb u\|_{L^2_{(0,1)}(\tm_1)}^2
$$
since $\tdbarb \colon L^2(\tm_0) \to L^2_{(0,1)}(\tm_1)$ has closed range. However, the right hand side of this equation is just $C (u, \tboxb u)_{\tm_0}$ by (\ref{eq:tdbarb^2u=utboxb}), which is bounded by $C \|u\|_{L^2(\tm_0)} \|\tboxb u\|_{L^2(\tm_0)}$. Thus
$$
\|u\|_{L^2(\tm_0)} \leq C \|\tboxb u\|_{L^2(\tm_0)},
$$
which shows that $\tboxb \colon L^2(\tm_0) \to L^2(\tm_0)$ has closed range. 

It remains to prove Lemma~\ref{lem:tdbarb*adjoint}. This will follow from the definitions  of $\tdbarb$ and $\tdbarb^*$ on $L^2$, once we prove the following claim: we claim that for any $u \in C^{\infty}(\hM)$ and any $C^{\infty}$ $(0,1)$ form $\alpha$ on $\hM$, we have
\begin{equation} \label{eq:tvdbarbnoncompact}
(\tdbarb u, \alpha)_{\tm_1} = (u, \tvdbarb \alpha)_{\tm_0}.
\end{equation}
Note that we do not require $u$ nor $\alpha$ to be compactly supported in $M$; otherwise this would follow from the definition of $\tvdbarb$. To see that (\ref{eq:tvdbarbnoncompact}) is true, first assume in addition that $u$ is compactly supported in $M$. Let $\phi_j$ be a sequence of $C^{\infty}_c$ cut-off functions such that it is identically 1 except near $p$, and vanishes identically near $p$. One can pick such a sequence such that both $1-\phi_j$ and $\hZ \phi_j \to 0$ in $L^1(\tm_0)$; then $\phi_j \alpha \to \alpha$ in $L^1_{(0,1)}(\tm_1)$, and $\tvdbarb (\phi_j \alpha) = \phi_j \tvdbarb \alpha + (\hZ \phi_j) i_{\hobar} \alpha \to \tvdbarb \alpha$ in $L^1(\tm_0)$.  Thus from
$$
(\tdbarb u, \phi_j \alpha)_{\tm_1} = (u, \tvdbarb (\phi_j \alpha))_{\tm_0},
$$
letting $j \to \infty$, we get the identity (\ref{eq:tvdbarbnoncompact}) in this case as desired.

Next, if both $u$ and $\alpha$ are only $C^{\infty}$ in $\hM$, but not necessarily compactly supported in $M$, we note
$$
(\tdbarb (\phi_j u), \alpha)_{\tm_1} = (\phi_j u, \tvdbarb \alpha)_{\tm_0}
$$
by what we have just proved, where $\phi_j$ is the same sequence of cut-offs we have chosen above. Then $\phi_j u \to u$ in $L^1(\tm_0)$, and $\tdbarb (\phi_j u) \to \tdbarb u$ in $L^1_{(0,1)}(\tm_1)$. Thus (\ref{eq:tvdbarbnoncompact}) follows in full generality by letting $j \to \infty$ in the above identity. This completes the proof of Proposition~\ref{prop1}.
\end{proof}

We remark, for later convenience, that (\ref{eq:tvdbarbnoncompact}) remains true, as long as the following holds: 
\begin{enumerate}[(i)]
\item $u \in C^1(M)$ with $u \in L^{\infty}(M)$, $g_{\hthe}(\hnabla_b u,\hnabla_b u) \in L^{\infty}(M)$, and 
\item $\alpha \in C^1(M)$, with $\alpha = v \hat{\omegabar}$ near $p$, where $\langle \hat{\omegabar}, \hat{\omegabar} \rangle_{\hthe} = 1$, $v \in L^{\infty}$ and $g_{\hthe}(\hnabla_b v,\hnabla_b v) \in L^{\infty}$ near $p$.
\end{enumerate}  
This follows directly from the proof above.

Next, to prove Proposition~\ref{prop2}, we need to understand, at the level of $L^2$, the relation between $\tN$ and $\hN$, and that between $\tS$ and $\hS$. In order to do so, we need to first understand the relation between $\tdbarb^*$ and $\hdbarb^*$ on the level of $L^2$. That in turn requires an alternative characterization of $\hdbarb^*$ on $L^2$, to which we now turn.

\begin{lem} \label{lem:domhdbarb*}
$\alpha \in \text{Dom}[\hdbarb^* \colon L^2_{(0,1)}(\hm_1) \to L^2(\hm_0)]$, if and only if there exists a sequence $\alpha_j \in C^{\infty}_{(0,1)}(\hM)$, such that 
$$\alpha_j \to \alpha \text{ in $L^2_{(0,1)}(\hm_1)$}, \quad \text{and} \quad \hdbarb^* \alpha_j \to u \text{ in $L^2(\hm_0)$}$$ 
for some $u \in L^2(\hm_0)$. In that case $\hdbarb^* \alpha = u$. 
\end{lem}

\begin{proof}
By Proposition~\ref{prop:Fact 1}, $\hdbarb^* \colon L^2_{(0,1)}(\hm_1) \to L^2(\hm_0)$ has closed range. Thus one can define a relative solution operator 
$$
\hK^* \colon L^2(\hm_0) \to \text{Dom}[\hdbarb^*] \subseteq L^2_{(0,1)}(\hm_1),
$$ 
such that
$$
\hdbarb^* \hK^* + \hS = I \quad \text{on $L^2(\hm_0)$},
$$
and
$$
\hK^* \hdbarb^* + \hS_1 = I \quad \text{on $\text{Dom}[\hdbarb^*] \subseteq L^2_{(0,1)}(\hm_1)$},
$$
where
$$
\hS_1 \colon L^2_{(0,1)}(\hm_1) \to L^2_{(0,1)}(\hm_1) 
$$
is the orthogonal projection onto the closed subspace of $L^2_{(0,1)}(\hm_1)$ given by the kernel of $\hdbarb^* \colon L^2_{(0,1)}(\hm_1) \to L^2(\hm_0)$. It is known by classical theory that $\hK^*$ and $\hS_1$ are pseudolocal on $\hM$. In particular, if $\alpha$ is in the domain of $\hdbarb^* \colon L^2_{(0,1)}(\hm_1)$, then letting $u_j$ be a sequence of $C^{\infty}$ functions on $\hM$ that satisfies $u_j \to \hdbarb^* \alpha$ in $L^2(\hm_0)$, and $\beta_j$ be a sequence of $C^{\infty}$ $(0,1)$ forms on $\hM$ that converges to $\alpha$ in $L^2_{(0,1)}(\hm_1)$, then
$$
\alpha_j := \hK^* u_j + \hS_1 \beta_j \in C^{\infty}_{(0,1)}(\hM)
$$
satisfies
$$
\alpha_j \to \hK^* \hdbarb^* \alpha + \hS_1 \alpha = (I-\hS_1) \alpha + \hS_1 \alpha = \alpha \quad \text{in $L^2_{(0,1)}(\hm_1)$},
$$
and 
$$
\hdbarb^* \alpha_j = \hdbarb^* \hK^* u_j = (I - \hS) u_j \to (I - \hS) \hdbarb^* \alpha = \hdbarb^* \alpha.
$$
(The last identity uses $\hS \hdbarb^* \alpha = 0$ for all $\alpha$ in the domain of $\hdbarb^* \colon L^2_{(0,1)}(\hm_1) \to L^2(\hm_0)$, which is clear from our definition of $\hdbarb^*$ as the Hilbert space adjoint of $\hdbarb$.) This establishes half of our lemma.

The other half of the lemma is easier (and does not rely on $\hdbarb$ having closed range in $L^2$): in fact, suppose $\alpha \in L^2_{(0,1)}(\hM)$, and suppose there exists a sequence $\alpha_j \in C^{\infty}_{(0,1)}(\hM)$, such that 
$$
\alpha_j \to \alpha \text{ in $L^2_{(0,1)}(\hm_1)$}, \quad \text{and} \quad \hdbarb^* \alpha_j \to u \text{ in $L^2(\hm_0)$}
$$ 
for some $u \in L^2(\hm_0)$. Then given any $U \in \text{Dom}[\hdbarb \colon L^2(\hm_0) \to L^2_{(0,1)}(\hm_1)]$, we take a sequence $U_j \in C^{\infty}(\hM)$ such that 
$$
U_j \to U \text{ in $L^2(\hm_0)$}, \quad \text{and} \quad \hdbarb U_j \to \hdbarb U \text{ in $L^2_{(0,1)}(\hm_1)$}.
$$ 
Now
$$
(\hdbarb U_j, \alpha_j)_{\hm_1} = (U_j, \hdbarb^* \alpha_j)_{\hm_0}
$$
for all $j$, since both $U_j$ and $\alpha_j$ are smooth on $\hM$. Letting $j \to \infty$, we get
$$
(\hdbarb U, \alpha)_{\hm_1} = (U, u)_{\hm_0}.
$$
Since this is true for all $U$ in the domain of $\hdbarb \colon L^2(\hm_0) \to L^2_{(0,1)}(\hm_1)$, it follows that $\alpha \in \text{Dom}[\hdbarb^* \colon L^2_{(0,1)}(\hm_1) \to L^2(\hm_0)]$, and $\hdbarb^* \alpha = u$. This completes the proof of our lemma.
\end{proof}

\noindent{\textbf{Remark}.} The analog of the implication ($\Rightarrow$) in Lemma~\ref{lem:domhdbarb*} for $\tdbarb^*$ may not be true, because we do not yet pseudolocality of the relative solution operator of $\tdbarb^*$. This is why we had to define $\tdbarb^*$ on $L^2$ by density rather than as an $L^2$ adjoint.

Now from the above lemma, and our definition of $\tdbarb$ on $L^2$, we claim the following:

\begin{lem} 
$\tdbarb^* \colon L^2_{(0,1)}(\tm_1) \to L^2(\tm_0)$ and $\hdbarb^* \colon L^2_{(0,1)}(\hm_1) \to L^2(\hm_0)$ have the same domain of definition, and that for $\alpha$ in this common domain of definition, we have
\begin{equation} \label{eq:tdbarb*hdbarb*}
\tdbarb^* \alpha = \hdbarb^* \alpha + g \alpha,
\end{equation}
where 
$$
g\alpha := -|\psi|^2 G^2 (1+\chi a)^{-2} [\hZ(\chi a)] i_{\hobar} \alpha.
$$
\end{lem}

\begin{proof}
Suppose first $\alpha \in C^{\infty}(\hM)$. Then for $u \in C^{\infty}(\hM)$, we have
\begin{align*}
(\hdbarb u, \alpha)_{\hm_1}
&= (\tdbarb u, (1+\chi a)^{-1} \alpha)_{\tm_1}\\
&= (\tdbarb [(1+\chi a)^{-1} u], \alpha)_{\tm_1} + ((1+\chi a)^{-2} \hat{\Zbar} (\chi a) u, i_{\hobar} \alpha)_{\tm_1} \\
&= ((1+\chi a)^{-1} u, \tvdbarb \alpha)_{\tm_0} + (u, (1+\chi a)^{-2} [\hZ (\chi a)] i_{\hobar} \alpha)_{\tm_1} \\
&= (u, \tvdbarb \alpha)_{\hm_0} + (u, |\psi|^2 G^2 (1+\chi a)^{-2} [\hZ (\chi a)] i_{\hobar} \alpha)_{\hm_0}.
\end{align*}
(The third equality uses the remark after the proof of Proposition~\ref{prop1}.) Thus 
\begin{equation} \label{eq:tdbarb*hdbarb*smooth}
\hdbarb^* \alpha = \tvdbarb \alpha - g \alpha.
\end{equation}

Next, suppose $\alpha \in \text{Dom}[\hdbarb^* \colon L^2_{(0,1)}(\hm_1) \to L^2(\hm_0)]$. Then there exists a sequence $\alpha_j \in C^{\infty}_{(0,1)}(\hM)$, such that 
$$
\alpha_j \to \alpha \text{ in $L^2_{(0,1)}(\hm_1)$}, \quad \text{and} \quad \hdbarb^* \alpha_j \to \hdbarb^* \alpha \text{ in $L^2(\hm_0)$}.
$$ 
But by (\ref{eq:tdbarb*hdbarb*smooth}),
$$
\tvdbarb \alpha_j = \hdbarb^* \alpha_j + g \alpha_j
$$ 
for all $j$, and 
$$
g \alpha_j \to g \alpha \text{ in $L^2(\hm_0)$}.
$$
Thus
$$
\tvdbarb \alpha_j \to \hdbarb^* \alpha + g \alpha
$$
as $j \to \infty$. This proves $\alpha \in \text{Dom}[\tdbarb^* \colon L^2_{(0,1)}(\tm_1) \to L^2(\tm_0)]$, and (\ref{eq:tdbarb*hdbarb*}) holds. Similarly one can prove the converse.
\end{proof}

Now by what we have just shown,
\begin{equation} \label{eq:compare_tboxb1}
\text{Dom}[\tboxb \colon L^2(\tm_0) \to L^2(\tm_0)] = \text{Dom}[\hboxb \colon L^2(\hm_0) \to L^2(\hm_0)],
\end{equation}
and for $u$ in this common domain of definition,
\begin{equation} \label{eq:compare_tboxb2}
\tboxb u = \hboxb u + g \hdbarb u.
\end{equation}
Thus from (\ref{eq:hNdef1}) and (\ref{eq:hNdef2}), we have 
$$
\hN \colon L^2(\tm_0) \to \text{Dom}(\tboxb) \subseteq L^2(\tm_0),
$$
with
\begin{equation} \label{eq:tboxb1}
\tboxb \hN + \hS = I + \hR \quad \text{on $L^2(\tm_0)$},
\end{equation}
where 
$$
\hR \colon L^2(\tm_0) \to L^2(\tm_0)
$$
is defined by
$$
\hR u = g \hdbarb \hN u \quad \text{for all $u \in L^2(\tm_0)$}.
$$
$\hR$ is not a pseudodifferential operator since $g$ is not necessarily smooth across $p$. Nevertheless, we have the following properties of $\hR$ (whose proof we defer towards the end):
\begin{prop} \label{prop:hR}
$(I+\hR)$ is invertible on $L^p(\tm_0)$ for all $p \in (1,\infty)$, and
$$
(I+\hR)^{-1} \colon L^p(\tm_0) \to L^p(\tm_0)
$$
is a bounded linear operator for all such $p$. Furthermore, $(I+\hR)^{-1}$ maps $\E{-1+\delta}$ into itself, and maps $\E{-3+\delta}$ into itself, for all $0 < \delta < 1$.
\end{prop}

We can now state the relationship between $\hS$ and $\tS$, and that between $\hN$ and $\tN$:

\begin{lem}
We have
\begin{equation} \label{eq:tShS}
\tS = \hS (I + \hR)^{-1}  \quad \text{on $L^2(\tm_0)$},
\end{equation}
and
\begin{equation} \label{eq:tNhN}
\tN = (I - \tS) \hN (I + \hR)^{-1}  \quad \text{on $L^2(\tm_0)$}.
\end{equation}
\end{lem}

\begin{proof}
By (\ref{eq:tboxb1}), and the invertibility of $I+\hR$ on $L^2(\tm_0)$, we have
\begin{equation} \label{eq:tboxb2}
\tboxb [\hN (I + \hR)^{-1}] + [\hS (I + \hR)^{-1}] = I \quad \text{on $L^2(\tm_0)$}.
\end{equation}
Now
\begin{align*}
\text{kernel}(\hboxb \colon L^2(\hm_0) \to L^2(\hm_0)) 
&= \text{kernel}(\hdbarb \colon L^2(\hm_0) \to L^2(\hm_0)) \\
= \text{kernel}(\tdbarb \colon L^2(\tm_0) \to L^2(\tm_0))
&= \text{kernel}(\tboxb \colon L^2(\tm_0) \to L^2(\tm_0)).
\end{align*}
(The first identity is well-known; the second follows from the identity of the operators $\hdbarb \colon L^2(\hm_0) \to L^2(\hm_0)$ and $\tdbarb \colon L^2(\tm_0) \to L^2(\tm_0)$; for the last identity, see argument after (\ref{eq:tdbarb^2u=utboxb}).) Thus for all $u \in L^2(\tm_0)$, we have
\begin{equation} \label{eq:tSproof1}
[\hS (I + \hR)^{-1}] u \in \text{kernel}[\tboxb \colon L^2(\tm_0) \to L^2(\tm_0)].
\end{equation}
Furthermore,
\begin{equation} \label{eq:tSproof2}
\tboxb [\hN (I + \hR)^{-1}] u \perp \text{kernel}[\tboxb \colon L^2(\tm_0) \to L^2(\tm_0)] \quad \text{in $L^2(\tm_0)$}.
\end{equation}
In fact, if $v \in \text{kernel}[\tboxb \colon L^2(\tm_0) \to L^2(\tm_0)]$, then
$$
(\tboxb [\hN (I + \hR)^{-1}] u, v)_{\tm_0} = (\tdbarb [\hN (I + \hR)^{-1}] u, \tdbarb v)_{\tm_1} = 0
$$
by Lemma~\ref{lem:tdbarb*adjoint}. Thus by (\ref{eq:tboxb2}), (\ref{eq:tSproof1}) and (\ref{eq:tSproof2}), we have $\hN (I + \hR)^{-1}$ being the orthogonal projection onto $\text{kernel}(\tboxb \colon L^2(\tm_0) \to L^2(\tm_0))$ in $L^2(\tm_0)$. (\ref{eq:tShS}) follows.

Next, by (\ref{eq:tShS}) and (\ref{eq:tboxb2}), we have
\begin{equation} \label{eq:tboxb3}
\tboxb [(I - \tS) \hN (I + \hR)^{-1}] + \tS = I \quad \text{on $L^2(\tm_0)$}.
\end{equation}
Now if $u \in L^2(\tm_0)$ is in the kernel of $\tboxb \colon L^2(\tm_0) \to L^2(\tm_0))$ in $L^2(\tm_0)$, then writing $v = [(I - \tS) \hN (I + \hR)^{-1}]u$, we have 
$$
\tboxb v = 0, \quad \text{and} \quad v \perp \text{kernel}[\tboxb \colon L^2(\tm_0) \to L^2(\tm_0)] \quad \text{in $L^2(\tm_0)$}.
$$
Thus $v = 0$, i.e. $[(I - \tS) \hN (I + \hR)^{-1}]u = 0 = \tN u$.

On the other hand, if $u \in L^2(\tm_0)$ is orthogonal to $\text{kernel}[\tboxb \colon L^2(\tm_0) \to L^2(\tm_0)]$ in $L^2(\tm_0)$, then writing $v = [(I - \tS) \hN (I + \hR)^{-1}]u$ again, we have
$$
\tboxb v = u, \quad \text{and} \quad v \perp \text{kernel}[\tboxb \colon L^2(\tm_0) \to L^2(\tm_0)] \quad \text{in $L^2(\tm_0)$}.
$$
Thus by definition of $\tN$, we have $v = \tN u$, i.e. $[(I - \tS) \hN (I + \hR)^{-1}]u = 0 = \tN u$. Together with what we proved above, (\ref{eq:tNhN}) follows.
\end{proof}

We are now ready to prove Proposition~\ref{prop2}.

\begin{proof}[Proof of Proposition~\ref{prop2}]
By Proposition~\ref{prop:hR}, 
$$(I+\hR)^{-1} \colon L^p(\tm_0) \to L^p(\tm_0)  \quad \text{for all $p \in (1,\infty)$}.$$
Also, by (\ref{eq:hSLp}) and (\ref{eq:hN_L43}) 
$$\hS \colon L^p(\tm_0) \to L^p(\tm_0)  \quad \text{for all $p \in (1,\infty)$},$$
$$\hN \colon L^{4/3}(\tm_0) \to L^{4}(\tm_0).$$
Thus by (\ref{eq:tShS}) and (\ref{eq:tNhN}),
$$\tS \colon L^p(\tm_0) \to L^p(\tm_0)  \quad \text{for all $p \in (1,\infty)$},$$
$$\tN \colon L^{4/3}(\tm_0) \to L^{4}(\tm_0).$$
Next, by (\ref{eq:tShS}) and (\ref{eq:tNhN}),
\begin{align*}
\hnabla_b \tN 
&= \hnabla_b [I - \hS (I + \hR)^{-1}] \hN (I+ \hR)^{-1} 
\end{align*}
Since $(I+\hR)^{-1}$ preserves $L^{4/3}(\tm_0)$, it suffices to show 
$$
\hnabla_b [I - \hS (I + \hR)^{-1}] \hN \colon L^{4/3}(\tm_0) \to L^2(\tm_0).
$$
By (\ref{eq:hN_D_L43}),
$$
\hnabla_b \hN \colon L^{4/3}(\tm_0) \to L^2(\tm_0).
$$
Writing 
$$
(I+\hR)^{-1} = I - \hR (I + \hR)^{-1},
$$
we have
\begin{equation} \label{eq:hnablabhN}
\hnabla_b \hS (I + \hR)^{-1} \hN
= \hnabla_b \hS \hN - \hnabla_b \hS \hR (I + \hR)^{-1} \hN.
\end{equation}
But by (\ref{eq:hShN_D_L43}), we can bound the first term of (\ref{eq:hnablabhN}):
$$
\hnabla_b \hS \hN \colon L^{4/3}(\tm_0) \to L^2(\tm_0).
$$
Finally, by Corollary~\ref{cor:nablab_comm} applied to $T_0 = \hS$, we can write
$$
\hnabla_b \hS = T_0' \hnabla_b
$$
for some operators $T_0'$ of order 0 (note $\hS \, 1 = 0$). Thus
the second term of (\ref{eq:hnablabhN}) satisfies
$$
\hnabla_b \hS \hR (I + \hR)^{-1} \hN 
= T_0' \hnabla_b \hR (I + \hR)^{-1} \hN. 
$$
But 
$$
\hnabla_b \hR = (\hnabla_b g) \hdbarb \hN + g \hnabla_b \hdbarb \hN \colon L^4(\tm_0) \to L^4(\tm_0),
$$
(here we use the fact that $\hnabla_b g \in L^{\infty}$, which one can check), and
$$
(I + \hR)^{-1} \hN \colon L^{4/3}(\tm_0) \to L^4(\tm_0).
$$
Thus the second term of (\ref{eq:hnablabhN}) maps $L^{4/3}(\tm_0)$ into $L^4(\tm_0) \subseteq L^2(\tm_0)$. Altogether, 
$$
\hnabla_b \tN \colon L^{4/3}(\tm_0) \to L^2(\tm_0),
$$
as desired.

Finally, the last part of Proposition~\ref{prop2} is a special case of Proposition~\ref{prop4}. We defer its proof until we prove Proposition~\ref{prop4}.
\end{proof}

We note in passing that now we have shown
\begin{equation} \label{eq:tShSLp}
\tS = \hS (I + \hR)^{-1}  \quad \text{on $L^{4/3}(\tm_0)$},
\end{equation}
and
\begin{equation} \label{eq:tNhNLp}
\tN = [I - \hS (I + \hR)^{-1}] \hN (I + \hR)^{-1}  \quad \text{on $L^{4/3}(\tm_0)$}
\end{equation}
(not just on $L^2(\tm_0)$).

We are now ready to prove Proposition~\ref{prop3}.

\begin{proof}[Proof of Proposition~\ref{prop3}]
Recall that 
$$
\tf = (1+ \chi a) \psibar G^4 f = (1+ \chi a) \psibar G^4 \tboxb \tb
$$
by definitions of $\tf$ and $f$. Thus by (\ref{eq:boxbtboxb}), we have
$$
\tf = \tboxb (\psi^{-1} \tb) = \tvdbarb (\psi^{-1} \tdbarb \tb).
$$
Now
$$
\tb = \beta_0 + \beta_1, \quad \beta_0 = \chi \frac{i \zbar}{|z|^2 - it}, \quad \beta_1 \in \E{1},
$$
where $(z,t)$ is the CR normal coordinate around $p$. Since
$$
\psi^{-1} = \frac{1}{2\pi (|z|^2 + it)} + \E{2},
$$
$$
\tdbarb \beta_0 = i \frac{|z|^2 + it}{(|z|^2-it)^2} \hobar + \E{2} \hobar,
$$
together with $\beta_1 \in \E{1}$, we have $$\psi^{-1} \tdbarb \tb \in \E{-4} \hobar;$$ in fact,
$$
\psi^{-1} \tdbarb \tb = \frac{i}{2\pi} \frac{1}{(|z|^2-it)^2} \hobar + \E{-2} \hobar.
$$
Now let 
$$
\tm_0 = V dz d\zbar dt \quad \text{in $B_{\varepsilon_0}(p)$},
$$
so that 
$$
V = 1 + \E{2}.
$$
Let 
$$
\gamma_0 = 2 \pi i \chi V^{-1} \psibar^{-2} \in \E{-4}.
$$
Then
$$
\gamma_0 = \chi \frac{i}{2 \pi} (1 + \E{2})  \left[ \frac{1}{(|z|^2 - it)^2} + \E{0} \right] = \chi \frac{i}{2 \pi} \frac{1}{(|z|^2-it)^2} + \E{-2}.
$$
Thus $\psi^{-1} \tdbarb \tb$ matches $\gamma_0 \hobar$ up to $\E{-2}$. This motivates us to write
$$
\tf = \tvdbarb \alpha + \tvdbarb (\gamma_0 \hobar)
$$
where
$$
\alpha := \psi^{-1} \tdbarb \tb - \gamma_0 \hobar \in \E{-2}.
$$
We will show separately that 
\begin{equation} \label{eq:tSkernel1}
\tS (\tvdbarb \alpha) = 0
\end{equation}
and
\begin{equation} \label{eq:tSkernel2}
\tS [\tvdbarb (\gamma_0 \hobar)] = 0.
\end{equation}
If both of these are true, then 
$$
\tS \tf = 0
$$
as desired.

First we show (\ref{eq:tSkernel1}). Note $\tvdbarb \alpha \in \E{-3} \subseteq L^p(\tm_0)$ for all $1 < p < 4/3$. Thus we can pick a sequence of cut-off functions $\phi_j$, with each $\phi_j$ identically equal to 1 away from $p$, and equal to zero near $p$, such that 
$$
\phi_j \tvdbarb \alpha \to \tvdbarb \alpha \quad \text{in $L^p(\tm_0)$}.
$$
Now
$$
\phi_j \tvdbarb \alpha = \tvdbarb (\phi_j \alpha) + O((\hnabla_b \phi_j ) \alpha)
$$
and $\alpha \in L^{p^*}(\tm_0)$ where $1/p^* = 1/p - 1/4$. We may thus choose $\phi_j$ so that
$$
O((\hnabla_b \phi_j ) \alpha) \to 0 \quad \text{in $L^p(\tm_0)$}
$$
as well, and then
$$
\tvdbarb (\phi_j \alpha) \to \tvdbarb \alpha \quad \text{in $L^p(\tm_0)$}.
$$
It follows by continuity of $\tS$ on $L^p(\tm_0)$ that 
$$
\tS \tvdbarb (\phi_j \alpha) \to \tS \tvdbarb \alpha \quad \text{in $L^p(\tm_0)$}.
$$
But 
$$
\tS \tvdbarb (\phi_j \alpha) = \tS \tdbarb^* (\phi_j \alpha) = 0
$$
for all $j$, since $\phi_j \alpha$ is smooth on $\hM$, and is in $L^2_{(0,1)}(\tm_1)$. Hence (\ref{eq:tSkernel1}) follows.

Next, we to prove (\ref{eq:tSkernel2}), let $\cZ$ be the adjoint of $\hZbar$ under $L^2(dz d\zbar dt)$. Then since 
$$
\hZbar = \frac{\partial}{\partial \zbar} + iz \frac{\partial}{\partial t} + \O{4} \frac{\partial}{\partial z} + \O{4} \frac{\partial}{\partial \zbar} + \O{5} \frac{\partial}{\partial t},
$$
we have
$$
\cZ = -\hZ + s, \quad s \in \O{3}.
$$
Now
$$
\tvdbarb (\gamma_0 \hobar) = V^{-1} \cZ( \gamma_0 V ) = V^{-1} (- \hZ + s ) (2\pi i \chi \psibar^{-2}) = 2\pi i  V^{-1} [-(\hZ \chi) + s \chi] \psibar^{-2}.
$$
Here we used $\hZ \psibar = 0$, which holds by construction of $\psi$. Near $p$ we have $\chi = 1$, and so there
\begin{equation} \label{eq:tvdbarbgamma0}
\tvdbarb (\gamma_0 \hobar) = 2\pi i  V^{-1} \O{3} \psibar^{-2}.
\end{equation}
In particular, 
$$
\tvdbarb (\gamma_0 \hobar) \in \E{-1} \subseteq L^2(\tm_0).
$$
To compute $\tS [\tvdbarb (\gamma_0 \hobar)]$, note that for any $1 < p < 4/3$,
\begin{equation} \label{eq:tvdbarbgamma0conv}
\tvdbarb [2\pi i \chi V^{-1} \psibar^{-1} (\psibar + \delta)^{-1} \hobar] \to \tvdbarb (\gamma_0 \hobar) \quad \text{in $L^p(\hm_0)$ as $\delta \to 0$}.
\end{equation}
In fact, 
\begin{align*} 
\tvdbarb [2\pi i  \chi V^{-1} \psibar^{-1} (\psibar + \delta)^{-1}  \hobar] 
&= 2\pi i  V^{-1} (-\hZ + s)[\chi \psibar^{-1} (\psibar + \delta)^{-1}] \\
&= 2\pi i  V^{-1} [-(\hZ \chi) + s \chi] \psibar^{-1} (\psibar + \delta)^{-1}.
\end{align*}
(\ref{eq:tvdbarbgamma0conv}) then follows by comparing this with the corresponding expression for $\tvdbarb(\gamma_0 \hobar)$, since now $V^{-1} \in L^{\infty}$, $-(\hZ \chi) + s \chi \in \E{3}$, and 
$$
|\psibar^{-1} (\psibar + \delta)^{-1} - \psibar^{-2}| \leq C \delta \hrho^{-6}.
$$
From (\ref{eq:tvdbarbgamma0conv}), it follows that
$$
\tS \tvdbarb [2\pi i \chi V^{-1} \psibar^{-1} (\psibar + \delta)^{-1} \hobar] \to \tS \tvdbarb (\gamma_0 \hobar) \quad \text{in $L^p(\hm_0)$ as $\delta \to 0$}.
$$
But
$$
2\pi i \chi V^{-1} \psibar^{-1} (\psibar + \delta)^{-1} \hobar \in \E{-2} \subseteq L^{p^*}_{(0,1)}(\hm_0)
$$
for all $\delta > 0$. In particular, 
$$
\tS \tvdbarb [2\pi i \chi V^{-1} \psibar^{-1} (\psibar + \delta)^{-1} \hobar] = 0.
$$
Thus (\ref{eq:tSkernel2}) follows, and we are done.
\end{proof}

We can also prove Proposition~\ref{prop4}:

\begin{proof}[Proof of Proposition~\ref{prop4}]
This follows from (\ref{eq:tNhNLp}), Proposition~\ref{prop:regularity_in_E} and Proposition~\ref{prop:hR}. 
\end{proof}

It remains to prove Proposition \ref{prop:hR}. To prove the first part of Proposition~\ref{prop:hR}, we first claim that $\hdbarb \hN$ is bounded from $L^2(\hm_0)$ to $L^2_{(0,1)}(\hm_1)$ uniformly in $\varepsilon_0$. In fact, note that $\hm_0$ is defined independent of $\varepsilon_0$, and $\hm_1$ remains comparable to each other as $\varepsilon_0$ varies. So the operator $\hdbarb \colon L^2(\hm_0) \to L^2(\hm_1)$ is independent of $\varepsilon_0$ (in particular, so is its kernel). Suppose now $u$ is a function in $L^2(\hm_0)$ that is orthogonal to the kernel of $\hdbarb \colon L^2(\hm_0) \to L^2(\hm_1)$. (This is a condition independent of $\varepsilon_0$.) We then have
$$
\|u\|_{L^2(\hm_0)} \leq C \|\hdbarb u\|_{L^2_{(0,1)}(\hm_1)}.
$$ 
The constant $C$ can be chosen independent of $\varepsilon_0$. It follows that for such $u$,
$$
\|u\|_{L^2(\hm_0)} \leq C^2 \|\hboxb u\|_{L^2(\hm_0)}
$$ 
with $C$ independent of $\varepsilon_0$. Now take $u = \hN v$ for $v \in L^2(\hm_0)$. Then 
$$
\|\hN v\|_{L^2(\hm_0)} \leq C^2 \|v\|_{L^2(\hm_0)}.
$$
Thus
\begin{align*}
\|\hdbarb \hN v\|_{L^2_{(0,1)}(\hm_1)}^2 
&= (\hdbarb \hN v, \hdbarb \hN v)_{\hm_1} \\
&= (\hN v, \hboxb \hN v)_{\hm_0} \\
&= (\hN v, (I-\hS) v)_{\hm_0} \\
&\leq C^2 \|v\|_{L^2(\hm_0)} \|(I-\hS) v\|_{L^2(\hm_0)} \\
&\leq C^2 \|v\|_{L^2(\hm_0)}^2.
\end{align*}
This is true for all $v \in L^2(\hm_0)$ with constant $C$ independent of $\varepsilon_0$. Hence the claim.

Recall now $\hR = g \hdbarb \hN$. It now follows that 
\begin{equation} \label{eq:epsilon0_choice}
\|\hR\|_{L^2(\tm_0) \to L^2(\tm_0)} \leq C \varepsilon_0 \leq \frac{1}{2}
\end{equation}
if $\varepsilon_0$ is sufficiently small. We fix from now on such $\varepsilon_0$. Then we can invert $I+\hR$ on $L^2$ by a Neumann series: if $f \in L^2(\tm_0)$, then
$$
(I - \hR + \hR^2 - \hR^3 + \dots) f
$$
converges in $L^2(\tm_0)$, and $I+\hR$ of this limit is $f$.
It follows that
$$
(I+\hR)^{-1} = I - \hR + \hR^2 - \hR^3 + \dots
$$
when acting on functions in $L^2(\tm_0)$.
In particular, we have both
\begin{equation} \label{eq:IhR-11}
(I+\hR)^{-1} = I - \hR + (I+\hR)^{-1} \hR^2
\end{equation}
and
\begin{equation} \label{eq:IhR-12}
(I+\hR)^{-1} = I - \hR + \hR^2 (I+\hR)^{-1} 
\end{equation}
when acting on $L^2(\tm_0)$.
Now we extend $(I+\hR)^{-1}$ to $L^p(\tm_0)$ for all $p \in (1,\infty)$: we divide into 2 cases.

\noindent{\textbf{Case 1:} $1 < p < 2$. Then the right hand side of (\ref{eq:IhR-11}) extends to a bounded linear operator $L^p(\tm_0) \to L^p(\tm_0)$, since 
$$
\hR^2 \colon L^p(\tm_0) \to L^2(\tm_0)
$$
and
$$
(I+\hR)^{-1} \colon L^2(\tm_0) \to L^2(\tm_0) \hookrightarrow L^p(\tm_0).
$$
Thus $I+\hR$ is invertible on $L^p(\tm_0)$ in this case.

\noindent{\textbf{Case 2:} $2 < p < \infty$. Then the right hand side of (\ref{eq:IhR-12}) extends to a bounded linear operator $L^p(\tm_0) \to L^p(\tm_0)$, since $L^p(\tm_0) \hookrightarrow L^2(\tm_0)$, 
$$
(I+\hR)^{-1} \colon L^2(\tm_0) \to L^2(\tm_0),
$$
and
$$
\hR^2 \colon L^2(\tm_0) \to L^p(\tm_0).
$$
Thus $I+\hR$ is invertible on $L^p(\tm_0)$ in this case as well.

Thus $I+\hR$ is invertible on $L^p(\tm_0)$ for all $p \in (1,\infty)$, and $(I+\hR)^{-1}$ is a bounded linear operator on $L^p(\tm_0)$ for all such $p$. This proves the first part of Proposition~\ref{prop:hR}. (Note, on the other hand, that the above argument does not claim uniformity of the norm, nor kernel estimates, of $\hdbarb \hN$ as an operator from $L^p(\tm_0)$ to itself as $\varepsilon_0$ varies. This is ok since $\varepsilon_0$ has been fixed already. In what follows, all kernel estimates, cancellation conditions, etc all depend on this fixed $\varepsilon_0$.)

The second part of Proposition~\ref{prop:hR} is the analog of Theorem 3.1 in \cite{MR3135073}. First, we have the following lemma:

\begin{lem} \label{lem:hRLinfty}
For every positive integer $k$, there exists a constant $C_k$ such that
$$\hrho^{k} \hnabla_b^{k} \hR^{2k} \colon L^{\infty} \to L^{\infty}$$ is bounded with norm $\leq C_k$.
\end{lem}

Now let's assume the lemma, and prove the second part of Proposition~\ref{prop:hR}. Let's write $(I+\hR)_{L^p}^{-1}$ for the inverse of $I+\hR$ on $L^p(\tm_0)$. Then for all $1 < p < \infty$ and all positive integers $k$,
$$
(I+\hR)_{L^p}^{-1} = [I - \hR + \hR^2 - \dots - \hR^{2k+3}] + \hR^{2k} \hR^4 (I+\hR)_{L^p}^{-1}.
$$
Now suppose $f \in \E{-\delta}$ for some $0 < \delta < 4$. Then $f \in L^p(\tm_0)$ for some $1 < p < \infty$, so we can apply the above identity. But on the right hand side,
$$
\hR^4 (I+\hR)_{L^p}^{-1} f \in L^{\infty},
$$ 
because $(I+\hR)_{L^p}^{-1} f \in L^p$, and $\hR^4 \colon L^p \to L^{\infty}$ for all $1 < p < \infty$. (Remember $\hR \colon L^p \to L^{p^*}$ if $1 < p < 4$, $1/p^* = 1/p - 1/4$, and $\hR \colon L^4 \to L^q$ for any $q < \infty$; also $\hR \colon L^p \to L^{\infty}$ if $4 < p < \infty$.)
Hence by Lemma~\ref{lem:hRLinfty}, 
$$
|\hnabla_b^k \hR^{2k} \hR^4 (I+\hR)_{L^p}^{-1} f| \lesssim_k \hrho^{-k} \lesssim \hrho^{-(k+\delta)}.
$$
But we also have
$$
|\hnabla_b^k  (I - \hR + \hR^2 - \dots - \hR^{2k+3}) f| \lesssim_k \hrho^{-(k+\delta)},
$$
since at the very least, $\hR \colon \E{-\delta} \to \E{-\delta}$. Thus 
$$
|\hnabla_b^k  (I+\hR)_{L^p}^{-1} f| \lesssim_k \hrho^{-(k+\delta)},
$$
and since this is true for all $k$, we see that 
$$
(I+\hR)_{L^p}^{-1} f \in \E{-\delta}.
$$
This completes the proof of Proposition~\ref{prop:hR}.

%We have now completed the proof of Theorem~\ref{thm:tboxb}, thus the proof of Theorem~\ref{thm:goal}.
It remains to prove Lemma~\ref{lem:hRLinfty}. To do so, we need a couple lemmas:

\begin{lem} \label{lem:dT2}
If $g \in \E{1}$ and $T_{-1} \in \Psi^{-1}_{\hD}(\hM)$, then 
$$\hnabla_b (g T_{-1})^2 \colon L^{\infty} \to L^{\infty}.$$
\end{lem}

\begin{proof}
This is because $T_{-1} \colon L^{\infty} \to NL^{1,p}$ for all $p < \infty$, where $NL^{1,p}$ is the non-isotropic Sobolev spaces consisting of functions $F$ on $\hM$ for which both $F$, $\hnabla_b F \in L^p$. Also, multiplication by $g$ preserves $NL^{1,p}$. Then $\hnabla_b (g T_{-1}) \colon NL^{1,p} \to NL^{1,p}$, which embeds back into $L^{\infty}$.
\end{proof}

Next, we need a number of commutation relations: 

\begin{lem} \label{lem:com1}
If $T_{-1} \in \Psi^{-1}_{\hD}(\hM)$, then 
$$
\hnabla_b T_{-1} = T'_{-1} \hnabla_b + T_{-\infty}
$$
for some other $T'_{-1} \in \Psi^{-1}_{\hD}(\hM)$, and some $T_{-\infty} \in \Psi^{-\infty}_{\hD}(\hM)$.
\end{lem}

\begin{proof}
This follows from Corollary~\ref{cor:nablab_comm}.
\end{proof}

Also, if $\eta \in C^{\infty}$ in a neighborhood of $p$ and $k \in \mathbb{N}$, we say that $\eta$ vanishes to non-isotropic order at least $k$ at $p$, if 
$$
|\eta(x)| \lesssim \hat{d}(p,x)^k
$$
for all $x \in \hM$, where $\hat{d}$ is the metric on $\hM$ induced by $\hthe$. Equivalently, we could replace $\hat{d}$ above by the quasi-distance $d$ on $\hM$ determined by the contact distribution $\hD$, since $d$ and $\hat{d}$ are comparable. If a function is $C^{\infty}$ near $p$, then it vanishes to non-isotropic order at least 1 at $p$, if and only if it vanishes at $p$; and it vanishes to non-isotropic order at least 2 at $p$, if and only if both the function and its subelliptic gradient vanishes at $p$. Alternatively, if we choose a coordinate system near $p$ such that the coordinates of $p$ is $(0,0,0)$ and $\hD$ is spanned by $\frac{\partial}{\partial x^1}$ and $\frac{\partial}{\partial x^2}$ at $p$, then a function $\eta$ on $\hM$ vanishes to non-isotropic order at least 2 at $p$, if and only if there exists $C^{\infty}$ functions $h_{11}, h_{12}, h_{22}$ and $h_3$ in a sufficiently small neighborhood $U$ of $p$ such that 
$$
\eta(x) = \sum_{1 \leq i \leq j \leq 2} h_{ij}(x) x^i x^j + h_3(x) x^3
$$
for all $x \in U$. We denote the set of all $C^{\infty}(\hM)$ functions that vanishes to non-isotropic order at least $k$ at $p$ by $O^k_p$.

\begin{lem} \label{lem:com2}
If $\eta_1 \in O^1_p$, and $T_{-1} \in \Psi^{-1}_{\hD}(\hM)$, then 
$$
[\eta_1,T_{-1}] = T_{-2}
$$
for some other $T_{-2} \in \Psi^{-2}_{\hD}(\hM)$.
\end{lem}

\begin{proof}
This is a consequence of Theorem~\ref{thm:comm}; in fact this holds without having to assume that $\eta_1$ vanishes at $p$.
\end{proof}

\begin{lem} \label{lem:com3}
If $\eta_2 \in O^2_p$, and $T_{-1} \in \Psi^{-1}_{\hD}(\hM)$, then
$$
[\eta_2,T_{-1}] = \eta_1 T_{-2} + T_{-3}
$$
for some $\eta_1 \in O^1_p$, $T_{-2} \in \Psi^{-2}_{\hD}(\hM)$, and  $T_{-3} \in \Psi^{-3}_{\hD}(\hM)$.
\end{lem}

\begin{proof}
Let $U$ be a sufficiently small open neighborhood of $p$, so that we can choose a frame $\hX_1, \hX_2, \hX_3$ of $T\hM$ on $U$, with $\hX_1$ and $\hX_2$ spanning $\hD$ on $U$. Choose a coordinate system on $U$ such that the coordinates of $p$ is $(0,0,0)$, and
$$
\hX_i = \sum_{j=1}^3 A_i^j(x) \frac{\partial}{\partial x^j} \quad \text{on $U$},
$$
with $A_i^j(0) = \delta_i^j$ for $1 \leq i, j \leq 3$. Now suppose $\eta_2 \in O^2_p$ is compactly supported in $U$. Let $\chi_0 \in C^{\infty}_c(U)$ be identically 1 on the support of $\eta_2$. Then $\eta_2 = \chi_0 \eta_2$ can be decomposed on $\hM$ as
$$
\eta_2(x) = \sum_{1 \leq i \leq j \leq 2} \chi_0(x) h_{ij}(x) x^i x^j + \chi_0(x) h_3(x) x^3
$$
for some $h_{11}, h_{12}, h_{22}, h_3 \in C^{\infty}(U)$. To study $[\eta_2, T_{-1}]$ where $T_{-1} \in \Psi^{-1}_{\hD}(\hM)$, it suffices to study the commutator of each piece of $\eta_2$ with $T_{-1}$. But 
$$
[\eta_1 \eta_1', T_{-1}] = \eta_1 [\eta_1', T_{-1}] + \eta_1' [\eta_1, T_{-1}] + [[\eta_1,T_{-1}], \eta_1']
$$
is of the form $O^1_p T_{-2} + T_{-3}$, where $T_{-2} \in \Psi^{-2}_{\hD}(\hM)$ and $T_{-3} \in \Psi^{-3}_{\hD}(\hM)$; we will show that the same is true for $[h x^3, T_{-1}]$, if $h \in C^{\infty}_c(U)$. 

Indeed, let $\chi \in C^{\infty}_c(U)$ be identically 1 on the support of $h$, and $\tilde{\chi} \in C^{\infty}_c(U)$ be identically 1 on the support of $\chi$. Then 
$$
[h x^3, T_{-1}] = [h x^3, T_{-1} \chi] \quad \text{(mod $\Psi^{-\infty}_{\hD}(\hM))$}
$$
since $[h x^3, T_{-1} (1-\chi)] = h x^3 T_{-1} (1-\chi) - T_{-1} h x^3 (1-\chi) = h x^3 T_{-1} (1-\chi) \in \Psi^{-\infty}_{\hD}(\hM)$. Furthermore, 
$$
[h x^3, T_{-1} \chi] = \tilde{\chi} [h x^3, T_{-1} \chi]  \quad \text{(mod $\Psi^{-\infty}_{\hD}(\hM))$},
$$
since $(1-\tilde{\chi}) [h x^3, T_{-1} \chi] = (1-\tilde{\chi}) h x^3 T_{-1} \chi - (1-\tilde{\chi}) T_{-1} \chi h x^3 = - (1-\tilde{\chi}) T_{-1} \chi h x^3 \in \Psi^{-\infty}_{\hD}(\hM)$. 
Now
$$
\tilde{\chi} [h x^3, T_{-1} \chi]  = \tilde{\chi} [x^3, T_{-1} \chi] h  + \tilde{\chi} x^3 [h, T_{-1} \chi], 
$$
the second term being of the form $O^1_p T_{-2}$ already. We will analyse the first term by writing down its integral kernel. To do so, we define $$\Theta_0(x,y) = L_x(x-y)$$ on $U \times U$ using the coefficients $(A_i^j(x))$ as in Section~\ref{sect:pdo}, i.e. we let the $k$-th coordinate of $\Theta_0(x,y)$ be $\Theta_0(x,y)^k = \sum_{j=1}^3 B_j^k(x) (x^j-y^j)$ for $k = 1, 2, 3$, where $(B_j^k(x))$ is the inverse matrix of $(A_i^j(x))$. Then the quasi-distance $d$ determined by $\hD$ is given on $U$ by
$$
d(x,y) \simeq |x-y| + |\Theta_0(x,y)^3|^{1/2},
$$
and by Theorem~\ref{thm:equiv_smoothing}, there exists a kernel $k_0(x,u)$, satisfying
$$
|\partial_x^I \partial_u^{\gamma} k_0(x,u)| \lesssim_{I,\gamma,M} \|u\|^{-3-\|\gamma\|-M}
$$
for all multiindices $\gamma$, $I$, and all $M \geq 0$,
such that for any $f \in C^{\infty}_c(U)$ and any $x \in U$, we have
$$
T_{-1} f(x) = \int_{U} f(y) k_0(x,\Theta_0(x,y)) dy.
$$
It follows that the integral kernel of $\tilde{\chi} [x^3, T_{-1} \chi] h$ is
$$
\tilde{\chi}(x) (x^3-y^3) k_0(x,\Theta_0(x,y)) \chi(y) \tilde{h}(y).
$$
Writing
$$
x^3-y^3 = \sum_{i=1}^3 A_i^3(x) \Theta_0(x,y)^i, 
$$
we see that the integral kernel of $\tilde{\chi} [x^3, T_{-1} \chi] h$ is
$$
\sum_{i=1}^3 \tilde{\chi}(x) A_i^3(x) \Theta_0(x,y)^i k_0(x,\Theta_0(x,y)) \chi(y) \tilde{h}(y).
$$
By Theorem~\ref{thm:equiv_smoothing} again, the term $i = 3$ is the integral kernel of an operator in $\Psi^{-3}_{\hD}(\hM)$; on the other hand, since $A_i^3(0) = 0$ for $i = 1,2$, we see that the terms $i = 1,2$ are the integral kernels of an operator in $O^1_p \Psi^{-2}_{\hD}(\hM)$. Thus altogether, we have shown that $[hx^3,T_{-1}]$ is of the form $O^1_p T_{-2} + T_{-3}$, with $T_{-2} \in \Psi^{-2}_{\hD}(\hM)$ and $T_{-3} \in \Psi^{-3}_{\hD}(\hM)$, as desired.

Finally, we return to the general case, where $\eta_2 \in O^2_p$, but is not necessarily supported in $U$. We just have to fix a relatively compact open subset $U_0$ of $U$, and note that every $\eta_2 \in O^2_p$ can be written as the sum of two parts, one identically zero on $U_0$, another supported only in $U$; we have already seen that the commutator of the latter with $T_{-1}$ is of the desired form, so it remains to show that if $\eta \in C^{\infty}(\hM)$ vanishes identically on $U_0$, then $[\eta,T_{-1}]$ is of the desired form. To do so, let $\chi \in C^{\infty}_c(U_0)$ such that $\chi \equiv 1$ near $p$; note that the supports of $\eta$ and $\chi$ are disjoint. Then
$$
[\eta, T_{-1}] = (1-\chi) [\eta, T_{-1}] \quad \text{(mod $\Psi^{-\infty}_{\hD}(\hM))$},
$$
since 
$
\chi [\eta, T_{-1}] = \chi \eta T_{-1} - \chi T_{-1} \eta = - \chi T_{-1} \eta \in \Psi^{-\infty}_{\hD}(\hM).
$
But $1-\chi \in O^1_p$, and $[\eta, T_{-1}] \in \Psi^{-2}_{\hD}(\hM)$. As a result, $[\eta, T_{-1}]$ has the form $ O^1_p T_{-2} + T_{-3}$, as desired.
\end{proof}

We now return to the proof of Lemma~\ref{lem:hRLinfty}. Using the commatation relations as Lemma~\ref{lem:com1}, \ref{lem:com2} and \ref{lem:com3}, we have
$$
\eta_1\hnabla_b T_{-1} =T'_{-1} \eta_1 \hnabla_b + T''_{-1}
$$
$$
\eta_2 \hnabla_b^2 T_{-1} = T'_{-1} \eta_2 \hnabla_b^2 + T''_{-1} \eta_1 \hnabla_b + T'''_{-1}
$$
As a result, if $g \in \E{1}$, then
\begin{equation} \label{eq:comm1}
(\eta_1\hnabla_b) (g T_{-1})=(g T'_{-1}) (\eta_1 \hnabla_b) + g'' T''_{-1}
\end{equation}
\begin{equation} \label{eq:comm2}
(\eta_2 \hnabla_b^2 ) (g T_{-1}) = (g T'_{-1} ) (\eta_2 \hnabla_b^2) + (g'' T''_{-1} ) (\eta_1 \hnabla_b )+ (g''' T'''_{-1})
\end{equation}
for some functions $g'', g''' \in \E{1}$.

Using these, one can proceed as follows:

\begin{proof}[Proof of Lemma~\ref{lem:hRLinfty}]
Recall $\hR = g T_{-1}$ for some $g \in \E{1}$ and $T_{-1} \in \Psi^{-1}_{\hD}(\hM)$. Let $f \in L^{\infty}(\hM)$. We will bound $|\hrho(q)^k \hnabla_b^k (g T_{-1})^{2k} f(q)|$, uniformly for $q \in \hM$. 
If $q \in \hM$, either $\hrho(q) \simeq |\eta_1(q)|$ for some $\eta_1 \in O^1_p$, or $\hrho(q) \simeq |\eta_2(q)|^{1/2}$ for some $\eta_2 \in O^2_p$.

\noindent{\textbf{Case 1:}} $\hrho(q) \simeq |\eta_1(q)|$ for some $\eta_1 \in O^1_p$. Then we rewrite $\eta_1^k \hnabla_b^k (g T_{-1})^{2k}$ by first commuting the $\eta_1$'s through the $\hnabla_b$'s to form blocks of $\eta_1 \hnabla_b$, and then commute $\eta_1 \hnabla_b$ through the $g T_{-1}$'s using (\ref{eq:comm1}). We obtain
$$
\eta_1^k \hnabla_b^k (g T_{-1})^{2k} = (\eta_1 \hnabla_b  (g' T_{-1})^2 )^k + \text{better errors}
$$ 
so in this case, if $f \in L^{\infty}$, then by Lemma~\ref{lem:dT2},
$$
|\hrho^k(q) \hnabla_b^k  (g T_{-1})^{2k} f(q)| \leq C.
$$

\noindent{\textbf{Case 2:}} $\hrho \simeq |\eta_2|^{1/2}$ for some $\eta_1 \in O^2_p$. If $k$ is even, write $k = 2 \ell$, and bound $\eta_2^{\ell} \hnabla_b^k \hR^{2k}$ by writing it as
\begin{align*}
\eta_2^{\ell} \hnabla_b^{k} (g T_{-1})^{2k}  
&= (\eta_2 \hnabla_b^2 (g' T_{-1})^4 )^{\ell} + \text{better errors} \\
&= (\hnabla_b (g' T_{-1})^2 \eta_2 \hnabla_b (g' T_{-1})^2)^{\ell} + \text{better errors}.
\end{align*}
Here we have rewritten $\eta_2^{\ell} \hnabla_b^k (g T_{-1})^{2k}$ by first commuting the $\eta_2$'s through the $\hnabla_b$'s to form blocks of $\eta_2 \hnabla_b^2$, and then commute $\eta_2 \hnabla_b^2$ through the $g T_{-1}$'s using (\ref{eq:comm2}). Similarly, if $k$ is odd, write $k = 2 \ell + 1$, and bound $\eta_2^{\ell} \hnabla_b^{k} \hR^{2k}$ by writing it as
\begin{align*}
\eta_2^{\ell} \hnabla_b^{k} (g T_{-1})^{2k}  
&= \hnabla_b (g' T_{-1})^2 (\eta_2 \hnabla_b^2 (g' T_{-1})^4)^{\ell} + \text{better errors} \\
&= \hnabla_b (g' T_{-1})^2  (\hnabla_b (g' T_{-1})^2 \eta_2 \hnabla_b (g' T_{-1})^2)^{\ell} + \text{better errors}.
\end{align*}
For either parity of $k$, if $f \in L^{\infty}$, then by Lemma~\ref{lem:dT2},
$$
|\hrho^k(q) \hnabla_b^k  (g T_{-1})^{2k} f(q)| \leq C.
$$
This completes the proof of Lemma~\ref{lem:hRLinfty}.
\end{proof}


\begin{bibdiv}
\begin{biblist}
 
\bib{BG88}{book}{
   author={Beals, Richard},
   author={Greiner, Peter},
   title={Calculus on Heisenberg manifolds},
   series={Annals of Mathematics Studies},
   volume={119},
   publisher={Princeton University Press},
   place={Princeton, NJ},
   date={1988},
   pages={x+194},
%   isbn={0-691-08500-5},
%   isbn={0-691-08501-3},
%   review={\MR{953082 (89m:35223)}},
}

\bib{BoSh}{article}{
   author={Boas, Harold P.},
   author={Shaw, Mei-Chi},
   title={Sobolev estimates for the Lewy operator on weakly pseudoconvex
   boundaries},
   journal={Math. Ann.},
   volume={274},
   date={1986},
   number={2},
   pages={221--231},
%   issn={0025-5831},
%   review={\MR{838466 (87i:32029)}},
%   doi={10.1007/BF01457071},
}

\bib{MR0409893}{article}{
   author={Boutet de Monvel, L.},
   title={Int\'egration des \'equations de Cauchy-Riemann induites formelles},
   language={French},
   conference={
      title={S\'eminaire Goulaouic-Lions-Schwartz 1974--1975; \'Equations aux
      deriv\'ees partielles lin\'eaires et non lin\'eaires},
   },
   book={
      publisher={Centre Math., \'Ecole Polytech., Paris},
   },
   date={1975},
   pages={Exp. No. 9, 14},
%   review={\MR{0409893}},
}


\bib{BouSj76}{article}{
   author={Boutet de Monvel, L.},
   author={Sj{\"o}strand, J.},
   title={Sur la singularit\'e des noyaux de Bergman et de Szeg\H o},
   language={French},
   conference={
      title={Journ\'ees: \'Equations aux D\'eriv\'ees Partielles de Rennes
      (1975)},
   },
   book={
      publisher={Soc. Math. France},
      place={Paris},
   },
   date={1976},
   pages={123--164. Ast\'erisque, No. 34-35},
%   review={\MR{0590106 (58 \#28684)}},
}

\bib{MR2999315}{article}{
   author={Chanillo, Sagun},
   author={Chiu, Hung-Lin},
   author={Yang, Paul},
   title={Embeddability for 3-dimensional Cauchy-Riemann manifolds and CR
   Yamabe invariants},
   journal={Duke Math. J.},
   volume={161},
   date={2012},
   number={15},
   pages={2909--2921},
%   issn={0012-7094},
%   review={\MR{2999315}},
%   doi={10.1215/00127094-1902154},
}

\bib{MR3202474}{article}{
   author={Chanillo, Sagun},
   author={Chiu, Hung-Lin},
   author={Yang, Paul},
   title={Embedded three-dimensional CR manifolds and the non-negativity of
   Paneitz operators},
   conference={
      title={Geometric analysis, mathematical relativity, and nonlinear
      partial differential equations},
   },
   book={
      series={Contemp. Math.},
      volume={599},
      publisher={Amer. Math. Soc., Providence, RI},
   },
   date={2013},
   pages={65--82},
%   review={\MR{3202474}},
%   doi={10.1090/conm/599/11905},
}

\bib{MR3600060}{article}{
   author={Cheng, Jih-Hsin},
   author={Malchiodi, Andrea},
   author={Yang, Paul},
   title={A positive mass theorem in three dimensional Cauchy-Riemann
   geometry},
   journal={Adv. Math.},
   volume={308},
   date={2017},
   pages={276--347},
   issn={0001-8708},
%   review={\MR{3600060}},
%   doi={10.1016/j.aim.2016.12.012},
}

\bib{CS01}{book}{
   author={Chen, So-Chin},
   author={Shaw, Mei-Chi},
   title={Partial differential equations in several complex variables},
   series={AMS/IP Studies in Advanced Mathematics},
   volume={19},
   publisher={American Mathematical Society},
   place={Providence, RI},
   date={2001},
   pages={xii+380},
%   isbn={0-8218-1062-6},
%   review={\MR{1800297 (2001m:32071)}},
}
		
\bib{Ch88I}{article}{
   author={Christ, Michael},
   title={Regularity properties of the $\overline\partial_b$ equation on
   weakly pseudoconvex CR manifolds of dimension $3$},
   journal={J. Amer. Math. Soc.},
   volume={1},
   date={1988},
   number={3},
   pages={587--646},
%   issn={0894-0347},
%   review={\MR{928903 (89e:32027)}},
%   doi={10.2307/1990950},
}

\bib{Ch88II}{article}{
   author={Christ, Michael},
   title={Pointwise estimates for the relative fundamental solution of
   $\overline\partial_b$},
   journal={Proc. Amer. Math. Soc.},
   volume={104},
   date={1988},
   number={3},
   pages={787--792},
%   issn={0002-9939},
%   review={\MR{929407 (89d:35126)}},
%   doi={10.2307/2046793},
}

\bib{F}{article}{
   author={Fefferman, Charles},
   title={The Bergman kernel and biholomorphic mappings of pseudoconvex
   domains},
   journal={Invent. Math.},
   volume={26},
   date={1974},
   pages={1--65},
%   issn={0020-9910},
%   review={\MR{0350069 (50 \#2562)}},
}


\bib{FeKo88}{article}{
   author={Fefferman, C. L.},
   author={Kohn, J. J.},
   title={Estimates of kernels on three-dimensional CR manifolds},
   journal={Rev. Mat. Iberoamericana},
   volume={4},
   date={1988},
   number={3-4},
   pages={355--405},
%   issn={0213-2230},
%   review={\MR{1048582 (91h:32013)}},
%   doi={10.4171/RMI/78},
}

\bib{FoSt}{article}{
   author={Folland, G. B.},
   author={Stein, E. M.},
   title={Estimates for the $\bar \partial _{b}$ complex and analysis on
   the Heisenberg group},
   journal={Comm. Pure Appl. Math.},
   volume={27},
   date={1974},
   pages={429--522},
%   issn={0010-3640},
%   review={\MR{0367477 (51 \#3719)}},
}

\bib{MR2925386}{article}{
   author={Frank, Rupert L.},
   author={Lieb, Elliott H.},
   title={Sharp constants in several inequalities on the Heisenberg group},
   journal={Ann. of Math. (2)},
   volume={176},
   date={2012},
   number={1},
   pages={349--381},
%   issn={0003-486X},
%   review={\MR{2925386}},
%   doi={10.4007/annals.2012.176.1.6},
}
	

\bib{MR1831872}{article}{
   author={Gamara, Najoua},
   title={The CR Yamabe conjecture---the case $n=1$},
   journal={J. Eur. Math. Soc. (JEMS)},
   volume={3},
   date={2001},
   number={2},
   pages={105--137},
%   issn={1435-9855},
%   review={\MR{1831872 (2003d:32040a)}},
%   doi={10.1007/PL00011303},
}

\bib{MR1867895}{article}{
   author={Gamara, Najoua},
   author={Yacoub, Ridha},
   title={CR Yamabe conjecture---the conformally flat case},
   journal={Pacific J. Math.},
   volume={201},
   date={2001},
   number={1},
   pages={121--175},
%   issn={0030-8730},
%   review={\MR{1867895 (2003d:32040b)}},
%   doi={10.2140/pjm.2001.201.121},
}

\bib{MR975118}{article}{
   author={Graham, C. Robin},
   author={Lee, John M.},
   title={Smooth solutions of degenerate Laplacians on strictly pseudoconvex
   domains},
   journal={Duke Math. J.},
   volume={57},
   date={1988},
   number={3},
   pages={697--720},
%   issn={0012-7094},
%   review={\MR{975118 (90c:32031)}},
%   doi={10.1215/S0012-7094-88-05731-6},
}
		
	
\bib{GrSt}{book}{
   author={Greiner, P. C.},
   author={Stein, E. M.},
   title={Estimates for the $\overline \partial $-Neumann problem},
   note={Mathematical Notes, No. 19},
   publisher={Princeton University Press},
   place={Princeton, N.J.},
   date={1977},
   pages={iv+195},
%   isbn={0-691-08013-5},
%   review={\MR{0499319 (58 \#17218)}},
}

\bib{Hor85}{book}{
   author={H{\"o}rmander, Lars},
   title={The analysis of linear partial differential operators. III},
   series={Grundlehren der Mathematischen Wissenschaften [Fundamental
   Principles of Mathematical Sciences]},
   volume={274},
   note={Pseudodifferential operators},
   publisher={Springer-Verlag},
   place={Berlin},
   date={1985},
   pages={viii+525},
%   isbn={3-540-13828-5},
%   review={\MR{781536 (87d:35002a)}},
}

\bib{Hsiao08}{article}{
   author={Hsiao, Chin-Yu},
   title={Projections in several complex variables},
   language={English, with English and French summaries},
   journal={M\'em. Soc. Math. Fr. (N.S.)},
   number={123},
   date={2010},
   pages={131},
%   issn={0249-633X},
%   isbn={978-2-85629-304-1},
%   review={\MR{2780123 (2011m:32004)}},
}	


\bib{MR3135073}{article}{
   author={Hsiao, Chin-Yu},
   author={Yung, Po-Lam},
   title={The tangential Cauchy-Riemann complex on the Heisenberg group via
   conformal invariance},
   journal={Bull. Inst. Math. Acad. Sin. (N.S.)},
   volume={8},
   date={2013},
   number={3},
   pages={359--375},
%   issn={2304-7909},
%   review={\MR{3135073}},
}

\bib{MR3366852}{article}{
   author={Hsiao, Chin-Yu},
   author={Yung, Po-Lam},
   title={Solving the Kohn Laplacian on asymptotically flat CR manifolds of
   dimension 3},
   journal={Adv. Math.},
   volume={281},
   date={2015},
   pages={734--822},
%   issn={0001-8708},
%   review={\MR{3366852}},
%   doi={10.1016/j.aim.2015.04.028},
}

\bib{MR3308372}{book}{
   author={Jean, Fr\'ed\'eric},
   title={Control of nonholonomic systems: from sub-Riemannian geometry to
   motion planning},
   series={SpringerBriefs in Mathematics},
   publisher={Springer, Cham},
   date={2014},
   pages={x+104},
%   isbn={978-3-319-08689-7},
%   isbn={978-3-319-08690-3},
%   review={\MR{3308372}},
%   doi={10.1007/978-3-319-08690-3},
}

\bib{MR880182}{article}{
   author={Jerison, David},
   author={Lee, John M.},
   title={The Yamabe problem on CR manifolds},
   journal={J. Differential Geom.},
   volume={25},
   date={1987},
   number={2},
   pages={167--197},
%   issn={0022-040X},
%   review={\MR{880182 (88i:58162)}},
}

\bib{MR924699}{article}{
   author={Jerison, David},
   author={Lee, John M.},
   title={Extremals for the Sobolev inequality on the Heisenberg group and
   the CR Yamabe problem},
   journal={J. Amer. Math. Soc.},
   volume={1},
   date={1988},
   number={1},
   pages={1--13},
%   issn={0894-0347},
%   review={\MR{924699 (89b:53063)}},
%   doi={10.2307/1990964},
}
		
\bib{MR982177}{article}{
   author={Jerison, David},
   author={Lee, John M.},
   title={Intrinsic CR normal coordinates and the CR Yamabe problem},
   journal={J. Differential Geom.},
   volume={29},
   date={1989},
   number={2},
   pages={303--343},
%   issn={0022-040X},
%   review={\MR{982177 (90h:58083)}},
}		

\bib{Koe02}{article}{
   author={Koenig, Kenneth D.},
   title={On maximal Sobolev and H\"older estimates for the tangential
   Cauchy-Riemann operator and boundary Laplacian},
   journal={Amer. J. Math.},
   volume={124},
   date={2002},
   number={1},
   pages={129--197},
%   issn={0002-9327},
%   review={\MR{1879002 (2002m:32061)}},
}

\bib{Koh85}{article}{
   author={Kohn, J. J.},
   title={Estimates for $\bar\partial_b$ on pseudoconvex CR manifolds},
   conference={
      title={Pseudodifferential operators and applications (Notre Dame,
      Ind., 1984)},
   },
   book={
      series={Proc. Sympos. Pure Math.},
      volume={43},
      publisher={Amer. Math. Soc.},
      place={Providence, RI},
   },
   date={1985},
   pages={207--217},
%   review={\MR{812292 (87c:32025)}},
}
	
\bib{Koh86}{article}{
   author={Kohn, J. J.},
   title={The range of the tangential Cauchy-Riemann operator},
   journal={Duke Math. J.},
   volume={53},
   date={1986},
   number={2},
   pages={525--545},
%   issn={0012-7094},
%   review={\MR{850548 (87m:32041)}},
%   doi={10.1215/S0012-7094-86-05330-5},
}

\bib{KoRo}{article}{
   author={Kohn, J. J.},
   author={Rossi, Hugo},
   title={On the extension of holomorphic functions from the boundary of a
   complex manifold},
   journal={Ann. of Math. (2)},
   volume={81},
   date={1965},
   pages={451--472},
%   issn={0003-486X},
%   review={\MR{0177135 (31 \#1399)}},
}



\bib{MM07}{book}{
   author={Ma, Xiaonan},
   author={Marinescu, George},
   title={Holomorphic Morse inequalities and Bergman kernels},
   series={Progress in Mathematics},
   volume={254},
   publisher={Birkh\"auser Verlag},
   place={Basel},
   date={2007},
   pages={xiv+422},
%   isbn={978-3-7643-8096-0},
%   review={\MR{2339952 (2008g:32030)}},
}

\bib{Ma1}{article}{
   author={Machedon, Matei},
   title={Estimates for the parametrix of the Kohn Laplacian on certain
   domains},
   journal={Invent. Math.},
   volume={91},
   date={1988},
   number={2},
   pages={339--364},
%   issn={0020-9910},
%   review={\MR{922804 (89d:58118)}},
%   doi={10.1007/BF01389371},
}

\bib{Ma2}{article}{
   author={Machedon, Matei},
   title={Szeg\H o kernels on pseudoconvex domains with one degenerate
   eigenvalue},
   journal={Ann. of Math. (2)},
   volume={128},
   date={1988},
   number={3},
   pages={619--640},
%   issn={0003-486X},
%   review={\MR{970613 (89i:32043)}},
%   doi={10.2307/1971438},
}


\bib{NaSt}{book}{
   author={Nagel, Alexander},
   author={Stein, E. M.},
   title={Lectures on pseudodifferential operators: regularity theorems and
   applications to nonelliptic problems},
   series={Mathematical Notes},
   volume={24},
   publisher={Princeton University Press},
   place={Princeton, N.J.},
   date={1979},
   pages={159},
%   isbn={0-691-08247-2},
%   review={\MR{549321 (82f:47059)}},
}

\bib{NSW}{article}{
   author={Nagel, Alexander},
   author={Stein, Elias M.},
   author={Wainger, Stephen},
   title={Balls and metrics defined by vector fields. I. Basic properties},
   journal={Acta Math.},
   volume={155},
   date={1985},
   number={1-2},
   pages={103--147},
%   issn={0001-5962},
%   review={\MR{793239 (86k:46049)}},
%   doi={10.1007/BF02392539},
}

\bib{NRSW1}{article}{
   author={Nagel, Alexander},
   author={Rosay, Jean-Pierre},
   author={Stein, Elias M.},
   author={Wainger, Stephen},
   title={Estimates for the Bergman and Szeg\H o kernels in certain weakly
   pseudoconvex domains},
   journal={Bull. Amer. Math. Soc. (N.S.)},
   volume={18},
   date={1988},
   number={1},
   pages={55--59},
%   issn={0273-0979},
%   review={\MR{919661 (89a:32025)}},
%   doi={10.1090/S0273-0979-1988-15598-X},
}

\bib{NRSW2}{article}{
   author={Nagel, A.},
   author={Rosay, J.-P.},
   author={Stein, E. M.},
   author={Wainger, S.},
   title={Estimates for the Bergman and Szeg\H o kernels in ${\bf C}^2$},
   journal={Ann. of Math. (2)},
   volume={129},
   date={1989},
   number={1},
   pages={113--149},
%   issn={0003-486X},
%   review={\MR{979602 (90g:32028)}},
%   doi={10.2307/1971487},
}

\bib{RoSt}{article}{
   author={Rothschild, Linda Preiss},
   author={Stein, E. M.},
   title={Hypoelliptic differential operators and nilpotent groups},
   journal={Acta Math.},
   volume={137},
   date={1976},
   number={3-4},
   pages={247--320},
%   issn={0001-5962},
%   review={\MR{0436223 (55 \#9171)}},
}

\bib{Ste93}{book}{
   author={Stein, Elias M.},
   title={Harmonic analysis: real-variable methods, orthogonality, and
   oscillatory integrals},
   series={Princeton Mathematical Series},
   volume={43},
   note={With the assistance of Timothy S. Murphy;
   Monographs in Harmonic Analysis, III},
   publisher={Princeton University Press},
   place={Princeton, NJ},
   date={1993},
   pages={xiv+695},
%   isbn={0-691-03216-5},
%   review={\MR{1232192 (95c:42002)}},
}

\bib{MR3228630}{article}{
   author={Stein, Elias M.},
   author={Yung, Po-Lam},
   title={Pseudodifferential operators of mixed type adapted to
   distributions of $k$-planes},
   journal={Math. Res. Lett.},
   volume={20},
   date={2013},
   number={6},
   pages={1183--1208},
%   issn={1073-2780},
%   review={\MR{3228630}},
%   doi={10.4310/MRL.2013.v20.n6.a15},
}

\end{biblist}
\end{bibdiv}
\end{document}